\documentclass{article}
\usepackage{amssymb,amsmath,bm,geometry,graphics,graphicx,url,color}
\usepackage{amsfonts}
\usepackage{graphicx}
\usepackage{multirow}

\def\eps{\varepsilon}
\def\sinc{\mathrm{sinc}}

\def\pdt2{\partial_t^2}
\def\pdx2{\partial_x^2}
\newcommand{\norm}[1]{\left\Vert#1\right\Vert}
\newcommand{\normmm}[1]{{\left\vert\kern-0.25ex\left\vert\kern-0.25ex\left\vert #1
    \right\vert\kern-0.25ex\right\vert\kern-0.25ex\right\vert}}

\newcommand{\abs}[1]{\left\vert#1\right\vert}

\def\RR{{\mathbb{R}}}
\def\ZZ{{\mathbb{Z}}}
\def\ii{\mathrm{i}}

\newtheorem{mydef}{Definition}
\newtheorem{mytheo}{Theorem}

\newtheorem{cor}{Corollary}
\newtheorem{lem}{Lemma}
\newtheorem{assum}{Assumption}

\newtheorem{prop}{Proposition}
\newtheorem{algo}{Algorithm}

\newtheorem{rem}{Remark}

\def\ii{\textmd{i}}

\def\RR{\mathbb{R}}
\def\ZZ{\mathbb{Z}}

\def\aa{\mathbf{a}}
\def\bb{\mathbf{b}}
\def\AA{\mathbf{A}}
\def\BB{\mathbf{B}}
\def\cc{\mathbf{c}}

\def\vv{\mathbf{v}}
\def\zz{\mathbf{z}}
\def\FF{\mathbf{F}}
\def\NN{\mathbf{N}}
\def\www{\mathbf{w}}

\def\WW{\boldsymbol{\Omega}}
\def\MM{\mathcal{M}}
\def\sinc{\mathrm{sinc}}

\def\no{\noindent}

\title{Optimal convergence and long-time conservation of exponential integration for Schr\"{o}dinger equations in a normal or highly oscillatory regime}

\author{Bin Wang\,\footnote{School of Mathematics and Statistics,
Xi'an Jiaotong University, Xi'an, Shannxi   710049, P.R.China.
E-mail:~{\tt wangbinmaths@xjtu.edu.cn}} \and Yaolin Jiang
\thanks{School of Mathematics and Statistics,
Xi'an Jiaotong University, Xi'an, Shannxi   710049, P.R.China. E-mail:~{\tt
yljiang@mail.xjtu.edu.cn}}}

\begin{document}
\maketitle

\begin{abstract}
In this paper, we formulate and analyse exponential integrations when applied to nonlinear Schr\"{o}dinger equations in a normal or highly oscillatory regime. A kind of exponential integrators with energy  preservation, optimal convergence and long time
near conservations of actions, momentum and density will be formulated and
analysed. To this end, we derive continuous-stage exponential integrators and show that the integrators can exactly preserve the
energy of Hamiltonian systems. Three practical energy-preserving integrators are presented. It is shown that these
integrators  exhibit optimal convergence and have  near conservations of actions, momentum and
density over long times. A numerical experiment is carried out to support all the theoretical
results presented in this paper. Some applications of the integrators to other kinds of ordinary/partial differential equations are also presented.
\medskip

\no{Keywords:} Schr\"{o}dinger equations; exponential integration; energy-preserving methods; optimal convergence; modulated Fourier
expansion

\medskip
\no{MSC:} 65P10, 65M70.

\end{abstract}


\section{Introduction}\label{sec:introduction}
The main aim of this paper is to present the {formulation and analysis of
exponential integration when
applied to the nonlinear Schr\"{o}dinger equation (NSE) with periodic boundary conditions} (see \cite{Chartier15,Chartier16})
\begin{equation}\label{sch system}
\left\{\begin{aligned} &\mathrm{i}u_{t}(t,x)=-{\frac{1}{\varepsilon}}\triangle u
(t,x)+\lambda|u(t,x)|^{2}u(t,x),\ \ (t,x)\in[0,T]\times
[-\pi,\pi]^d,\\ & u(0,x)=u^0(x),\ \ \ \ \ \ x\in [-\pi,\pi]^d,
\end{aligned}\right.
\end{equation}
  where  $\lambda$ is a parameter and
 {$\varepsilon$ determines the regime of the solution. In this paper, we consider two different regimes: the normal regime $\varepsilon=1$ and the  highly oscillatory regime $0 < \varepsilon\ll1$ which means that the solution is  highly oscillatory.} It is known that the solution of this
equation exactly conserves the following energy
\begin{equation}\label{H}
H[u,\bar{u}]=\frac{1}{2(2 \pi)^d}\int_{[-\pi,\pi]^d}\Big({\frac{1}{\varepsilon}} |\nabla u|^2
+\frac{1}{2}\lambda|u|^4\Big)dx,
\end{equation}
where $|\cdot|$ denotes the Euclidean norm. Apart from this, the
solution also has the conservations of the momentum
\begin{equation}\label{momentum1}
K[u,\bar{u}]=\ii\frac{1}{(2 \pi)^d}\int_{[-\pi,\pi]^d}(u \nabla
\bar{u}-\bar{u}\nabla u)dx,
\end{equation}
  and of the density or mass
\begin{equation}\label{momentum}
m[u,\bar{u}]=\ii\frac{1}{(2 \pi)^d}\int_{[-\pi,\pi]^d}|u|^2dx.
\end{equation}
 For the    linear Schr\"{o}dinger equation, its solution exactly conserves the
actions
\begin{equation}\label{actions}
I_j(u,\bar{u})= \frac{1}{2}\abs{u_j}^2,\qquad j\in \ZZ^d,
\end{equation}
where $u_j$ is defined by  $u(t,x)=\sum\limits_{j\in
\ZZ^d}u_j(t)e^{\ii(j\cdot x)}$ with $j\cdot x=j_1x_1+\cdots+j_dx_d$.
For nonlinear equation \eqref{sch system}, it has been shown that
 these actions are
approximately conserved over long times under conditions of small
initial data and non-resonance  (see \cite{18,19}).
In this paper, only cubic  Schr\"{o}dinger equation with $x \in
[-\pi,\pi]^d$ is considered
for  brevity, although all our ideas, algorithms and analysis can be
easily extended to the solutions of other NSEs.

As is known, NSEs  often arise in a wide range
of applications such as in  fiber optics,  physics, quantum
transport and other applied sciences, and we refer the reader to
\cite{Faou12,jin11,Lubich2020}. In order to
effectively solve NSEs,   various numerical
methods   have been developed and researched in recent decades. With
regard to some related methods of this topic, we refer the reader to
exponential-type integrators (see, e.g.
\cite{Berland07,Besse17,Cano15,Celledoni08,12,Dujardin09,wang-2016}),
splitting methods (see, e.g.
\cite{Iserles14,Besse02,Chartier16,Eilinghoff16,19,Lubich08,Thalhammer12}),
multi-symplectic methods (see, e.g. \cite{Berland07}),
Fourier integrators (see, e.g.
\cite{Faou14,Ostermann19,Ostermann17}), waveform relaxation
algorithms  (see, e.g. \cite{Gander-19})  and other effective
methods (see, e.g.
\cite{Balac16,Bao14,Bejenaru06,Germain09,Islas01,Kishimoto09}).

In the last two decades,  {structure-preserving
algorithms of Hamiltonian partial differential equations (PDEs) have
also been received
 much attention  and  we refer to
\cite{Bhatt17,hairer2006,Islas01,wubook2018}.
Amongst the typical subjects of structure-preserving algorithms  are energy-preserving
(EP) schemes (see, e.g.
 \cite{Dahlby10,Sonnend15,Gong18,jiang20,Shen19,Sonnend14,IMA2018,Wang2019}).} One important property of EP methods is
that they can exactly preserve the energy of the considered system.
 On the other hand, long-time conservation properties of different methods
when applied to  Hamiltonian systems have been researched in many
research publications (see, e.g.
\cite{12,18,19,lubich19,hairer2006}). All
the long-time analyses can be achieved by using the technique of
modulated Fourier expansions, which was developed by Hairer and
Lubich in \cite{Hairer00}.

With regard to the existing researches on these two topics for
Schr\"{o}dinger equations, we have comments as follows:

a) Concerning EP methods for NSEs, although the
average vector field  method (see \cite{Celledoni12}) and
Hamiltonian Boundary Value Methods (see \cite{Brugnano-18}) were
considered, exponential EP methods have not been studied well
for Schr\"{o}dinger equations in the literature. Recently, the
authors in \cite{Wang2019new} derived a kind of exponential
collocation methods, but the energy conservation only holds under
some special conditions. Exponential structure-preserving Runge-Kutta
methods have been studied in \cite{Bhatt17} for first-order ODEs and
the methods are shown to exactly preserve conformal symplecticity
and  decay (or growth) rates in linear and quadratic invariants.
However, energy-preserving exponential Runge-Kutta   methods have
not been considered there. Exponential EP  integrators as well as their convergence have
not been established rigorously for NSEs.

b) For the long time analysis of numerical methods applied to
NSEs, there have also been many publications,
and we refer the reader to
 \cite{12,17-Gauckler,18,19}. Unfortunately, however, all the methods described in these publications are not EP methods.
Too little attention has been paid to the long term analysis of EP
methods in other qualitative aspects for solving NSEs in the literature.

The above facts motivate this paper and the  main contributions will be made as follows:

A) By using the idea of continuous-stage methods, we formulate a
kind of exponential  integration. This formulation will provide novel energy-preserving
methods and this will be discussed in detail in  Sect. \ref{sec:method}.

B) For the obtained EP methods, we analyze their optimal convergence for the first time.   We prove by using the averaging technique \cite{Chartier16}, that some schemes exhibit improved error bounds for highly oscillatory NSEs (Sect. \ref{sec:convergence}).

C) It is also shown that these EP
integrators have near conservations of actions, momentum and
density   over long times by using modulated Fourier expansions (Sect.
\ref{sec:long time con}).  

After these steps,  a novel kind of exponential  integration with energy  preservation, optimal convergence  and long time
near conservations of actions, momentum and
density   is obtained.
 All the theoretical results presented in
this paper will be supported numerically by a numerical experiment
 carried out in Sect. \ref{sec:num exp}. The last section concerns
 some applications of the integrators and some issues which will be
 studied further.


\section{Energy-preserving  exponential integrators}\label{sec:method}
In order to derive energy-preserving exponential integrators, we
consider the simple but   classical way: Duhamel formulation of the equation and the discretization of the integral, which has been used in many publications (see, e.g. \cite{Bao14,Besse17,Bhatt17,Cano15,Celledoni08,12,Dujardin09,Hochbruck2010,Li_Wu(sci2016),Ostermann17}). Although this formulation is not new, the obtained methods will have some advantages    and we will make some important notes in Remark \ref{rem1} below.

Rewrite the NSE \eqref{sch system}  as
\begin{equation}\label{sch ode}
\begin{aligned}
&\frac{\partial u }{\partial t}(t,x)=\ii \mathcal{A}u(t,x)+
f(u(t,x)), \ \  \ \ u(0,x)=u^0(x),
\end{aligned}
\end{equation}
where $\mathcal{A}$ is   the   differential operator defined  by $
(\mathcal{A}u)(t,x)=  {\frac{1}{\varepsilon}}\triangle u(t,x)$ and $f(u)=-
\ii\lambda|u|^{2}u$. The Duhamel principle of this system gives
\begin{equation}\label{Duhamel formu}
\begin{aligned}
u(t_n+h,x)=e^{ \ii h \mathcal{A}}u(t_n,x)+  h \int_{0}^1
e^{(1-\xi)\ii h \mathcal{A}}f(u(t_n+\xi h,x))d\xi
\end{aligned}
\end{equation}
with the time stepsize $h$ and $t_n=nh$. Then we define  the operator-argument
functions $\varphi_j$ 
by
\begin{equation}\label{series}
\varphi_0(\ii t\mathcal{A}):=e^{\ii t \mathcal{A}},\ \
\varphi_j(\ii t\mathcal{A}):=\int_{0}^{1} e^{\ii (1-\xi)t
\mathcal{A}}\frac{\xi^{j-1}}{(j-1)!}{\rm d}\xi ,\quad j=1,2,\dots.
\end{equation}

We deal with the integral appearing in
\eqref{Duhamel formu} by the idea of continuous-stage methods and
define the novel integrators as follows.

\begin{mydef} \label{def:CSEI}{({\textbf{Exponential time integrators.}})} For solving the NSE \eqref{sch system}, a continuous-stage
  exponential time integrator is defined as
follows:
\begin{equation}
\begin{aligned}
u^{n+\tau}(x)&=\Phi^{\tau h}(u^{n}(x)):=C_{\tau}(\mathcal{V})u^n(x)+ h \int_{0}^{1}A_{\tau,\sigma}(\mathcal{V})f(u^{n+\sigma}(x))d\sigma,\ \ 0\leq \tau\leq1,\ n=0,1,\ldots,
\end{aligned}\label{cs ei}%
\end{equation}
where $\mathcal{V}=\ii h\mathcal{A}$,  $C_{\tau}(\mathcal{V})$ and $A_{\tau,\sigma}(\mathcal{V})$ are bounded operator-argument
functions and   $C_{\tau}(\mathcal{V})$ is required to
satisfy
$C_{c_{j}}(\mathcal{V})= e^{c_{j}\mathcal{V}}\  \textmd{for}\
j=0, \ldots , s$
with the fitting nodes $c_{j}$ and $s\geq1$. It is required that  $c_0=0$ and
$c_{s}=1$. 
 The numerical solution after one time stepsize $h$ is obtained by letting $\tau=1$ in \eqref{cs ei}.
\end{mydef}
\begin{rem}\label{rem1}
Although this exponential time integrator is formulated by the Duhamel formulation and the discretization of the integral, which is a very simple and  classical way, it is important to note that this scheme has the following advantages.
\begin{itemize}

  \item At the first sight, for a $p$-th order exponential integrator, it will produce errors of order $\mathcal{O}\big(\frac{h^p}{\eps^p}\big)$ when it is used to solving \eqref{sch system} with a time step size $h$.
However, for the scheme \eqref{cs ei} presented above, we will show  that some obtained methods exhibit improved error bounds such as $\mathcal{O}\big(\frac{h^2}{\eps}\big)$ or $\mathcal{O}\big(\frac{h^3}{\eps^2}\big)$.

  \item  We have noticed that some   novel methods with improved or uniform accuracy  have been presented (see, e.g. \cite{Bao14,Chartier15,Chartier16,Ostermann19,Ostermann17}). These methods have good even better convergence result than the methods given in this paper but they do not have   energy, actions, momentum and
density conservations. Based on the  scheme \eqref{cs ei}, we will obtain some  energy-preserving
exponential integrators with improved error bounds. We will also show that this scheme \eqref{cs ei} can provide  methods   with near conservations of actions, momentum and
density   over long times. In other words, the scheme \eqref{cs ei} can produce some practical methods with three properties simultaneously:
energy preservation, improved error bounds and near conservations of actions, momentum and
density.

\end{itemize}

\end{rem}

For the integrator  \eqref{cs ei}, its energy conservation property is shown as follows.

\begin{mytheo} \label{thm:ep} \textbf{(Energy-preserving  conditions.)}
Let  $\mathcal{K}=hJ\mathcal{M}$ with  $ \mathcal{M}=\left(
   \begin{array}{cc}
     \mathcal{A} & 0 \\
    0 & \mathcal{A} \\
   \end{array}
 \right)$ and $J=\left(
      \begin{array}{cc}
        0 & -1 \\
        1 & 0 \\
      \end{array}
    \right)$. If the coefficients of the   scheme   \eqref{cs ei}  satisfy
\begin{equation}
\left\{\begin{aligned} &A_{0,\sigma}(\mathcal{K})=\mathbf{0},\\
&(e^{\mathcal{K}})^\intercal  \mathcal{M} A_{1,\tau}(\mathcal{K}) \mathcal{K}+(C'_{\tau}(\mathcal{K}))^\intercal   \mathcal{M}=\mathbf{0},\\
&\mathcal{K}^\intercal  (A_{1,\tau}(\mathcal{K}))^\intercal
\mathcal{M} A_{1,\sigma}(\mathcal{K})
\mathcal{K}+\mathcal{M}A'_{\tau,\sigma}(\mathcal{K}) \mathcal{K}
+(\mathcal{M} A'_{\sigma,\tau} (\mathcal{K}) \mathcal{K})^\intercal   =\mathbf{0},\\
\end{aligned}\right.
\label{epcond}%
\end{equation}
with $ C_{\tau}^{'}(\mathcal{K})=\frac{d}{{\rm
d}\tau}C_{\tau}(\mathcal{K})$    and $
A'_{\tau,\sigma}(\mathcal{K})=\frac{\partial}{\partial\tau}A_{\tau,\sigma}(\mathcal{K}),$
 then  the   integrator  \eqref{cs ei} exactly preserves the energy \eqref{H}, i.e.,
$H[u^{n+1},\bar{u}^{n+1}]=H[u^{n},\bar{u}^{n}]$ for $ n=0,1,\ldots.$
\end{mytheo}
 \begin{proof}
By letting $u = p + \textmd{i}q,$  we rewrite the equation
\eqref{sch system}
 as a  infinite-dimensional real Hamiltonian system
\begin{equation}\label{new sch}
\begin{aligned}
\frac{\partial y }{\partial t}=J\mathcal{M}y+J\nabla_y U(y)\ \ \ \
y_0(x)= \left(
                       \begin{array}{c}
                        \textmd{Re}(u_0(x)) \\
                                 \textmd{Im}(u_0(x)) \\
                       \end{array}
                     \right),
\end{aligned}
\end{equation}
where $y=\left(
                             \begin{array}{c}
                               p \\
                               q \\
                             \end{array}
                           \right)$ and $U(y)=-\frac{\lambda}{4}(p^2+q^2)^2 .$
The energy  of this system  accordingly  becomes
 \begin{equation}\label{rea H}
\begin{aligned}
&\mathcal{H}(p,q) =\frac{-1}{2(2 \pi)^d}\int_{\mathbf{T}^d}\Big(
p \mathcal{A}p+q \mathcal{A}q-\frac{\lambda}{2}(p^2+q^2)^2 \Big)dx.\\
\end{aligned}
\end{equation}
Our continuous-stage
  exponential integrator \eqref{cs ei}  applying to \eqref{new sch}
  gives
\begin{equation}
\left\{\begin{aligned}
Y^{n+\tau}(x)&=C_{\tau}(\mathcal{K})y^n(x)+ h \int_{0}^{1}A_{\tau,\sigma}(\mathcal{K})g(Y^{n+\sigma}(x))d\sigma,\ \ 0\leq \tau\leq1,\\
y^{n+1}(x)&=e^{\mathcal{K}}y^n(x)+h \displaystyle\int_{0}^{1}A_{1,\tau}(\mathcal{K})  g(Y^{n+\tau}(x))d\tau,\\
\end{aligned}\right.\label{cs ei-real}%
\end{equation}
 where $g(y)=J\nabla_y U(y).$

 Inserting the numerical scheme \eqref{cs ei-real} into \eqref{rea H} yields
\begin{equation}
\begin{aligned}\label{H2}
 \mathcal{H}[y^{n+1} ]
=&\frac{-1}{2(2
\pi)^d}\int_{\mathbf{T}^d}\Big\{\frac{1}{2}(y^n)^\intercal
\mathcal{M} y^n
+(y^n)^\intercal (e^{\mathcal{K}})^\intercal \mathcal{M}\int_{0}^{1}A_{1,\tau}(\mathcal{K}) \mathcal{K} \tilde{g}(Y^{n+\tau})d\tau\\
&+\frac{1}{2}  \int_{0}^{1}\big(A_{1,\tau}(\mathcal{K}) \mathcal{K} \tilde{g}(Y^{n+\tau})\big)^\intercal d\tau \mathcal{M}\int_{0}^{1}A_{1,\tau}(\mathcal{K}) \mathcal{K} \tilde{g}(Y^{n+\tau})d\tau+U(y^{n+1})\Big \}dx,
\end{aligned}
\end{equation}
where $\tilde{g}=\mathcal{M}^{-1}\nabla_y U(y)$ and we have used the   result  $ (e^{\mathcal{K}})^\intercal \mathcal{M}e^{\mathcal{K}}=\mathcal{M}$ (see  \cite{Li_Wu(sci2016)}).
 It follows from the first condition of  \eqref{epcond}   that $Y^{n}=y^n$ and $ Y^{n+1}=y^{n+1}.$
Then one arrives at
\begin{equation*}
\begin{aligned}
&U(y^{n+1})-U(y^{n})
=\int_{0}^1 \big(\nabla_y U(Y^{n+\tau})\big)^\intercal dY^{n+\tau}\\
=& \int_{0}^1 \big(\nabla_y U(Y^{n+\tau})\big)^\intercal d\Big(C_{\tau}(\mathcal{K})y^n+ h \int_{0}^{1}A_{\tau,\sigma}(\mathcal{K})g(Y^{n+\sigma})d\sigma\Big)\\
=&(y^n)^\intercal \int_{0}^1 (C'_{\tau}(\mathcal{K}))^\intercal   \mathcal{M} \tilde{g}(Y^{n+\tau})d\tau + \int_{0}^{1} \int_{0}^{1}
 \big(\tilde{g}(Y^{n+\tau})\big)^\intercal \mathcal{M} A'_{\tau,\sigma}(\mathcal{K})\mathcal{K}\tilde{g}(Y^{n+\sigma})d\tau d\sigma.
\end{aligned}
\end{equation*}
Therefore,  using the above results and the second   condition of  \eqref{epcond}, we obtain
\begin{equation*}
\begin{aligned}
&\mathcal{H}[y^{n+1} ]-\mathcal{H}[y^{n} ]\\
=&\frac{-1}{2(2 \pi)^d}\int_{\mathbf{T}^d}\frac{1}{2}\int_{0}^{1} \int_{0}^{1} \big(\tilde{g}(Y^{n+\tau})\big)^\intercal \Big\{(\mathcal{K})^\intercal  (A_{1,\tau}(\mathcal{K}))^\intercal \mathcal{M} A_{1,\sigma}(\mathcal{K}) \mathcal{K}+2\mathcal{M}A'_{\tau,\sigma}(\mathcal{K}) \mathcal{K}\Big \} \tilde{g}(Y^{n+\sigma})d\tau d\sigma dx\\
=&\frac{-1}{2(2 \pi)^d}\int_{\mathbf{T}^d}\frac{1}{2}\int_{0}^{1} \int_{0}^{1} \big(\tilde{g}(Y^{n+\sigma})\big)^\intercal \Big\{(\mathcal{K})^\intercal  (A_{1,\sigma}(\mathcal{K}))^\intercal \mathcal{M} A_{1,\tau}(\mathcal{K}) \mathcal{K}+2\mathcal{M}A'_{\sigma,\tau}(\mathcal{K}) \mathcal{K}\Big \}  \tilde{g}(Y^{n+\tau}) d\sigma d\tau dx\\
=& \frac{-1}{2(2 \pi)^d}\int_{\mathbf{T}^d}\frac{1}{2}\int_{0}^{1}
\int_{0}^{1} \big(\tilde{g}(Y^{n+\tau})\big)^\intercal
\Big\{(\mathcal{K})^\intercal  (A_{1,\tau}(\mathcal{K}))^\intercal
\mathcal{M} A_{1,\sigma}(\mathcal{K})
\mathcal{K} +(2\mathcal{M}A'_{\sigma,\tau}(\mathcal{K})
\mathcal{K})^\intercal\Big \} \tilde{g}(Y^{n+\sigma})  d\tau d\sigma
dx.
\end{aligned}
\end{equation*}
It is clear from the third  equality  of \eqref{epcond} that
\begin{equation*}
\begin{aligned}
2(\mathcal{H}[y^{n+1} ]-\mathcal{H}[y^{n} ])
= &\frac{-1}{2(2 \pi)^d}\int_{\mathbf{T}^d} \int_{0}^{1}
\int_{0}^{1} \big(\tilde{g}(Y^{n+\tau})\big)^\intercal
\Big\{(\mathcal{K})^\intercal  (A_{1,\tau}(\mathcal{K}))^\intercal
\mathcal{M} A_{1,\sigma}(\mathcal{K})
\mathcal{K}\\& +\mathcal{M}A'_{\sigma,\tau}(\mathcal{K})
\mathcal{K}+(\mathcal{M}A'_{\sigma,\tau}(\mathcal{K})
\mathcal{K})^\intercal\Big \} \tilde{g}(Y^{n+\sigma})  d\tau d\sigma
dx\\=&0.
\end{aligned}
\end{equation*} The proof is
completed.
\hfill \end{proof}

  In what follows,  we present three practical
energy-preserving algorithms based on the scheme \eqref{cs ei} and on the conditions \eqref{epcond} of energy preservation.
The coefficients are obtained by solving the conditions \eqref{epcond} and we omit the details of calculations for brevity.  

\begin{algo} \label{alg:1} \textbf{(Energy-preserving algorithm 1.)}
For the integrator given in Definition \ref{def:CSEI},
 consider $s=1$ and define a practical method \eqref{cs ei} with the coefficients
$$C_{\tau}(\mathcal{V})=  (1-\tau)I+ \tau e^{\mathcal{V}},\
   A_{\tau,\sigma}(\mathcal{V})=\tau\varphi_1(\mathcal{V}).
$$
 We shall refer to this integrator  by EP1.\end{algo}

\begin{algo} \label{alg:2} \textbf{(Energy-preserving algorithm 2.)}  We   choose   $s=2$ and the coefficients of   \eqref{cs ei} are given by
\begin{equation*}
\begin{array}{ll}
  &C_{\tau}(\mathcal{V})=  \frac{(\tau-1)(\tau-m)}{m} I+\frac{\tau(\tau-1)}{m(m-1)}e^{m\mathcal{V}}+\frac{\tau(m-\tau)}{m-1}
 e^{\mathcal{V}},\ A_{\tau,\sigma}(\mathcal{V})=\sum\limits_{l=1}^2\sum\limits_{n=1}^2a_{ln}(\mathcal{V})\tau^l\sigma^{n-1},
\end{array}
\end{equation*}
where $m$ is a   parameter   required that $m\neq0,1$, and
\begin{equation*}
\begin{array}{ll}
 &a_{11}(\mathcal{V})=
 \frac{1+m}{m(1-m)}\varphi_1(m\mathcal{V})+\frac{m+1}{m-1}\varphi_1(\mathcal{V})+\frac{1}{1-m}\varphi_1((1-m)\mathcal{V}),\\
 &a_{22}(\mathcal{V})=
 \frac{2}{m(1-m)}\big(\varphi_1(m\mathcal{V})-\varphi_1(\mathcal{V})+\varphi_1((1-m)\mathcal{V})\big),\\
  &a_{21}(\mathcal{V})=(1+1/m)\varphi_1(\mathcal{V})-1/m\varphi_1((1-m)\mathcal{V})-a_{11}(\mathcal{V}),\\ &a_{12}(\mathcal{V})=-2/m(\varphi_1(\mathcal{V})-\varphi_1((1-m)\mathcal{V}))-a_{22}(\mathcal{V}).
\end{array}
\end{equation*}
As an example of this  method, we choose  $m=1/2$ and denoted it by EP2.\end{algo}

\begin{algo} \label{alg:3} \textbf{(Energy-preserving algorithm 3.)} As another example, we choose $s=3$ and
\begin{equation*}
\begin{aligned}
  C_{\tau}(\mathcal{V})=& \sum\limits_{k=0}^3l_k(\tau)e^{c_k\mathcal{V}},
 \
  A_{\tau,\sigma}(\mathcal{V})=\sum\limits_{l=1}^3\sum\limits_{n=1}^3a_{ln}(\mathcal{V})\tau^l\sigma^{n-1},
\end{aligned}
\end{equation*}
where $l_j(\tau)=\prod_{k\neq j}\frac{\tau-c_k}{c_j-c_k}$ for
$j=0,\ldots,3$ and
\begin{equation*}
\begin{aligned}
a_{jj}(\mathcal{V})
=&-\big(c_1C_{j0}C_{j1}\varphi_{1,c_1}+c_2C_{j0}C_{j2}\varphi_{1,c_2}+(c_2-c_1)C_{j1}C_{j2}\varphi_{1,c_2-c_1}\\
 &+C_{j0}C_{j3}\varphi_{1,1}+(1-c_1)C_{j1}C_{j3}\varphi_{1,1-c_1}+(1-c_2)C_{j2}C_{j3}\varphi_{1,1-c_2}\big)/j,\  j=1,2,3,\\
 a_{j+1,1}(\mathcal{V})
=&-\big(c_1C_{j1}C_{00}\varphi_{1,c_1}+c_2C_{j2}C_{00}\varphi_{1,c_2}+(c_2-c_1)C_{j2}C_{01}\varphi_{1,c_2-c_1}\\
 &+C_{j3}C_{00}\varphi_{1,1}+(1-c_1)C_{j3}C_{01}\varphi_{1,1-c_1}+(1-c_2)C_{j3}C_{02}\varphi_{1,1-c_2}\big)/j,\ j=1,2,\\
   a_{1,j+1}(\mathcal{V})
=&-\big(c_1C_{j0}C_{01}\varphi_{1,c_1}+c_2C_{j0}C_{02}\varphi_{1,c_2}+(c_2-c_1)C_{j1}C_{02}\varphi_{1,c_2-c_1}\\
 &+C_{j0}C_{03}\varphi_{1,1}+(1-c_1)C_{j1}C_{03}\varphi_{1,1-c_1}+(1-c_2)C_{j2}C_{03}\varphi_{1,1-c_2}\big)/j,\
  j=1,2,\\
     a_{32}(\mathcal{V})
=&-\big(c_1C_{21}C_{10}\varphi_{1,c_1}+c_2C_{22}C_{10}\varphi_{1,c_2}+(c_2-c_1)C_{22}C_{11}\varphi_{1,c_2-c_1}\\
 &+C_{23}C_{10}\varphi_{1,1}+(1-c_1)C_{23}C_{11}\varphi_{1,1-c_1}+(1-c_2)C_{23}C_{12}\varphi_{1,1-c_2}\big),\\
      a_{23}(\mathcal{V})
=&-\big(c_1C_{20}C_{11}\varphi_{1,c_1}+c_2C_{20}C_{12}\varphi_{1,c_2}+(c_2-c_1)C_{21}C_{12}\varphi_{1,c_2-c_1}\\
 &+C_{20}C_{13}\varphi_{1,1}+(1-c_1)C_{21}C_{13}\varphi_{1,1-c_1}+(1-c_2)C_{22}C_{13}\varphi_{1,1-c_2}\big).
\end{aligned}
\end{equation*}
Here we choose 
$c_1= 1/3,\ c_2= \frac{1}{18}(14+(71-9\sqrt{58})^{\frac{1}{3}} +(71+9\sqrt{58})^{\frac{1}{3}} )$
  and use the notations
\begin{equation*}
\begin{array}{ll}
\varphi_{1,1}=\varphi_1(\mathcal{V}),\
&\varphi_{1,c_1}=\varphi_1(c_1\mathcal{V}),\qquad \qquad \
\varphi_{1,c_2}=\varphi_1(c_2\mathcal{V}),\\
\varphi_{1,1-c_1}=\varphi_1((1-c_1)\mathcal{V}),\
&\varphi_{1,1-c_2}=\varphi_1((1-c_2)\mathcal{V}),\ \ \
\varphi_{1,c_2-c_1}=\varphi_1((c_2-c_1)\mathcal{V}),\\
 C_{00} =  \frac{c_1 +
c_2 + c_1c_2}{-c_1c_2},\ &C_{01} = \frac{c_2}{(-1 + c_1) c_1 (c_1 -
c_2)},\quad \ \ C_{02}
=\frac{-c_1}{(c_1 - c_2) (-1 + c_2) c_2)}, \\
C_{10} =  \frac{2 (1 + c_1 + c_2)}{c_1 c_2},\ &C_{11} =  \frac{2(1 +
c_2)}{(-c_1 + c_1^2) (-c_1 + c_2)},\ \ \  \   C_{12} =\frac{2 (1 +
c_1)}{(c_1 - c_2)
(-1 + c_2) c_2},\\
C_{20} =\frac{ -3}{c_1 c_2},\ &C_{21} = \frac{3}{(-1 + c_1) c_1 (c_1
- c_2)},\ \ \ \ \  C_{22} =\frac{-3}{(c_1 - c_2) (-1 + c_2) c_2)}.
\end{array}
\end{equation*}
We shall refer to this semi-discrete integrator  by EP3.

\end{algo}

 The presented three  algorithms EP1-EP3 are obtained by considering the conditions \eqref{epcond} of energy preservation and this shows that all of them are energy-preserving schemes. It is noted that some more energy-preserving schemes can be derived from other value of $s$ and \eqref{epcond} and  we omit them  for brevity.  The main observation of the paper is that some  of these energy-preserving algorithms show optimal error bound and good near conservations of actions, momentum and
density   over long times. All of these observations will be illustrated by numerical experiments in Sect. \ref{sec:num exp}.  The next two sections are devoted to the optimal convergence and long time conservations in actions, momentum and density.

\section{Optimal convergence}\label{sec:convergence}
In this section, we analyze the convergence  of the presented three schemes EP1-EP3.

\subsection{Notations and auxiliary results}
In this part, we present some auxiliary results which will be used in the analysis.

For the exact solution to \eqref{sch system}, we require the following assumption.
\begin{assum}\label{AE} It is assumed that the initial value $u^0(x)$ is chosen
in $H^{\alpha}$ with the sufficiently large exponent $\alpha > 0$. Then the exact solution to \eqref{sch system} is sufficiently regular.
\end{assum}

In the analysis of convergence, we will
 reparametrize the time variable $t$ as
\begin{equation}\label{rescaling}
\kappa:=t/\eps.
\end{equation}
By letting
\begin{equation}\label{imp rea}
w(\kappa,x):=u(t,x),
\end{equation}
it is obtained that $$w_{\kappa}(\kappa,x)=\frac{\partial}{\partial \kappa} u(t,x)=\eps u_t(t,x).
$$
Thus in this section, we consider the following equivalent long-term
NSE (\cite{Chartier16})
\begin{equation}\label{lon sch system}
\left\{\begin{aligned} &\mathrm{i}w_{\kappa}(\kappa,x)=-\triangle w
(\kappa,x)+\varepsilon \lambda|w(\kappa,x)|^{2}w(\kappa,x),\ \ (\kappa,x)\in[0,T/\varepsilon]\times
[-\pi,\pi]^d,\\ & w(0,x)=w^0(x):=u^0(x),\ \ \ \ \ \ x\in [-\pi,\pi]^d,
\end{aligned}\right.
\end{equation}
which helps to zoom-in to see the different scales between  $\eps$ and time step, and to see the averaging effect  which will be used in the proof of the convergence.
The solution of \eqref{lon sch system} satisfies the following properties.
\begin{mytheo}\label{cathm} (See \cite{Castella15}.) For any $\varepsilon > 0$ and $w^0 \in H^{\alpha}$,  there exists a constant
$T > 0$ such that, the long-term NSE \eqref{lon sch system} has a unique solution which satisfies
$$ w\in C^0([0,T/\varepsilon];H^{\alpha})\bigcap  C^1([0,T/\varepsilon];H^{\alpha-2})$$ and
 $$\norm{w(\kappa,\cdot)}_{H^{\alpha}}\leq K \norm{w^0}_{H^{\alpha}}\  \textmd{for \  any}\ \kappa\in[0,T/\varepsilon],$$
where $\alpha > d/2+2$ and $K > 1$.
\end{mytheo}

\begin{prop} \label{capro} (See \cite{Chartier16}.) Let $f(w)=-
\ii\lambda|w|^{2}w$ and the following two estimates hold for this function.
\begin{itemize}

\item  For the function $f(w) \in C^{\infty}: H^{\alpha}\rightarrow H^{\alpha}$,  there exists a constant $M > 0$ such that for all
  $(w, v) \in H^{\alpha}\times H^{\alpha}$, it has the estimates
$$\norm{f(w)}_{H^{\alpha}}\leq M,\ \ \norm{f'(w)(v)}_{H^{\alpha}}\leq M \norm{v}_{H^{\alpha}}.$$
Moreover, similar estimates for higher derivatives also hold.
If $\alpha$ is changed into $\alpha-2>0$, all the results are still true.
\item  The function has the Lipschitz estimate
$$ \norm{f(w)-f(v)}_{H^{\beta}}\leq L \norm{u-v}_{H^{\beta}},\  \ (w,v)\in H^{\alpha-2}\times H^{\alpha-2},$$
where $\beta \in [0, \alpha- 2]$ and $L > 0$ is   a constant.
\end{itemize}
\end{prop}

\begin{prop} \label{pro bo} (See  \cite{Dujardin09}.)
Denote by $\varphi$ a  bounded function
(bounded by $C\geq0$) from $\textmd{i}\mathbb{R}$ to $\mathbb{C}$ and then the operator-argument
function $\varphi(\textmd{i} h \Delta)$ is bounded by
$$\norm{\varphi(\textmd{i} h \Delta)}_{H^{\alpha}\hookrightarrow
H^{\alpha}}\leq C$$
for all $h
>0$ and $\alpha\geq0$.  For
example, the estimate
$\norm{e^{\textmd{i}h \Delta}}_{H^{\alpha}\hookrightarrow
H^{\alpha}}=1$
holds.
\end{prop}


\subsection{Main result}

We first note that for the long term NSE \eqref{lon sch system}, the evolution operator  $e^{\ii t\Delta} $ is periodic with period $T_0$ (\cite{Chartier16}).
For simplicity, it is assumed that $T_0 = 1$ in this section  since this can be achieved by a simple rescaling of time. For simplicity of notations, we shall denote $$A\lesssim B$$ for $A\leq CB$ with a generic constant $C>0$  independent of $n$ or the time step size  or $\varepsilon$ but depends on $T$ and the constants appeared in Theorem \ref{cathm} and Propositions \ref{capro}-\ref{pro bo}.
We use the abbreviation  $w(\kappa)$ instead of $w(\kappa,x)$ for brevity.
For solving the long term NSE \eqref{lon sch system}, the
  exponential time integrator becomes
\begin{equation}
\begin{aligned}
w^{n+\tau}(x)&=\Phi^{\tau \delta\kappa}(w^{n}(x)):=C_{\tau}(\mathcal{W})w^n(x)+ \eps\delta\kappa \int_{0}^{1}A_{\tau,\sigma}(\mathcal{W})f(w^{n+\sigma}(x))d\sigma,\ \ 0\leq \tau\leq1,
\end{aligned}\label{cs ei2}%
\end{equation}
where $\delta\kappa:=\kappa_{n+1}-\kappa_n$ is the time step size and $\mathcal{W}=\ii \delta\kappa \Delta$. Then  EP1-EP3 for solving  \eqref{lon sch system} can also be obtained by considering Algorithms \ref{alg:1}-\ref{alg:3}, respectively.
The optimal convergence of these algorithms is given by the following theorem.

\begin{mytheo} \label{thm ful convergence} (\textbf{Optimal convergence of  algorithms for the long term system.})
 There exists a constant $N_0>0$ independent of $\varepsilon$, such that for any  time step $\delta\kappa=\frac{T_0}{N}$  with any integer $N\geq N_0$, the  EP1-EP3 for solving the long term system \eqref{lon sch system} have the following   error bounds for both regimes $\varepsilon$:
\begin{equation}\label{FEWV}
 \begin{aligned}
 &\textmd{EP1}:\ \   \norm{(\Phi^{ \delta\kappa})^n(w^0)-w(\kappa_n)}_{H^{\alpha-4}}
\lesssim  \delta\kappa^2,\qquad \quad\ \  \alpha > \max(d/2+2,4)\\
&\textmd{EP2}:\ \     \norm{(\Phi^{ \delta\kappa})^n(w^0)-w(\kappa_n)}_{H^{\alpha-6}}
\lesssim  \varepsilon \delta\kappa^2+ \delta\kappa^3,\ \ \alpha > \max(d/2+2,6),\\
&\textmd{EP3}:\ \     \norm{(\Phi^{ \delta\kappa})^n(w^0)-w(\kappa_n)}_{H^{\alpha-8}}
\lesssim  \varepsilon \delta\kappa^3+ \delta\kappa^4,\ \ \alpha > \max(d/2+2,8),
 \end{aligned}
\end{equation}
where   $n\delta\kappa \leq \frac{T}{\varepsilon}$. When $\varepsilon=1$, the above results of EP2 and EP3 can be given in the $H^{\alpha -4}$-norm
and $H^{\alpha -6}$-norm, respectively.
\end{mytheo}

\begin{rem}
Similarly to \cite{Chartier16,zhao2020}, the time step $\delta\kappa=T_0/N$ with some integer $N$  is only a technique condition for rigorous proof and we only need $\delta\kappa\lesssim1$ in practice, which will be shown numerically  in Sect. \ref{sec:num exp}.
In the whole paper, it is noted that estimates are considered  in non-negative Sobolev spaces.
\end{rem}

 Before we present the proof of Theorem \ref{thm ful convergence},   some remarks are given here.
 By the relation (\ref{imp rea}) and by directly comparing (\ref{cs ei}) and (\ref{cs ei2}), it is clear that for $h=\eps\delta\kappa$ and for all $n\geq0$,
$$u(t_n,x)=w(\kappa_n\eps,x),\quad u^n(x)=w^{n}(x).$$
Therefore, the convergence of EP1-EP3 in   the original scaling \eqref{sch system} is equivalently presented as follows.
\begin{cor} \label{thm ful convergence2} (\textbf{Optimal convergence of algorithms for the original system.})
 For the methods  EP1-EP3 with a time step size $h\lesssim \eps$ applied to the original system  \eqref{sch system}, their  error bounds are given by
\begin{equation}\label{FEWV2}
 \begin{aligned}
 &\textmd{EP1}:\ \   \norm{(\Phi^{ h})^n(u^0)-u(t_n)}_{H^{\alpha-4}}
\lesssim  \frac{h^2}{\eps^2},\qquad \quad \alpha > \max(d/2+2,4),\\
&\textmd{EP2}:\ \     \norm{(\Phi^{ h})^n(u^0)-u(t_n)}_{H^{\alpha-6}}
\lesssim  \frac{h^2}{\eps}+ \frac{h^3}{\eps^3},\ \ \alpha > \max(d/2+2,6),\\
&\textmd{EP3}:\ \     \norm{(\Phi^{ h})^n(u^0)-u(t_n)}_{H^{\alpha-8}}
\lesssim  \frac{h^3}{\eps^2}+ \frac{h^4}{\eps^4},\ \ \alpha > \max(d/2+2,8),
 \end{aligned}
\end{equation}
where   $nh \leq T$.  The results of EP2 and EP3 can be respectively  given in the $H^{\alpha -4}$-norm
and $H^{\alpha -6}$-norm when $\varepsilon=1$.
\end{cor}


\subsection{Proof of Theorem \ref{thm ful convergence}}
In the light of  Proposition  \ref{pro bo}, it is obtained that  the coefficients of integrators EP1-EP3 are bounded as
$\norm{C_{\kappa}(\mathcal{W})}_{H^{\alpha}\hookrightarrow H^{\alpha}}\leq
1$ and $
\norm{A_{\tau,\sigma}(\mathcal{W})}_{H^{\alpha}\hookrightarrow
H^{\alpha}}\leq C_A,$  where the constant  $C_{A}$ is
independent of $\norm{\mathcal{W}}_{H^{\alpha}\hookrightarrow
H^{\alpha}}$.
For simplicity, the proof will be given only for  EP2
 because   with little modifications it can be adapted to EP1 and EP3. We begin with the local errors and stability of EP2.

\begin{lem} \label{lem loc} (\textbf{Local errors.})
 For the local errors
\begin{equation*}
\begin{aligned}\delta^{n+\tau}&:=\Phi^{\tau \delta\kappa}(w(\kappa_n))-w(\kappa_{n}+\tau \delta\kappa),\ \ \textmd{for}\ \  0<\tau<1,\\
\delta^{n+1}&:=\Phi^{ \delta\kappa}(w(\kappa_n))-w(\kappa_{n+1}),
\end{aligned}
\end{equation*}
there exits  $\widehat{\delta\kappa}_0>0$  independent of $\varepsilon$ such that for any
 $0 <\delta\kappa<\widehat{\delta\kappa}_0$, the following bounds hold for EP2
\begin{equation*}
\begin{aligned}
 &\|\delta^{n+\tau}\|_{H^{\alpha-2}}\lesssim \delta\kappa,\qquad\ \|\delta^{n+\tau}\|_{H^{\alpha-4}}\lesssim \delta\kappa^2, \quad \textmd{for}\ \  0<\tau<1, \\
 &\|\delta^{n+1}\|_{H^{\alpha-2}}\lesssim  \varepsilon \delta\kappa^2,
 \ \ \ \   \|\delta^{n+1}\|_{H^{\alpha-4}}\lesssim  \varepsilon \delta\kappa^3.
\end{aligned}
\end{equation*}
\end{lem}
 \begin{proof}
 Firstly,  according to the scheme \eqref{cs ei}, the Duhamel principle \eqref{Duhamel formu} and the fact that $$\norm{C_{\tau}(\mathcal{W})w(\kappa_n)-e^{\ii \tau \delta\kappa \triangle}w(\kappa_n)}_{H^{\alpha-2}}\lesssim \delta\kappa,$$ it is clearly that
$
\|\delta^{n+\tau}\|_{H^{\alpha-2}}\lesssim \delta\kappa.$
 Then  it follows from the Duhamel principle \eqref{Duhamel formu} that
\begin{equation*}
\begin{aligned}w(\kappa_n+\tau \delta\kappa)
=&e^{\ii \tau \delta\kappa \triangle}w(\kappa_n)+ \varepsilon\tau \delta\kappa \varphi_{1}(\tau
\mathcal{W})f(w(\kappa_n))\\
&+  \varepsilon\tau^2 \delta\kappa^2    \int_{0}^1 \int_{0}^1\xi
e^{(1-\xi)\ii \tau \delta\kappa \triangle}  f'(w(\kappa_n+\zeta \xi \tau \delta\kappa))w'(\kappa_n+\zeta \xi \tau \delta\kappa) d\zeta d\xi.\end{aligned}
\end{equation*}
For the integrator \eqref{cs ei}, we have
\begin{equation*}
\begin{aligned}\Phi^{\tau \delta\kappa}(w(\kappa_n))
=&C_{\tau}(\mathcal{W})w(\kappa_n)+ \varepsilon  \delta\kappa \int_{0}^{1}A_{\tau,\sigma}(\mathcal{W}) d\sigma f(w(\kappa_n))+ \delta\kappa^2 C_1\\&+  \varepsilon \delta\kappa^2   \int_{0}^1 \int_{0}^1\sigma   A_{\tau,\sigma}(\mathcal{W}) f'(w(\kappa_n+\zeta \sigma \delta\kappa))w'(\kappa_n+\zeta \sigma \delta\kappa) d\zeta d\sigma\end{aligned}
\end{equation*}
with $\norm{C_1}_{H^{\alpha-4}}\lesssim 1$, where we replace $\Phi^{\sigma \delta\kappa}(w(\kappa_n))$ by $w(\kappa_n+\sigma \delta\kappa)$ in the numerical scheme and the error brought by this is denoted by $\delta\kappa^2 C_1$.
The combination of the above two equalities yields $\|\delta^{n+\tau}\|_{H^{\alpha-4}}\lesssim \delta\kappa^2$
for $0<\tau<1$,
where the inequality $$\norm{\int_{0}^{1}A_{\tau,\sigma}(\mathcal{W}) d\sigma-\tau   \varphi_{1}(\tau
\mathcal{W})}_{H^{\alpha-4}}\lesssim \delta\kappa$$ and the result of Lagrange interpolation have been used.

Then by the same arguments given above and by noticing $C_{1}(\mathcal{W})=e^{\ii \delta\kappa \triangle}$, the bound of $\|\delta^{n+1}\|_{H^{\alpha-2}}$ can be derived.

Finally, in the light of
\begin{equation*}
\begin{aligned}w(\kappa_{n+1})
=&e^{ \ii \delta\kappa \triangle}w(\kappa_n)+ \varepsilon \delta\kappa \varphi_{1}(
\mathcal{W})f(w(\kappa_n))+  \varepsilon \delta\kappa^2 \varphi_{2}(
\mathcal{W})f'(w(\kappa_n))w'(\kappa_n) \\
&+\varepsilon \delta\kappa^3  \int_{0}^1 \int_{0}^1(1-\zeta)\xi^2
e^{(1-\xi) \ii \delta\kappa \triangle}\big( f''(w(\kappa_n+\zeta \xi  \delta\kappa))(w'(\kappa_n+\zeta \xi  \delta\kappa))^2\\
&\qquad\qquad\qquad\ \ \ \ + f'(w(\kappa_n+\zeta \xi  \delta\kappa))w''(\kappa_n+\zeta \xi  \delta\kappa)\big)d\zeta d\xi, \end{aligned}
\end{equation*}
and
\begin{equation*}
\begin{aligned}\Phi^{ \delta\kappa}(w(\kappa_n))
=&e^{ \ii \delta\kappa \triangle}w(\kappa_n)+ \varepsilon  \delta\kappa \int_{0}^{1}A_{1,\sigma}(\mathcal{W}) d\sigma f(w(\kappa_n))+  \varepsilon \delta\kappa^2 \int_{0}^{1}\sigma A_{1,\sigma}(\mathcal{W}) d\sigma f'(w(\kappa_n))w'(\kappa_n) \\
&+ \varepsilon \delta\kappa^3 C_2+\varepsilon  \delta\kappa^3  \int_{0}^1 \int_{0}^1(1-\zeta)\sigma ^2 A_{1,\sigma}(\mathcal{W})\big( f''(w(\kappa_n+\zeta \sigma \delta\kappa))(w'(\kappa_n+\zeta \sigma \delta\kappa))^2\\
&\qquad\qquad\qquad\qquad\qquad \ \ \ + f'(w(\kappa_n+\zeta \sigma \delta\kappa))w''(\kappa_n+\zeta \sigma \delta\kappa)\big)d\zeta d\sigma,\end{aligned}
\end{equation*}
with $\norm{C_2}_{H^{\alpha-4}}\lesssim 1$,
 we obtain the bound of $\|\delta^{n+1}\|_{H^{\alpha-4}}$ as follows
\[
\begin{aligned} &\|\delta^{n+1}\|_{H^{\alpha-4}}
\lesssim &
 \sum\limits_{j=0}^{1}\eps\delta\kappa^{j+1}\norm{\varphi_{j+1}(
\mathcal{W})-\int_{0}^{1}A_{1,\sigma}(\mathcal{W})\dfrac{\sigma^j}{j!}{\rm
 d}\sigma}_{H^{\alpha-4}}+\varepsilon \delta\kappa^3.
\end{aligned}
\]
Using the results of $A_{1,\sigma}$:
$$\norm{\int_{0}^{1}A_{1,\sigma}(\mathcal{W}) d\sigma- \varphi_{1}(
\mathcal{W})}_{H^{\alpha-4}}\lesssim 0, \ \ \norm{\int_{0}^{1}A_{1,\sigma}(\mathcal{W}) \sigma d\sigma- \varphi_{2}(
\mathcal{W})}_{H^{\alpha-4}}\lesssim \delta\kappa,$$
 the last
local error  can be bounded.
  \hfill\end{proof}
\begin{lem} (\textbf{Stability.}) Consider the  abbreviations
$R=2K\norm{w^0}_{H^{\alpha}}, \ \mathcal{H}^s_R=\{w\in H^{s},\ \norm{w}_{H^{s}}\leq R\}.$
For the numerical solution $\Phi^{\tau \delta\kappa}$ of  EP2 applied to $v, w \in \mathcal{H}^{\alpha -2}_{3R/4}$, there exist $\varepsilon_0>0$ and $\delta\kappa_0>0$ independent of $\varepsilon$ such that for any $0 <\varepsilon<\varepsilon_0$ and $0 <\delta\kappa<\delta\kappa_0$, it holds that $\Phi^{\tau \delta\kappa}(v),\Phi^{\tau \delta\kappa}(w)\in  \mathcal{H}^{\alpha -2}_{R}$ and
\begin{equation}\label{sta lem}
\begin{aligned}
 &\|\Phi^{\tau \delta\kappa}(v)-\Phi^{\tau \delta\kappa}(w)\|_{H^{\beta}}\leq  e^{\varepsilon \tau \delta\kappa L C_A}\|v-w\|_{H^{\beta}},\ \ 0\leq \tau\leq 1,\\
  &\|(\Phi^{ \delta\kappa}(v)-e^{\ii \delta\kappa \triangle}v)-(\Phi^{ \delta\kappa}(w)-e^{\ii \delta\kappa \triangle}w)\|_{H^{\beta}}\leq \varepsilon \delta\kappa L C_A e^{\varepsilon \tau \delta\kappa L C_A}\|v-w\|_{H^{\beta}},
\end{aligned}
\end{equation}
where  $\beta \in [0, \alpha -2]$.
\end{lem}
 \begin{proof}
 Employing the definition of the method, the isometry $C_{\tau}(\mathcal{W})$ and the Lipschitz estimate of $f$, one gets
 \begin{equation*}
\begin{aligned}
 &\|\Phi^{\tau \delta\kappa}(v)-\Phi^{\tau \delta\kappa}(w)\|_{H^{\beta}}\leq  \|v-w\|_{H^{\beta}}+\varepsilon  L C_A \int_{0}^{\delta\kappa}
 \|\Phi^{\sigma }(v)-\Phi^{\sigma }(w)\|_{H^{\beta}}d\sigma,
\end{aligned}
\end{equation*}
as long as $\Phi^{\sigma }(v),\ \Phi^{\sigma }(w)\in \mathcal{H}^{\alpha -2}_{R}$ for $\sigma\in[0,\delta\kappa]$.
Considering $\tau=1$ and using the Gronwall's lemma yields
$$\|\Phi^{ \delta\kappa}(v)-\Phi^{ \delta\kappa}(w)\|_{H^{\beta}}\leq  e^{\varepsilon \delta\kappa L C_A}\|v-w\|_{H^{\beta}},$$ which gives the first statement of \eqref{sta lem} by modifying $\delta\kappa$ to $\tau \delta\kappa$. Setting in particular $w = 0$ implies  $\Phi^{\tau \delta\kappa}(v)\in \mathcal{H}^{\alpha -2}_{R}$ under the condition that $0 <\delta\kappa<\delta\kappa_0$.
It is also direct to have
 \begin{equation*}
\begin{aligned}
 &\|(\Phi^{ \delta\kappa}(v)-e^{\ii \delta\kappa \triangle}v)-(\Phi^{ \delta\kappa}(w)-e^{\ii \delta\kappa \triangle}w)\|_{H^{\beta}}\leq  \varepsilon \delta\kappa L C_A \|\Phi^{\tau \delta\kappa}(v)-\Phi^{\tau \delta\kappa}(w)\|_{H^{\beta}}.
\end{aligned}
\end{equation*}
The second result of \eqref{sta lem} follows immediately  from this inequality and the first statement.
  \hfill\end{proof}

  We are now in a position to prove Theorem \ref{thm ful convergence}.

\textbf{Proof of Theorem \ref{thm ful convergence}.}
   \begin{proof}
    \textbf{Boundedness of the method.}
    The stated  local errors and stability imply
\begin{equation*}
\begin{aligned} &\norm{(\Phi^{ \delta\kappa})^n(w^0)-w(\kappa_n)}_{H^{\alpha-2}}=\norm{\sum_{l=1}^{n}\big((\Phi^{\delta\kappa})^{n-l}\Phi^{\delta\kappa}(w(\kappa_{l-1}))
-(\Phi^{\delta\kappa})^{n-l}(w(\kappa_{l}))\big)}_{H^{\alpha-2}}\\
\leq&\sum_{l=1}^{n} e^{\varepsilon (n-l)  \delta\kappa L C_A}\norm{ \delta^l }_{H^{\alpha-2}}\leq C \varepsilon \delta\kappa^2 \sum_{l=1}^{n} e^{\varepsilon (n-l)  \delta\kappa L C_A}\leq \tilde{C} \frac{e^{LTC_A}-1}{  L  } \delta\kappa.\\
\end{aligned}
\end{equation*}
Therefore, there exist    $\widetilde{\delta\kappa}_0>0$  independent of $\varepsilon$ such that  $0 <\delta\kappa<\widetilde{\delta\kappa}_0$, the time-discrete
solutions  satisfy  $(\Phi^{ \delta\kappa})^n(w^0)  \in \mathcal{H}^{\alpha -2}_{3R/4}$, where $w(\kappa_n) \in \mathcal{H}^{\alpha -2}_{R/2}$ has been used here.
 Using a stability estimate with
respect to the $H^{\alpha -4}$-norm and considering the local error result in this norm yields
$$\norm{(\Phi^{ \delta\kappa})^n(w^0)-w(\kappa_n)}_{H^{\alpha-4}} \leq  \tilde{C} \frac{e^{LTC_A}-1}{  L  } \delta\kappa^2.$$ 

   \textbf{Refined local error.}
For the method \eqref{cs ei}, we expand the nonlinear function $f$ at $C_{\xi}w(\kappa_n)$ and then get
\begin{equation*}
\begin{aligned}\Phi^{ \delta\kappa}(w(\kappa_n))
=&e^{\ii \delta\kappa \triangle}w(\kappa_n)+ \varepsilon  \delta\kappa \int_{0}^{1}A_{1,\xi} f(C_{\xi}w(\kappa_n)) d\xi \\&+  \varepsilon^2 \delta\kappa^2   \int_{0}^1 \int_{0}^1 A_{1,\xi}  A_{\xi,\sigma}  f'(C_{\xi}w(\kappa_n)) f(\Phi^{\sigma \delta\kappa}(w(\kappa_n))) d\xi d\sigma \\&+
\varepsilon^3 \delta\kappa^3   \int_{0}^1 \int_{0}^1(1-\zeta) A_{1,\xi}  f''\Big(C_{\xi}w(\kappa_n)+\zeta \varepsilon \delta\kappa\int_{0}^1 A_{\xi,\sigma} f(\Phi^{\sigma \delta\kappa}(w(\kappa_n)))d\sigma\Big)\\
&\qquad\qquad\qquad\Big(\int_{0}^1 A_{\xi,\sigma} f(\Phi^{\sigma \delta\kappa}(w(\kappa_n)))d\sigma\Big)^2 d\xi d\zeta\\
=&e^{\ii \delta\kappa \triangle}w(\kappa_n)+ \varepsilon  \delta\kappa \int_{0}^{1}A_{1,\xi} f(C_{\xi}w(\kappa_n)) d\xi \\&+  \varepsilon^2 \delta\kappa^2   \int_{0}^1 \int_{0}^1 A_{1,\xi}  A_{\xi,\sigma}  f'(C_{\xi}w(\kappa_n)) f(C_{\sigma}w(\kappa_n)) d\xi d\sigma +
\varepsilon^3 \delta\kappa^3\Xi_{\Phi},\end{aligned}
\end{equation*}
with \begin{equation*}
\begin{aligned} \Xi_{\Phi}
=&\int_{0}^1 \int_{0}^1 \int_{0}^1 A_{1,\xi}  A_{\xi,\sigma}  f'\big(C_{\xi}w(\kappa_n))  f'(C_{\sigma}w(\kappa_n)+\zeta
(\Phi^{\sigma \delta\kappa}(w(\kappa_n))-C_{\sigma}w(\kappa_n))\big) \\
&\qquad\qquad\qquad\Big(\int_{0}^1 A_{\sigma,\varsigma} f(\Phi^{\varsigma \delta\kappa}(w(\kappa_n)))d\varsigma\Big) d \zeta d\xi d\sigma
\\&+ \int_{0}^1 \int_{0}^1(1-\zeta) A_{1,\xi}  f''\Big(C_{\xi}w(\kappa_n)+\zeta \varepsilon \delta\kappa\int_{0}^1 A_{\xi,\sigma} f(\Phi^{\sigma \delta\kappa}(w(\kappa_n)))d\sigma\Big)\\
&\qquad\qquad\qquad\Big(\int_{0}^1 A_{\xi,\sigma} f(\Phi^{\sigma \delta\kappa}(w(\kappa_n)))d\sigma\Big)^2 d\xi d\zeta.\end{aligned}
\end{equation*}
For the exact solution \eqref{Duhamel formu}, similarly we  obtain its expansion as
\begin{equation*}
\begin{aligned}w(\kappa_{n+1})=&e^{\ii \delta\kappa \triangle}w(\kappa_n)+ \varepsilon \delta\kappa \int_{0}^1
e^{(1-\xi)\ii \delta\kappa \triangle}f(e^{\ii \xi  \delta\kappa \triangle} w(\kappa_n))d\xi\\
 &+\varepsilon^2 \delta\kappa^2 \int_{0}^1 \int_{0}^1 \xi
e^{(1-\xi)\ii \delta\kappa \triangle}f'(e^{\ii \xi  \delta\kappa \triangle} w(\kappa_n)) e^{(1-\sigma)\ii\xi \delta\kappa \triangle}f(e^{ \ii\sigma  \delta\kappa \triangle} w(\kappa_n)) d\xi d \sigma+ \varepsilon^3 \delta\kappa^3 \Xi_{w},\end{aligned}
\end{equation*}
with \begin{equation*}
\begin{aligned} \Xi_{w}
=&\int_{0}^1 \int_{0}^1\int_{0}^1 e^{(1-\xi)  \ii \delta\kappa \triangle} e^{(1-\sigma)\xi\ii \delta\kappa \triangle}   f'(e^{\xi\ii \delta\kappa \triangle} w(\kappa_n))
f'\big( e^{ \ii\sigma  \delta\kappa \triangle} w(\kappa_n) \\
&+\zeta (w(\kappa_n+\sigma \delta\kappa-e^{ \ii\sigma  \delta\kappa \triangle} w(\kappa_n))\big)\Big(\int_{0}^1 e^{(1-\varsigma)\ii \sigma \delta\kappa \triangle} f(w(\kappa_n+\varsigma \delta\kappa))d\varsigma\Big)  d \zeta d\xi d\sigma
\\&+ \int_{0}^1 \int_{0}^1(1-\zeta) e^{(1-\xi)\ii  \delta\kappa \triangle}  f''\Big(e^{\xi \ii \delta\kappa \triangle}w(\kappa_n)+\zeta \varepsilon \delta\kappa\int_{0}^1 e^{(1-\sigma)\ii \xi \delta\kappa \triangle} f(w(\kappa_n+\sigma \delta\kappa))d\sigma\Big)\\
&\qquad\qquad\quad\Big(\int_{0}^1 e^{(1-\sigma)\ii \xi \delta\kappa \triangle} f(w(\kappa_n+\sigma \delta\kappa))d\sigma\Big)^2 d\xi d\zeta.\end{aligned}
\end{equation*}
  Then the local error $\delta^{n+1}$ can be refined as \begin{equation}\label{ref locerr}\delta^{n+1}=\varepsilon \delta\kappa  \Psi(\kappa_n)+\varepsilon^2 \delta\kappa^2 \Delta(\kappa_n),\end{equation}
where
\begin{equation*}
\begin{aligned}\Psi(\kappa_n)=&\int_{0}^{1}A_{1,\xi} f(C_{\xi}w(\kappa_n)) d\xi-\int_{0}^1
e^{(1-\xi) \ii \delta\kappa \triangle}f(e^{ \xi  \ii \delta\kappa \triangle} w(\kappa_n))d\xi,\\
\Delta(\kappa_n)=& \int_{0}^1 \int_{0}^1 A_{1,\xi}  A_{\xi,\sigma}  f'(C_{\xi}w(\kappa_n)) f(C_{\sigma}w(\kappa_n)) d\xi d\sigma \\&
-\int_{0}^1 \int_{0}^1 \xi
e^{(1-\xi) \ii \delta\kappa \triangle}f'(e^{ \xi \ii \delta\kappa \triangle} w(\kappa_n)) e^{(1-\sigma)\xi \ii \delta\kappa \triangle}f(e^{\ii\sigma  \delta\kappa \triangle} w(\kappa_n)) d\xi d \sigma+
\varepsilon \delta\kappa   \Xi_{\Phi}-
\varepsilon \delta\kappa   \Xi_{w}.
 \end{aligned}
\end{equation*}
Concerning the previous local errors given in Lemma \ref{lem loc}, one has
$$
\|\Psi(\kappa_n)\|_{H^{\alpha-4}}\lesssim   \delta\kappa^2,\quad \ \|\Delta(\kappa_n)\|_{H^{\alpha-4}}\lesssim   \delta\kappa.$$

    \textbf{Refined convergence over one period.} In this part, we consider convergence over one period, that is $n\delta\kappa=T_0=1$. For the global error
    \[
\begin{aligned} (\Phi^{ \delta\kappa})^n(w^0)-w(\kappa_n)&=\sum_{l=1}^{n}\big((\Phi^{\delta\kappa})^{n-l}\Phi^{\delta\kappa}(w(\kappa_{l-1}))-(\Phi^{\delta\kappa})^{n-l}(w(\kappa_{l}))\big),
\end{aligned}
\]
 we introduce $\Theta^{\delta\kappa}_{n-l}:=(\Phi^{\delta\kappa})^{n-l}-e^{\ii (n-l) \delta\kappa \triangle}$ and then rewrite it as
   \begin{equation}\label{global err}
\begin{aligned} (\Phi^{ \delta\kappa})^n(w^0)-w(\kappa_n)&=\underbrace{\sum_{l=1}^{n}e^{\ii (n-l) \delta\kappa \triangle}
\delta^{l}}_{\mathbb{E}_1}+\underbrace{\sum_{l=1}^{n}
\big(\Theta^{\delta\kappa}_{n-l}(\Phi^{\delta\kappa}(w(\kappa_{l-1})))-\Theta^{\delta\kappa}_{n-l}(w(\kappa_{l}))\big)}_{\mathbb{E}_2}.
\end{aligned}
\end{equation}
For the part $\mathbb{E}_2$, we first estimate
   \begin{equation*}
\begin{aligned} &\norm{\Theta^{\delta\kappa}_{l}v-\Theta^{\delta\kappa}_{l}w}_{H^{\beta}}= \norm{(\Phi^{\delta\kappa})^{l}v-e^{\ii l \delta\kappa \triangle}v-(\Phi^{\delta\kappa})^{l}w+e^{\ii l \delta\kappa \triangle}w}_{H^{\beta}}\\
\leq&\sum_{k=1}^l\norm{\Theta^{\delta\kappa}_{1} (\Phi^{\delta\kappa})^{k-1}v-\Theta^{\delta\kappa}_{1} (\Phi^{\delta\kappa})^{k-1}w}_{H^{\beta}}
\leq  \varepsilon \delta\kappa L C_A e^{\varepsilon \tau \delta\kappa L C_A}\sum_{k=1}^l \norm{(\Phi^{\delta\kappa})^{k-1}v-(\Phi^{\delta\kappa})^{k-1}w}_{H^{\beta}}\\
\leq& \varepsilon \delta\kappa L C_A e^{\varepsilon \tau \delta\kappa L C_A}\sum_{k=1}^l e^{\varepsilon (k-1) \delta\kappa L C_A} \norm{v-w}_{H^{\beta}}
\leq  \varepsilon   L C_A T_0 e^{\varepsilon   L C_A T_0} \norm{v-w}_{H^{\beta}}.
\end{aligned}
\end{equation*}
Then the following bound holds
   \begin{equation}\label{E2}
\begin{aligned} \norm{\mathbb{E}_2}_{H^{\alpha-4}}\leq  \varepsilon L C_A T_0 e^{\varepsilon   L C_A T_0}\sum_{l=1}^{n} \norm{\delta^{l}}_{H^{\alpha-4}}
\lesssim  \varepsilon^2 \delta\kappa^2.
\end{aligned}
\end{equation}
For the part $\mathbb{E}_1$, we use the refined local error \eqref{ref locerr} and then have
   \begin{equation}\label{E1}
\begin{aligned} \mathbb{E}_1=\sum_{l=1}^{n}e^{\ii (n-l) \delta\kappa \triangle}  \varepsilon \delta\kappa  \Psi(\kappa_{l-1})+\sum_{l=1}^{n}e^{\ii (n-l) \delta\kappa \triangle}\varepsilon^2 \delta\kappa^2 \Delta(\kappa_{l-1}).
\end{aligned}
\end{equation}
According to \eqref{global err}-\eqref{E1} and the following bound
 \begin{equation*}
\begin{aligned} \norm{\sum_{l=1}^{n}e^{\ii (n-l) \delta\kappa \triangle}\varepsilon^2 \delta\kappa^2 \Delta(\kappa_{l-1})}_{H^{\alpha-4}}
\leq \varepsilon^2 \delta\kappa^3 \sum_{l=1}^{n} \norm{e^{\ii (n-l) \delta\kappa \triangle} }_{H^{\alpha-4}}
\lesssim   \varepsilon^2 \delta\kappa^2,
\end{aligned}
\end{equation*}
the global error is bounded by
 \begin{equation}\label{for gerr}
\begin{aligned} \norm{(\Phi^{ \delta\kappa})^n(w^0)-w(\kappa_n)}_{H^{\alpha-4}}
\lesssim   \varepsilon \delta\kappa  \norm{\sum_{l=1}^{n}e^{\ii (n-l) \delta\kappa \triangle}   \Psi(\kappa_{l-1}) }_{H^{\alpha-4}}+ \varepsilon^2 \delta\kappa^2.
\end{aligned}
\end{equation}

In what follows, we derive the optimal bound for $\varepsilon \delta\kappa\norm{\sum_{l=1}^{n}e^{\ii (n-l) \delta\kappa \triangle}   \Psi(\kappa_{l-1}) }_{H^{\alpha-4}}$, which satisfies   \begin{equation*}
\begin{split}& \varepsilon \delta\kappa\norm{\sum_{l=1}^{n}e^{\ii (n-l) \delta\kappa \triangle}   \Psi(\kappa_{l-1}) }_{H^{\alpha-4}}\\
\lesssim & \varepsilon \delta\kappa\norm{  \sum_{l=1}^{n}e^{\ii (n-l) \delta\kappa \triangle}   \int_{0}^{1}A_{1,\xi} f(C_{\xi}e^{\ii (l-1) \delta\kappa \triangle} w_0) d\xi-\varepsilon \sum_{l=1}^{n}e^{\ii (n-l+1) \delta\kappa \triangle} \int_{0}^{\delta\kappa}
e^{- \ii \xi \triangle}f(e^{    \ii \xi \triangle} e^{\ii (l-1) \delta\kappa \triangle} w_0 )d\xi}  _{H^{\alpha-4}} \\&+ \varepsilon^2 \delta\kappa^2\\
\lesssim &\norm{\varepsilon \delta\kappa  \sum_{l=1}^{n}e^{-\ii  l  \delta\kappa \triangle}   \int_{0}^{1}A_{1,\xi} f(C_{\xi}e^{\ii (l-1) \delta\kappa \triangle} w_0) d\xi- \varepsilon \int_{0}^1
e^{- \ii \xi \triangle}f(e^{    \ii \xi \triangle}  w_0 )d\xi}_{H^{\alpha-4}}+ \varepsilon^2 \delta\kappa^2.
\end{split}
\end{equation*}
Here we used the result $\norm{w(\kappa_{l-1})-e^{\ii (l-1) \delta\kappa \triangle} w_0}_{H^{\alpha-4}}\lesssim \varepsilon.$
We first consider Fourier expansion $F_{\xi }(w)=\sum_{k\in \mathbb{Z}}e^{\ii 2 k \pi\xi  }\hat{F}_{k}(w)$ of $F_{\kappa }(w):=e^{-\ii \xi \triangle}f(e^{\ii \xi   \triangle} w)$,
which yields that  $\int_{0}^1
e^{- \ii \xi \triangle}f(e^{    \ii \xi \triangle}  w_0 )d\xi=\hat{F}_{0}(w_0).$ Then let $G_{l\delta\kappa}(w)=
e^{-\ii  l  \delta\kappa \triangle}   \int_{0}^{1}A_{1,\xi} f(C_{\xi}e^{\ii l \delta\kappa \triangle} w) d\xi$ and the Fourier expansion of $G_{l\delta\kappa}(w)$
is given by  $G_{l\delta\kappa}(w)=\sum_{k\in \mathbb{Z}}e^{\ii 2 k \pi l\delta\kappa }\hat{G}_{k}(w).$
Therefore, it is obtained that
 \begin{equation*}
\begin{aligned} &\varepsilon \delta\kappa  \sum_{l=1}^{n}e^{-\ii  l  \delta\kappa \triangle}   \int_{0}^{1}A_{1,\xi} f(C_{\xi}e^{\ii (l-1) \delta\kappa \triangle} w_0) d\xi
=\varepsilon \delta\kappa e^{-\ii    \delta\kappa \triangle}  \sum_{l=0}^{n-1}\sum_{k\in \mathbb{Z}}e^{\ii 2 k \pi l\delta\kappa }\hat{G}_{k}(w)\\
=&\varepsilon   e^{-\ii    \delta\kappa \triangle} \sum_{k\in \mathbb{Z}}\Big(\frac{1}{n} \sum_{l=0}^{n-1}e^{\ii 2 k \pi l\delta\kappa }\hat{G}_{k}(w)\Big)
=\varepsilon   e^{-\ii    \delta\kappa \triangle} \sum_{k\in \mathbb{Z}} \hat{G}_{nk}(w).
\end{aligned}
\end{equation*}
Based on the above results, it follows that
 \begin{equation}\label{For s}
\begin{aligned}& \varepsilon \delta\kappa\norm{\sum_{l=1}^{n}e^{\ii (n-l) \delta\kappa \triangle}   \Psi(\kappa_{l-1}) }_{H^{\alpha-6}}\\
\lesssim  &\varepsilon \norm{\hat{F}_{0}(w_0)-  e^{-\ii    \delta\kappa \triangle}   \hat{G}_{0}(w) }_{H^{\alpha-6}}+\varepsilon \norm{   \sum_{k\in \mathbb{Z}^{*}} \hat{G}_{nk}(w)}_{H^{\alpha-6}}+ \varepsilon^2 \delta\kappa^2\\
\lesssim &\varepsilon \norm{\int_{0}^{1}e^{ - \ii \xi   \triangle}f(e^{   \ii \xi   \triangle} w_0)d \xi-  \int_{0}^{1}e^{ - \ii \xi   \triangle}\Big[
e^{-\ii    \delta\kappa \triangle} \int_{0}^{1}A_{1,\xi} f(C_{\xi}e^{\ii \xi \triangle} w_0) d\xi\Big] d \xi}_{H^{\alpha-6}}\\
&+\varepsilon \delta\kappa^3+ \varepsilon^2 \delta\kappa^2\\
\lesssim & \varepsilon \delta\kappa^3+\varepsilon \delta\kappa^3+ \varepsilon^2 \delta\kappa^2.
\end{aligned}
\end{equation}
Here Lemma A.1 of \cite{Chartier16} and the results $A_{1,\xi}$ and $C_{\xi}$ of EP2 are used to obtain the last two inequalities, respectively.
Finally, combining \eqref{for gerr} with \eqref{For s}, we obtain the global error over one period
 \begin{equation}\label{LEO}
\begin{aligned} \norm{(\Phi^{ \delta\kappa})^n(w^0)-w(\kappa_n)}_{H^{\alpha-6}}
\lesssim  \varepsilon \delta\kappa^3+ \varepsilon^2 \delta\kappa^2,\ \ n\delta\kappa=T_0.
\end{aligned}
\end{equation}

        \textbf{Refined global error.}

       For $n\delta\kappa\leq T/\varepsilon$, the global error of EP2 given in \eqref{FEWV}
can be derived  by considering \eqref{LEO} and by using the same way presented in Sect.
5 of \cite{Chartier16}.

The whole proof is complete.
   \hfill\end{proof}
\begin{rem}
It is noted that for EP1, the estimate of \eqref{For s} is
only  $\varepsilon \delta\kappa^2$. Therefore, EP1 does not have  optimal convergence.
\end{rem}

 \section{Long time conservations in actions, momentum and density}\label{sec:long time con}

In this section,  we turn back to  the methods   applied to the original system  \eqref{sch system} and
in order to make the analysis   be succinct, we
choose $\lambda=1$.
  For our
integrator \eqref{cs ei}, spectral semi-discretisation  (see
\cite{Chen2001,12,18,19}) with   the
  points
$x_k=\frac{\pi}{M}k,\ k\in \mathcal{M}$
 is used in space, where $\mathcal{M}=\{-M,\ldots,M-1\}^d$ and $2M$ presents the number of internal discretisation points in space. Then
the fully discrete scheme of  \eqref{cs ei} is
\begin{equation}
\begin{aligned}
u^{n+\tau}&=C_{\tau}(V)u^n+ h \int_{0}^{1}A_{\tau,\sigma}(V)f(u^{n+\sigma})d\sigma,\ \ 0\leq \tau\leq1,
\end{aligned}\label{cs ei-spe}%
\end{equation}
where $V=\ii h \Omega$, $\Omega=-\textmd{diag}((\omega_j)_{j\in
\mathcal{M}})$ and $f(u)=-\ii \mathcal{Q}(\abs{u}^2u)$
\footnote{We still use the notation $f$ in this section without any  confusion.}
. Here, $ \omega_j= \frac{1}{\varepsilon}\abs{j}^2=\frac{1}{\varepsilon} (j_1^2+\cdots+j_d^2)$ {for
$j=(j_1,\ldots,j_d)\in \mathcal{M}$} are the   eigenvalues  of the
linear part of \eqref{sch system} after  spectral
semi-discretisation in space, and the notation $\mathcal{Q}(v)$
denotes the trigonometric interpolation of a periodic function
$v=\sum\limits_{j\in \ZZ^d}v_je^{\ii(j\cdot x)}$ in the collocation
points, i.e., $\mathcal{Q}(v)=\sum\limits_{j\in
\mathcal{M}}\big(\sum\limits_{l\in
\ZZ^d}v_{j+2Ml}\big)\mathrm{e}^{\mathrm{i}(j\cdot x)}.$

The following notations  are needed in this section which have been
used in \cite{12,18,19}.
 For  a sequence
$k = (k_j)_{j\in \mathcal{M}}$ of integers $k_j$ and the
sequence
  $\omega = (\omega_j)_{j\in \mathcal{M}},$
denote $$ \norm{k}=\sum\limits_{j\in \mathcal{M}}|k_j|,  \
k\cdot \omega=\sum\limits_{j\in
\mathcal{M}}k_j\omega_j,  \ \omega^{\sigma
|k|}=\Pi_{j\in \mathcal{M}} \omega_j^{\sigma |k_j|} $$ for a
real $\sigma$.
 Denote  by  $\langle j\rangle$  the unit coordinate vector $(0, \ldots , 0, 1, 0,
\ldots,0)^{\intercal}$  with the only entry $1$ at the $|j|$-th
position.

 \subsection{Result of  near-conservation properties}

\begin{mytheo}\label{main theo} {(\textbf{Long time
near-conservations.})} Consider the small initial data \begin{equation}\label{initi cond}
\norm{u^0}_{H^s}\leq \tilde{\epsilon} \ll 1, \end{equation}
and define the set
\begin{equation}
\mathcal{R}_{\tilde{\epsilon},M,h}=\Big\{(j, k):j=j(k),\
k\neq\langle j\rangle,\ \abs{\sin\big(\frac{1}{2}h(\omega_j-k
\cdot \omega)\big)}\leq \frac{1}{2}\tilde{\epsilon}^{1/2}h,\
\norm{k}\leq2N+2\Big\},
\label{near-resonant R}%
\end{equation}
where $j(k):=\sum\limits_{l\in \mathcal{M}}k_ll
\ \textmd{mod}\  2M\in \mathcal{M}. 
$
 For the near-resonant indices $(j,k)$ in
$\mathcal{R}_{\tilde{\epsilon},M,h}$, they are required such that
\begin{equation}
\sup_{(j, k)\in\mathcal{R}_{\tilde{\epsilon},M,h}}
\frac{\abs{\omega_j}^{s-\frac{d+1}{2}}}{\omega^{(s-\frac{d+1}{2})
|k|}}\tilde{\epsilon}^{\norm{k}+1}\leq \tilde{C}
\tilde{\epsilon}^{2N+4}
 \label{non-resonance cond}%
\end{equation}
with a constant $\tilde{C}$ independent of $\tilde{\epsilon}$.
 For given $N\geq1$ and $s \geq
d + 1$, the numerical solution $u^n$
of   EP1  has the following conservations of actions, momentum and density, respectively
\begin{equation*}
\begin{aligned}
&\sum\limits_{j\in
\MM}|\omega_j|^{s}\frac{|I_j(u^{n},\bar{u}^{n})-I_j(u^{0},\bar{u}^{0})|}{\tilde{\epsilon}^2}
\leq C  \tilde{\epsilon}^{\frac{3}{2}},\\
 &\sum\limits_{r=1}^d\frac{|K_r[u^{n},\bar{u}^{n}]-K_r[u^{0},\bar{u}^{0}]|}{\tilde{\epsilon}^2}
\leq C  \tilde{\epsilon}^{\frac{3}{2}},\\ & \frac{|m[u^{n},\bar{u}^{n}]-m[u^{0},\bar{u}^{0}]|}{\tilde{\epsilon}^2} \leq
C  \tilde{\epsilon}^{\frac{3}{2}},
\end{aligned}
\end{equation*}
where  $0\leq t_n=nh\leq \tilde{\epsilon}^{-N}$ and the constant $C$ depends
on $\tilde{C}$, $\max_{j\in \mathcal{M}}\big \{
\frac{1}{\abs{\cos(\frac{1}{2}h \omega_j)}}\big\}$, the dimension
$d, N, s$ and the norm of the potential but is independent of
$n$, the size of the initial value $\tilde{\epsilon}$, the regime of the solution $\varepsilon$, and the discretisation
parameters $M$ and $h$. Here $K_r$ is referred to the $r$th
component of $K$.
 For the   schemes  EP1-EP2,
if the midpoint rule is used to the integral appearing in these methods, the above near conservations still hold.
\end{mytheo}
\begin{rem}
We remark that the method EP3 does not have such near conservations and the reason will be explained at the end of this section.
\end{rem}

\begin{rem}
It is noted that the authors in \cite{12,17-Gauckler,19}
analysed the long-time behaviour of exponential integrators,
splitting integrators and split-step Fourier method for
Schr\"{o}dinger equations. However, those methods cannot preserve
the  energy  \eqref{rea H} exactly. We remark that  Theorem \ref{main theo}
shows that our energy-preserving integrators also have a  near conservation of  actions,
momentum and density over long times.
\end{rem}

%

\subsection{The proof of Theorem \ref{main theo}} \label{sec:proof}
 The proof makes use of a
modulated Fourier expansion \cite{12,18,19,Wang2019} in time of the numerical solution.
We will use the following expansion
\begin{equation}
\begin{aligned} &\tilde{u}(t,x)=
\sum\limits_{\norm{k}\leq K}z^{k}(\tilde{\epsilon} t,x)
\mathrm{e}^{-\mathrm{i}(k \cdot \omega) t}=
\sum\limits_{\norm{k}\leq K}\sum\limits_{j\in
\MM}z_j^{k}(\tilde{\epsilon} t) \mathrm{e}^{\mathrm{i}(j \cdot x)}
\mathrm{e}^{-\mathrm{i}(k \cdot \omega) t}
\end{aligned}
\label{MFE-EAVF}%
\end{equation}
 to describe the numerical solution $u^n$  at time
$t_n = nh$  after $n$ time steps, where the  functions $z^{k}$ are
termed the modulation functions which evolve on a slow time-scale
$\tilde{\tau} = \tilde{\epsilon} t.$ Following \cite{12},  these functions can be
assumed to be single spatial waves:
$
 z^{k}(\tilde{\epsilon} t,x)=z_{j(k)}^{k}(\tilde{\epsilon} t) \mathrm{e}^{\mathrm{i}(j(k) \cdot
 x)},
$
i.e., their Fourier coefficients $z_j^{k}$ vanish for $j\neq
j(k)$ with $j(k)=\sum\limits_{l\in \mathcal{M}}k_ll
\ \textmd{mod}\  2M\in \mathcal{M}$.

 It is noted that as a standard approach to the study of the
long-time behavior of numerical methods,  modulated Fourier
expansion is also used in the analysis of
\cite{12,18,19,Wang2019}. However, in this paper, there are novel
modifications adapted to our integrators, which come from the
implicitness of the integrator and the integral appearing in the
integrator. {We present the main differences in the
proof.} For the similar derivations as those of \cite{12,18,19}, we
skip them in the analysis for brevity.

\subsubsection{Modulation equations}\label{subsec:mod equ}
{
\begin{prop} (\textbf{Modulation equations.})
Define
\begin{equation*}
\begin{aligned}L^{k}:&=(L^{k}_2)^{-1} L^{k}_1,\\
L_1^{k}:&=\mathrm{e}^{-\mathrm{i}(k \cdot \omega)
h}\mathrm{e}^{\tilde{\epsilon}
hD}-2\cos(h\Omega)+\mathrm{e}^{\mathrm{i}(k
\cdot \omega)
h}\mathrm{e}^{-\tilde{\epsilon} hD},\\
L_2^{k}:&=\varphi_1(\ii h
\Omega)\mathrm{e}^{-\frac{1}{2}\mathrm{i}(k
\cdot \omega) h}\mathrm{e}^{\frac{1}{2}\tilde{\epsilon}
hD}-\varphi_1(-\ii h
\Omega)\mathrm{e}^{\frac{1}{2}\mathrm{i}(k
\cdot \omega)
h}\mathrm{e}^{-\frac{1}{2}\tilde{\epsilon} hD},\\
 \end{aligned}
\end{equation*}
 where $D$ is the
differential operator (see \cite{hairer2006}).
 The modulation equations for the coefficients $
z_j^{k}$  appearing in \eqref{MFE-EAVF} are given by
\begin{equation}\label{ljkqp}
\begin{aligned}L^{k} z_j^{k}(\tilde{\epsilon} t)= -\ii h
\sum\limits_{k^1+k^2-k^3=k
}\int_{0}^{1}w^{k^1}_{j(k^1)}(\tilde{\epsilon}
t,\sigma)w^{k^2}_{j(k^2)}(\tilde{\epsilon}
t,\sigma)\overline{w^{k^3}_{j(k^3)}}(\tilde{\epsilon} t,\sigma)d\sigma,
\end{aligned}
\end{equation}
where
 \begin{equation}\label{eta modula sys}
w^{k}_{j(k)}(\tilde{\epsilon}
t,\sigma)=L^{k}_3(\sigma)z_{j(k)}^{k}(\tilde{\epsilon}
t)
\end{equation} with $$
L_3^{k}(\sigma):=
(1-\sigma)\mathrm{e}^{\frac{1}{2}\mathrm{i}(k
\cdot \omega) h}\mathrm{e}^{-
 \frac{h}{2}\tilde{\epsilon} D} +\sigma\mathrm{e}^{-\frac{1}{2}\mathrm{i}(k \cdot \omega)
h}\mathrm{e}^{
 \frac{h}{2}\tilde{\epsilon} D}.$$
The initial condition for modulation equations
is given by \begin{equation}\label{initial pl}u_j^0= \sum\limits_{k} z_{j(k)}^{k}(0).\end{equation}
  \end{prop}}
\begin{proof}
In order to derive the modulation equations for EP1, a new approach
different from \cite{12,18,19} is considered here. To this end, we
define the   operator $L^{k}$ and it can
be expressed in   Taylor expansions as follows:
\begin{equation}\label{L expansion}
\begin{aligned}L^{\langle j\rangle}_j=&  \frac{1}{2}
\tilde{\epsilon} h^2\omega_j\csc\big(\frac{1}{2}h\omega_j\big)  D+
\frac{1}{48}
\tilde{\epsilon}^3 h^4\omega_j\csc\big(\frac{1}{2}h\omega_j\big)  D^3+\cdots,\\
 L^{k} = & \ii h\Omega\csc\big(\frac{1}{2}h\Omega\big)\sin\big(\frac{1}{2}h(-\Omega-(k \cdot\omega)I)\big) \\
 &+ \frac{1}{2}\tilde{\epsilon} h^2  \Omega\csc\big(\frac{1}{2}h\Omega\big)\cos\big(\frac{1}{2}h((k
 \cdot\omega)I+\Omega)\big)D
 +\cdots.
\end{aligned}
\end{equation}
 Moreover, for the operator $L_3^{k}(\sigma)$, we have $$
L_3^{k}(\frac{1}{2})=\cos\big(\frac{h(k \cdot
 \omega)}{2}\big)+\frac{1}{2}\sin\big(\frac{h(k
\cdot  \omega)}{2}\big)(\textmd{i} h\tilde{\epsilon} D)+\cdots.$$

By using the symmetry of the EP1
 integrator and $$\displaystyle\int_{0}^{1}f((1-\sigma)u^n+\sigma
u^{n-1})d\sigma=\displaystyle\int_{0}^{1}f((1-\sigma)u^{n-1}+\sigma
u^{n})d\sigma,$$  we can rewrite the scheme of EP1 as
\footnote{This form
 has been given in \cite{Li_Wu(sci2016)} for first-order ODEs.}
\begin{equation}
\begin{aligned}&u^{n+1}-2\cos(h\Omega)u^{n}+u^{n-1}\\
=&h\Big[\varphi_1(V)\displaystyle\int_{0}^{1}f((1-\sigma)u^{n}+\sigma
u^{n+1})d\sigma-\varphi_1(-V)
\displaystyle\int_{0}^{1}f((1-\sigma)u^{n-1}+\sigma
u^{n})d\sigma\Big].
\end{aligned}\label{MFE-2}%
\end{equation}
For the   term $(1-\sigma)u^{n}+\sigma u^{n+1}$,  we look for a
modulated Fourier expansion of the form
\begin{equation*}
\begin{aligned}\tilde{u}_h(t+\frac{h}{2},x,\sigma)=
\sum\limits_{\norm{k}\leq
K}w_{j(k)}^{k}\Big(\tilde{\epsilon}(t+\frac{h}{2}),\sigma\Big)\mathrm{e}^{\mathrm{i}(j(k)
\cdot x)} \mathrm{e}^{-\mathrm{i}(k
\cdot \omega) (t+\frac{h}{2})},
\end{aligned}
\end{equation*}
which leads to
\begin{equation}\label{xip}
\begin{aligned} w_{j(k)}^{k}\Big(\tilde{\epsilon}(t+\frac{h}{2}),\sigma\Big)
 =&L^{k}_3(\sigma)z_{j(k)}^{k}\Big(\tilde{\epsilon}(t+\frac{h}{2})\Big).
\end{aligned}
\end{equation}
 Likwise, for $(1-\sigma)u^{n-1}+\sigma u^{n}$, we have the following
modulated Fourier expansion
\begin{equation*}
\begin{aligned}\tilde{u}_h(t-\frac{h}{2},x,\sigma)=
\sum\limits_{\norm{k}\leq
K}w_{j(k)}^{k}\Big(\tilde{\epsilon}(t-\frac{h}{2}),\sigma\Big)\mathrm{e}^{\mathrm{i}(j(k)
\cdot x)} \mathrm{e}^{-\mathrm{i}(k
\cdot \omega)(t-\frac{h}{2})}.
\end{aligned}
\end{equation*}
Inserting  \eqref{MFE-EAVF} and \eqref{xip} into
\eqref{MFE-2} yields
\begin{equation*}
\begin{aligned}&\tilde{u}(t+h,x)-2\cos(h\Omega)\tilde{u}(t,x)+\tilde{u}(t-h,x)\\
=&h\Big[\varphi_1(V)\displaystyle\int_{0}^{1}f\big(\tilde{u}_{h}(t+\frac{h}{2},x,\sigma)\big)d\sigma-\varphi_1(-V)
\displaystyle\int_{0}^{1}f\big(\tilde{u}_{h}(t-\frac{h}{2},x,\sigma)\big)d\sigma\Big],
\end{aligned}
\end{equation*}
which can be expressed  by operators as \begin{equation}\label{new
q}
\begin{aligned}&(\varphi_1(\ii h \Omega)\mathrm{e}^{\frac{1}{2} hD}-\varphi_1(-\ii h
\Omega)\mathrm{e}^{-\frac{1}{2} hD})^{-1}( \mathrm{e}^{
hD}-2\cos(h\Omega)+ \mathrm{e}^{-
hD})\tilde{u}(t,x)=h\displaystyle\int_{0}^{1}f(\tilde{u}_{h}(t,x,\sigma))d\sigma.\end{aligned}\end{equation}
On the other hand, we rewrite the nonlinearity $f$ as:
\begin{equation*}
\begin{aligned}&f(u)
=-\ii\sum\limits_{\norm{k}\leq K}\sum\limits_{j(k)
\in\MM}\sum\limits_{k^1+k^2-k^3=k
}w^{k^1}_{l_1}w^{k^2}_{l_2}\overline{w^{k^3}_{l_3}}
e^{\mathrm{i}(j(k)\cdot x)}\mathrm{e}^{-\mathrm{i}(k
\cdot \omega) t},
\end{aligned}
\end{equation*}
where  $j(k)=(j(k^1)+j(k^2)-j(k^3))\ \textmd{mod}\ 2M$ if
$k=k^1+k^2-k^3$. On the basis of this fact and \eqref{new
q}, considering the $j$th Fourier coefficient and comparing the
coefficients of
$\mathrm{e}^{-\mathrm{i}(k\cdot \omega) t}$,
{the result of this proposition  is obtained.} \hfill
\hfill\end{proof}
%
%
%
%
%
%
\subsubsection{Iterative solution of  modulation system}\label{subsec:rev pic}
In order to achieve an approximate solution of the modulation system
\eqref{ljkqp}--\eqref{initial pl},  we introduce an iterative
procedure in this subsection which was used in \cite{12,19}.

For $j=j(k)$ with $k\neq\langle j\rangle$, the modulation system
  takes the form
\begin{equation}\label{Pic ite z-0}
\begin{aligned}&\ii h\omega_j\csc\big(\frac{1}{2}h\omega_j\big)\sin\big(\frac{1}{2}h(\omega_j-k
\cdot\omega)\big)z_{j(k)}^{k}(\tilde{\epsilon} t) =\NN(w(\tilde{\epsilon}
t))_{j(k)}^{k}+\BB(z(\tilde{\epsilon} t))_{j(k)}^{k},
\end{aligned}
\end{equation}
and for $j=j(\langle j\rangle)$, the modulation system becomes
\begin{equation}\label{Pic ite z2}
\begin{aligned}
&\frac{1}{2} \tilde{\epsilon}
h^2\omega_j\csc\big(\frac{1}{2}h\omega_j\big)\dot{z}_{j}^{\langle
j\rangle}(\tilde{\epsilon} t) =\NN(w(\tilde{\epsilon} t))_{j}^{\langle
j\rangle}+\AA(z(\tilde{\epsilon} t))_{j}^{\langle j\rangle},
\end{aligned}
\end{equation}
  where $\dot{z}_{j}^{\langle
j\rangle}$ stands for the derivative with respect to  $\tilde{\tau} =
\tilde{\epsilon} t$ and we  have used  the differential operators
\begin{equation*}
\begin{aligned}&\BB(z(\tilde{\epsilon} t))_{j(k)}^{k}=- \frac{1}{2}\tilde{\epsilon} h^2
\omega_j\csc\big(\frac{1}{2}h\omega_j\big)\cos\big(\frac{1}{2}h(k
\cdot\omega-\omega_j)\big)\dot{z}(\tilde{\epsilon}
t)_{j(k)}^{k}-\ldots,\\
&\AA(z(\tilde{\epsilon} t))_{j}^{\langle j\rangle}=-\frac{1}{48} \tilde{\epsilon}^3
h^4\omega_j\csc\big(\frac{1}{2}h\omega_j\big)z^{(3)}(\tilde{\epsilon}
t)_{j}^{\langle j\rangle}-\ldots,
\end{aligned}
\end{equation*}
and
$$\NN(w(\tilde{\epsilon}
t))_{j(k)}^{k}= -\ii h \sum\limits_{k^1+k^2-k^3=k
}\int_{0}^{1}w^{k^1}_{j(k^1)}(\tilde{\epsilon}
t,\sigma)w^{k^2}_{j(k^2)}(\tilde{\epsilon}
t,\sigma)\overline{w^{k^3}_{j(k^3)}}(\tilde{\epsilon}
t,\sigma)d\sigma.$$

Denote  by $[\cdot]^{l}$ the $l$th iterate and  we choose the
starting iterates ($l = 0$)   as $[z_j^{k}(\tilde{\tau})]^0=0$ for
$k\neq  \langle j\rangle$, and $[z_j^{\langle
j\rangle}(\tilde{\tau})]^0=u_j^0$. Then the modulation functions are
distinguished as follows.

\begin{mydef} \label{def:Iterative sol}{(\textbf{Iterative solution of  modulation system.})}
\begin{itemize}
\item For near-resonant indices $(j,
k)\in\mathcal{R}_{\tilde{\epsilon},M,h}$ or $\norm{k}>K=2N+2$, it
is set  for $0\leq \tilde{\epsilon} t=\tilde{\tau}\leq1$ that
$[z^{k}_j(\tilde{\tau})]^{l+1}=0.$

 \item For near-resonant indices $(j,
k)=(j,\langle j\rangle)$, in the light of \eqref{Pic ite
z2},  $\big[z_{j}^{\langle j\rangle}\big]^{l+1}$ is defined as the
solution of the differential equation
\begin{equation*}
\big[\dot{z}_{j}^{\langle j\rangle}(\tilde{\epsilon} t)\big]^{l+1}
=\Big[\frac{\sinc\big(\frac{1}{2}h\omega_j\big)}{h \tilde{\epsilon}}
\NN(w(\tilde{\epsilon} t))_{j}^{\langle
j\rangle}+\frac{\sinc\big(\frac{1}{2}h\omega_j\big)}{h
\tilde{\epsilon}}\AA(z(\tilde{\epsilon} t))_{j}^{\langle j\rangle}\Big]^{l}
\end{equation*}
with the initial value $\big[z_{j}^{\langle
j\rangle}(0)\big]^{l+1}=u_j^0-\Big[\sum\limits_{k\neq
\langle j\rangle} z_{j}^{k}(0)\Big]^{l}$ and $\sinc
(x)=\sin (x)/x.$

\item
 For the
remaining  indices $(j, k)$ in the set
\begin{equation}
\mathcal{L}_{\tilde{\epsilon},M,h}=\{(j, k):j=j(k),\
k\neq\langle j\rangle,\  (j,
k)\notin\mathcal{R}_{\tilde{\epsilon},M,h},\ \norm{k}\leq
K\},
 \label{near-resonant l}%
\end{equation}
it follows from   \eqref{Pic ite z-0} that
\begin{equation}\label{Pic ite z}
\begin{aligned}&\big[z_{j(k)}^{k}(\tilde{\epsilon} t)\big]^{l+1} =\Big[\frac{\sinc\big(\frac{1}{2}h\omega_j\big)}{2\ii
\sin\big(\frac{1}{2}h(\omega_j-k
\cdot\omega)\big)}\big(\NN(w(\tilde{\epsilon}
t))_{j(k)}^{k}+\BB(z(\tilde{\epsilon}
t))_{j(k)}^{k}\big)\Big]^{l}.
\end{aligned}
\end{equation}

\end{itemize}

\end{mydef}

It is noted that by this iterative construction, the iterated
modulation functions $[z_{j(k)}^{k}(\tilde{\epsilon} t)]^{l}$ are
polynomials in $\tilde{\epsilon} t$ of degree bounded in terms of the number
of iterations $l$.

\subsubsection{Rescaling}\label{subsec:res est}
Following \cite{12,19}, this subsection rescales and splits the
modulation functions in order to make  good  use of the powers of
$\tilde{\epsilon}$. By letting
$$ [[k]]=\left\{\begin{aligned} &
\max(2,(\norm{k}+1)/2),\quad k\neq \langle
j\rangle,\\
&(\norm{k}+1)/2=1,\qquad \ \ \  k=\langle
j\rangle,
\end{aligned}\right.
$$
we split the functions $z_j^{k}$ into two parts
$z_j^{k}=\tilde{\epsilon}^{[[k]]}a_j^{k}+\tilde{\epsilon}^{[[k]]}b_j^{k},$
where $a_j^{k}$ denotes the ``diagonal" entries (i.e.,
$a_j^{k}\neq 0$ only for $k =\langle j\rangle$) and
$b_j^{k}$ presents the ``off-diagonal" entries (i.e.,
$b_j^{k}\neq 0$ only for $k \neq\langle j\rangle$). We
use the following notations
\begin{equation} \label{21a1}\aa=(a^{k})_{k}=(a^{k}_{j(k)}e^{\mathrm{i}(j(k)\cdot x)})_{k},\ \ \
\bb=(b^{k})_{k}=(b^{k}_{j(k)}e^{\mathrm{i}(j(k)\cdot
x)})_{k}\end{equation} and define the operator
\begin{equation} \label{21a2}(\WW\cc)_j^{k}=\left\{\begin{aligned} &
\frac{2\ii \sin\big(\frac{1}{2}h(\omega_j-k
\cdot\omega)\big)}{\sinc\big(\frac{1}{2}h\omega_j\big)}c_j^{k},\quad
(j,
k)\in\mathcal{L}_{\tilde{\epsilon},M,h},\\
& \tilde{\epsilon}^{\frac{1}{2}}hc_j^{k},\qquad \qquad \qquad\qquad\qquad
\textmd{else}.
\end{aligned}\right.
\end{equation}
Furthermore, we rescale the non-linearity $\NN(\www)$  by
\begin{equation} \label{21a3}\FF(\vv)_j^{k}=\tilde{\epsilon}^{-\max([[k]],2)}\NN(\www),\end{equation}
where $\vv=(v^{k})_{k}$ is defined by $
v^{k}=\tilde{\epsilon}^{-[[k]]}w^{k}=\tilde{\epsilon}^{-[[k]]}w^{k}_{j(k)}e^{\mathrm{i}(j(k)\cdot
x)}.$

We are now in a position to rewrite the iteration from the previous
subsection in these rescaled variables.
\begin{prop}  (\textbf{Rescaling.})
Using the above rescaled variables, the iteration given by
Definition \ref{def:Iterative sol} can be formulated as
\begin{equation} \label{variables iteration}
\begin{aligned}&\big[b_{j}^{k}\big]^{l+1}
=\big[(\WW^{-1}\BB(\bb))_{j}^{k}\big]^{l}+\big[(\WW^{-1}\FF(\vv))_{j}^{k}\big]^{l},\quad
(j, k)\in\mathcal{L}_{\tilde{\epsilon},M,h},\\
&\big[\dot{a}_{j}^{\langle j\rangle}\big]^{l+1}
=\frac{\sinc(\frac{1}{2}h\omega_j)}{h\tilde{\epsilon}}\big[(\AA(\aa))_{j}^{\langle
j\rangle}\big]^{l}+\frac{\sinc(\frac{1}{2}h\omega_j)}{h }\big[(
\FF(\vv))_{j}^{\langle j\rangle}\big]^{l},\\
&\big[a_{j}^{\langle j\rangle}(0)\big]^{l+1}
=\tilde{\epsilon}^{-1}u_j^0-\Big[\sum\limits_{k\neq \langle
j\rangle}\tilde{\epsilon}^{[[k]]-1}b_{j}^{k}(0)\Big]^{l},
\end{aligned}
\end{equation}
where $[v_{j}^{k}]^{l} =
\tilde{\epsilon}^{-[[k]]}[w_{j}^{k}]^{l}$ defined by \eqref{eta
modula sys}.

Another rescaling of the variables will be used in this section
$$\hat{a}_{j}^{k}=\abs{\omega^{\frac{2s-d-1}{4}\abs{k}}}a_{j}^{k},\ \
\hat{b}_{j}^{k}=\abs{\omega^{\frac{2s-d-1}{4}\abs{k}}}b_{j}^{k},\
\
\hat{v}_{j}^{k}=\abs{\omega^{\frac{2s-d-1}{4}\abs{k}}}v_{j}^{k}.$$
For these rescaled variables,  the iteration for $\hat{\bb}$ becomes
$$\big[\hat{b}_{j}^{k}\big]^{l+1}
=\big[(\WW^{-1}\BB(\hat{\bb}))_{j}^{k}\big]^{l}+\big[(\WW^{-1}\hat{\FF}(\hat{\vv}))_{j}^{k}\big]^{l},\quad
(j, k)\in\mathcal{L}_{\tilde{\epsilon},M,h},$$ where
$\hat{\FF}(\hat{\vv})_{j}^{k}=\abs{\omega^{\frac{2s-d-1}{4}\abs{k}}}\FF(\vv)_{j}^{k}$.
\end{prop}

\subsubsection{Size of the iterated modulation functions}\label{subsec:Size}
In this subsection, we will control the size of the iterated
modulation functions.  The    norm $|||\mathbf{z}|||_s^2=\sum\limits_{j}\abs{\omega_j}^s
\big(\sum\limits_{k}\abs{z_j^{k}}\big)^2$ (see, e.g. \cite{12})
will be used in the rest of this paper.

Before presenting the size of the iterated modulation functions, we
first need to estimate the bounds of the operator $\WW$ \eqref{21a2} and the
non-linearity $\FF$ \eqref{21a3}.

\begin{prop}\label{pro: lem1} {(\textbf{Bounds of the operator $\WW$ and the
non-linearity $\FF$.})} The following bounds hold
\begin{equation}\label{bounds lem1}
\begin{aligned}
&|||\WW^{-1}\vv|||_s\leq \tilde{\epsilon}^{-\frac{1}{2}}h^{-1}|||\vv|||_s,\
\
 |||\FF(\vv)|||_s\leq C \tilde{\epsilon} h|||\check{\vv}|||_s^3,\\
&|||\FF(\mathbf{v_1})-\FF(\mathbf{v_2})|||_s\leq C \tilde{\epsilon}
h|||\check{\mathbf{v}}_1-\check{\mathbf{v}}_2|||_s
\max(|||\check{\mathbf{v}}_1|||_s,|||\check{\mathbf{v}}_2|||_s)^2,
\end{aligned}
\end{equation}
where  $\check{\mathbf{v}}=\sup_{0\leq
\sigma\leq1}\{\vv(\tilde{\tau},\sigma)\}$  and the constant $C$  is
independent of $\tilde{\epsilon}$ but depends on $d, s,$ and $V$. The same
estimates are ture for $\hat{\vv}, \hat{\FF}$ and
$|||\cdot|||_{\frac{d+1}{2}}$ instead of $\vv, \FF$ and
$|||\cdot|||_{s}$, respectively.
\end{prop}
\begin{proof} The proof is given in Appendix I. \end{proof}

 We next consider  the iterated modulation functions given in
\eqref{21a1}. Their sizes are controlled by the
following result.
\begin{prop}\label{pro: f} {(\textbf{Size of the iterated modulation functions.})}
For $0\leq\tilde{\tau}=\tilde{\epsilon} t\leq 1$ and for all $l\geq 0$, it is true
that
\begin{equation}\label{bounds f12}
\begin{aligned}
&|||\big[\mathbf{a}(\tilde{\tau})\big]^{l}|||_s\leq C,\ \
|||\big[\mathbf{a}^{(n)}(\tilde{\tau})\big]^{l}|||_s\leq C\tilde{\epsilon},\ \
\textmd{for}\ \ n\geq1,\\
&|||\big[\mathbf{b}^{(n)}(\tilde{\tau})\big]^{l}|||_s\leq
C\tilde{\epsilon}^{\frac{1}{2}},\ \ \textmd{for}\ \ n\geq0,
\end{aligned}
\end{equation}
 where the constant $C$ depends only on $C_0, d, n,
s$ and the norm of $V$. For  $\hat{\aa}$ and $\hat{\bb}$ instead of
$\aa$ and $\bb$, the same estimates  are true if  $|||\cdot|||_s$ is
replaced by $|||\cdot|||_{\frac{d+1}{2}}$. It follows from these
bounds that the modulated Fourier expansion of the numerical scheme
$\tilde{u}$ is bounded by
\begin{equation}\label{O}\norm{\tilde{u}(t,\cdot)}_s\leq C\tilde{\epsilon}\end{equation}
and   its coefficients $z$ are  controlled by
\begin{equation}\label{OO}\sum\limits_{j\in \mathcal{M}}\abs{\omega_j}^s
 \abs{z_j^{\langle j\rangle}}^2\leq C\tilde{\epsilon}^2,\ \
\sum\limits_{j\in \mathcal{M}}\abs{\omega_j}^s
\big(\sum\limits_{k\neq\langle j\rangle}\abs{z_j^{k}}\big)^2\leq
C\tilde{\epsilon}^5.\end{equation}
\end{prop}
\begin{proof} The proof is given in Appendix II. \end{proof}

\subsubsection{Defect of the iterated modulation functions}\label{subsec:Defect}

After $l$ iterations, the defect in the modulation system \eqref{Pic
ite z} with the initial value \eqref{initial pl} has the form
\begin{equation}\label{Defect z}
\begin{aligned}&\big[d_j^{k}\big]^{l}=\Big[\ii h\omega_j\csc\big(\frac{1}{2}h\omega_j\big)\sin\big(\frac{1}{2}h(\omega_j-k
\cdot\omega)\big)z_{j}^{k}-\BB(\zz)_{j}^{k}-\NN(\www)_{j}^{k}\Big]^{l},\\
&\big[\tilde{d}_j^{\langle j\rangle}(0)\big]^{l} =u_j^0-\Big[
\sum\limits_{k} z_{j(k)}^{k}(0)\Big]^{l}.
\end{aligned}
\end{equation}
Clearly, it can be   decomposed into four parts:
$\big[d_j^{k}\big]^{l}=\big[e_j^{k}+f_j^{k}+g_j^{k}+\dot{h}_j^{k}\big]^{l},$
where $[e_j^{k}]^{l}=0$  for  $(j,
k)\in\mathcal{L}_{\tilde{\epsilon},M,h}$,  $[f_j^{k}]^{l}=0$ for
non-near resonant indices $(j,
k)\in\mathcal{R}_{\tilde{\epsilon},M,h}$, $[\dot{h}_j^{k}]^{l}=0$
for $ k\neq \langle j\rangle$, and $[g_j^{k}]^{l}=0$ for
$ \norm{k}\leq K$. The size of each part can be estimated
as follows.
\begin{prop}\label{pro: defect} {(\textbf{Defect of the iterated modulation
functions.})} For all $l\geq 0$ and for $0\leq\tilde{\tau}=\tilde{\epsilon} t\leq
1$, it is true that
\begin{equation}\label{defects lem1}
\begin{aligned}
&|||[\mathbf{f}(\tilde{\tau})]^{l}|||_s\leq C \tilde{\epsilon}^{N+3}h,\ \
  |||[\mathbf{g}(\tilde{\tau})]^{l}|||_s\leq C \tilde{\epsilon}^{N+3}h,\\
&|||[\mathbf{e}(\tilde{\tau})]^{l}|||_s\leq C \tilde{\epsilon}^{\frac{p+4}{2}}h,\ \
  |||[\dot{\mathbf{h}}(\tilde{\tau})]^{l}|||_s\leq C \tilde{\epsilon}^{\frac{p+4}{2}}h,\\
&|||[\tilde{\mathbf{d}}(0)]^{l}|||_s\leq C
\tilde{\epsilon}^{\frac{p+2}{2}}h,
\end{aligned}
\end{equation}
where the constant $C$ depends on $C_0 , d, p, s$ and the norm of
$V$. We have the same estimates for $\hat{\mathbf{e}}$ and
$\hat{\mathbf{h}}$ instead of $\mathbf{e}$ and $\mathbf{h}$ provided
 $|||\cdot|||_s$ is replaced by $|||\cdot|||_{\frac{d+1}{2}}$.
\end{prop}
\begin{proof} The proof is given in Appendix III. \end{proof}

\subsubsection{The numerical solution on short time intervals}\label{subsec:solution on short time}
In this subsection,  the size of the numerical solution $u^n$ on a
short time interval of length $\tilde{\epsilon}$ is studied. It is noted
that since the considered integrator is implicit,   fixed point
arguments are considered for EP1 and we rewrite it as  the
following scheme
\begin{equation}\label{newEAVFmethod}%
\begin{aligned}
U^{n+1}=e^{V}u^{n}+h\varphi_1(V)
\displaystyle\int_{0}^{1}f((1-\sigma)u^{n}+\sigma U^{n+1})d\sigma,\\
u^{n+1}=e^{V}u^{n}+h\varphi_1(V)
\displaystyle\int_{0}^{1}f((1-\sigma)u^{n}+\sigma U^{n+1})d\sigma.
\end{aligned}
\end{equation}

\begin{prop}\label{pro: on short time} {(\textbf{The numerical solution on short time intervals.})} For $0\leq t_n=nh\leq \tilde{\epsilon}^{-1}$ with a
sufficiently small $\tilde{\epsilon}$, it is obtained that $||u^n||_s\leq
2\tilde{\epsilon}.$
\end{prop}
\begin{proof}
This result is proved by   induction on $n$ that
\begin{equation}\label{induc}
||u^n||_s\leq \tilde{\epsilon}+125 Cnh \tilde{\epsilon}^3\qquad \textmd{for} \quad
0\leq nh  \leq \tilde{\epsilon}^{-1}
\end{equation}
and by letting $\tilde{\epsilon}$ be sufficiently small compared to $C$.

For $n = 0$ the estimate \eqref{induc} is clear by considering
\eqref{initi cond}. For $n
> 0$, it follows from  the definition of the
integrator that
\begin{equation}\begin{aligned}\label{for123}
\norm{u^{n}}_s\leq&\norm{u^{n-1}}_s+h\norm{
\displaystyle\int_{0}^{1}f((1-\sigma)u^{n-1}+\sigma
U^{n})d\sigma}_s\\
\leq&\norm{u^{n-1}}_s+h   \big(\norm{u^{n-1}}_s+ \norm{U^{n}}_s
\big)^3,
\end{aligned}\end{equation}
where    (4.9) of \cite{12} is used and $U^{n}$ is  a fixed point of
$$G: U\rightarrow e^{V}u^{n-1}+h\varphi_1(V)
\displaystyle\int_{0}^{1}f((1-\sigma)u^{n-1}+\sigma U)d\sigma.$$ For
$0\leq nh  \leq \tilde{\epsilon}^{-1}$, since $ \norm{u^{n-1}}_s\leq
2\tilde{\epsilon}\leq 3\tilde{\epsilon}$,  the function $G$ maps   the ball $\{U:
\norm{U}_s\leq 3\tilde{\epsilon}\}$ to itself.  Furthermore, using  (4.11)
in \cite{12},   we obtain
\begin{equation*}\begin{aligned}
&\norm{h\varphi_1(V)
\displaystyle\int_{0}^{1}f((1-\sigma)u^{n-1}+\sigma
U)d\sigma-h\varphi_1(V)
\displaystyle\int_{0}^{1}f((1-\sigma)u^{n-1}+\sigma
\tilde{U})d\sigma}_s\\
\leq&3 C \max\big(\norm{u^{n-1}}_s+\norm{
U}_s,\norm{u^{n-1}}_s+\norm{\tilde{U}}_s\big)^2
\norm{U-\tilde{U}}_s.
\end{aligned}\end{equation*}
 This  shows that the map $G$ has a Lipschitz constant smaller than
one for sufficiently small $\tilde{\epsilon}$ in the norm $\norm{\cdot}_s$
on the ball $\{U: \norm{U}_s\leq 3\tilde{\epsilon}\}$.  In view of  the
Banach fixed point theorem, one has $||U^{n}||_s\leq 3\tilde{\epsilon}$ for
the fixed point $U^{n}$ of $G$. Therefore, \eqref{induc} can be
obtained by    the induction hypothesis applied to \eqref{for123}.
 \end{proof}

\subsubsection{The error between the modulated Fourier expansion and the numerical solution}\label{subsec:MFE NS}
This subsection pays attention to the error $u^n-\tilde{u}(t,x)$
between the numerical solution $u^n$ and the modulated Fourier
expansion
$$\tilde{u}(t,x)=
\sum\limits_{k}[z_{j(k)}^{k}(\tilde{\epsilon}
t)]^L \mathrm{e}^{\mathrm{i}(j \cdot x)}
\mathrm{e}^{-\mathrm{i}(k \cdot\omega) t},$$
where the iterated modulation functions $z_{j}^{k} =
[z_{j}^{k}]^L$ after $L := 2N + 2$ iterations replace the
exact solution of the modulation system which is not available in
fact. For brevity, the index $L$ in the following analysis is
omitted.

\begin{prop}\label{pro: small error} {(\textbf{The error between the modulated Fourier expansion and the numerical solution.})}
For $0\leq t_n=nh\leq \tilde{\epsilon}^{-1}$, it is obtained that
\begin{equation}\label{small error}
\begin{aligned}
&||u^n-\tilde{u}(t_n,x)||_s\leq C\tilde{\epsilon}^{N+2}
\end{aligned}
\end{equation}
for $\tilde{\epsilon}$ sufficiently small compared to $d, s$ and the norm of
the potential $V$.
\end{prop}
\begin{proof}
As stated in the previous  subsection,   fixed point arguments are
employed.   By the definition of the modulation system
\eqref{ljkqp}-\eqref{eta modula sys} and fixed point arguments, it
is arrived at that
\begin{equation*}
\begin{aligned}&\tilde{u}(t_n,x)=e^{V}\tilde{u}(t_{n-1},x)+h\varphi_1(V)
\displaystyle\int_{0}^{1}f((1-\sigma)\tilde{u}(t_{n-1},x)+\sigma
\tilde{U}(t_{n},x))d\sigma+\delta(t_{n},x)\end{aligned}\end{equation*}
with the defect $\delta(t,x)=
\sum\limits_{k}d_{j(k)}^{k}(\tilde{\epsilon} t)
\mathrm{e}^{\mathrm{i}(j \cdot x)}
\mathrm{e}^{-\mathrm{i}(k \cdot\omega)
(t+h)}.$ Here we have the following result
\begin{equation*}
\begin{aligned}(1-\sigma)\tilde{u}(t_{n-1},x)+\sigma
\tilde{U}(t_{n},x)= \sum\limits_{\norm{k}\leq
K}w_{j(k)}^{k}\Big(\tilde{\epsilon}(t+\frac{h}{2}),\sigma\Big)\mathrm{e}^{\mathrm{i}(j(k)
\cdot x)} \mathrm{e}^{-\mathrm{i}(k
\cdot\omega) (t+\frac{h}{2})}.
\end{aligned}
\end{equation*}
It follows from Proposition \ref{pro: defect} that
$||\delta(t,x)||_s\leq Ch\tilde{\epsilon}^{N+3}$
 for $0\leq t\leq \tilde{\epsilon}^{-1}$, where the constant $C$ depends
on $C_0, d, N, s$ and the norm of the potential $V$.

$\bullet$ Proof of the difference $U^{n}-\tilde{U}(t_{n},x)$.

For the solution $U^{n}$ appearing in the numerical method
\eqref{newEAVFmethod}, we first examine the difference
$U^{n}-\tilde{U}(t_{n},x)$.     By \eqref{ljkqp}-\eqref{eta modula
sys} and \eqref{Defect z}, $\tilde{U}(t_{n},x)$ is a fixed point of
$$\tilde{G}: \tilde{U}\rightarrow e^{V}\tilde{u}(t_{n-1},x)+h\varphi_1(V)
\displaystyle\int_{0}^{1}f((1-\sigma)\tilde{u}(t_{n-1},x)+\sigma
\tilde{U})d\sigma+\delta(t_{n-1},x).$$

 Obviously, it follows from the proof of Proposition \ref{pro: on
short time}  that  the fixed point iteration $[U]^l = G([U]^{l-1}),\
[U]^0 = e^{V}u^{n-1}$ converges in the norm $\norm{\cdot}_s$ to
$U^{n}$ and is bounded in this norm by $3\tilde{\epsilon}$. In what follows,
 we study the error between $[U]^l$ and $\tilde{U} =
\tilde{U}(t_{n},x)$, i.e., $[U]^l -\tilde{U}$, for $l=0,\ldots,$.

 On noticing the fact   $||\tilde{U} ||_s\leq C\tilde{\epsilon}$ by
Proposition \ref{pro: f} and the property (4.9) of \cite{12}, we
obtain the estimate of the defect for $l=0$
\begin{equation*}
\begin{aligned}&\norm{[U]^0-\tilde{U}}_s=\norm{e^{V}u^{n-1}-\tilde{G}(\tilde{U})}_s\\
\leq& \norm{u^{n-1}-\tilde{u}(t_{n-1},x)}_s+h\norm{
\int_{0}^{1}f((1-\sigma)\tilde{u}(t_{n-1},x)+\sigma
\tilde{U})d\sigma}_s+
\norm{\delta(t_{n-1},x)}_s\\
\leq&
\norm{u^{n-1}-\tilde{u}(t_{n-1},x)}_s+Ch\tilde{\epsilon}^{3}+Ch\tilde{\epsilon}^{N+3}.
\end{aligned}
\end{equation*}
For $l > 0$,  using (4.11) of \cite{12} gives  that
\begin{equation*}
\begin{aligned}&\norm{[U]^l-\tilde{U}}_s=\norm{G([U]^{l-1})-\tilde{G}(\tilde{U})}_s\\
\leq&
\norm{u^{n-1}-\tilde{u}(t_{n-1},x)}_s+Ch\tilde{\epsilon}^{2}\norm{[U]^{l-1}-\tilde{U}}_s+Ch\tilde{\epsilon}^{N+3}
\end{aligned}
\end{equation*}
with a constant $C$ independent of $l$. This leads to a recursion on
$l$ as follows
\begin{equation*}
\begin{aligned}\norm{[U]^l-\tilde{U}}_s
\leq\big(\norm{u^{n-1}-\tilde{u}(t_{n-1},x)}_s+Ch\tilde{\epsilon}^{N+3}\big)\sum\limits_{j=0}^l(Ch\tilde{\epsilon}^{2})^j+Ch\tilde{\epsilon}^{3}(Ch\tilde{\epsilon}^{2})^l.
\end{aligned}
\end{equation*}
Considering $l \rightarrow\infty$ and
$Ch\tilde{\epsilon}^{2}\leq\frac{1}{2}$ implies
\begin{equation}\label{UU}
\norm{U^{n}-\tilde{U}(t_{n},x)}_s \leq2
\norm{u^{n}-\tilde{u}(t_{n},x)}_s+2Ch\tilde{\epsilon}^{N+3}.
\end{equation}

$\bullet$ Proof of the difference $u^n-\tilde{u}(t_n,x)$.

We are now in a position to consider $u^n-\tilde{u}(t_n,x)$. When $n
> 0$,  using (4.11) of \cite{12} gives
\begin{equation*}
\begin{aligned}\norm{u^n-\tilde{u}(t_n,x)}_s
\leq
\norm{u^{n-1}-\tilde{u}(t_{n-1},x)}_s+Ch\tilde{\epsilon}^{2}\norm{U^{n}-\tilde{U}(t_{n},x)}_s
+Ch\tilde{\epsilon}^{N+3}.
\end{aligned}
\end{equation*}
Considering again the result \eqref{UU},  we have by induction on
$n$
\begin{equation}\label{rea induc}
\norm{u^n-\tilde{u}(t_n,x)}_s \leq
(1+2Ch\tilde{\epsilon}^{2})^n\big(Cnh\tilde{\epsilon}^{N+3}+\norm{u^0-\tilde{u}(0,x)}_s\big).
\end{equation}
On the other hand,  by Proposition \ref{pro: defect} with the defect
$\tilde{\mathbf{d}}$  in the initial condition, we have
$\norm{u^0-\tilde{u}(0,x)}_s\leq|||[\tilde{\mathbf{d}}(0)]^n|||_s\leq
C \tilde{\epsilon}^{N+3}.$ This result together with \eqref{rea induc}
 guarantees  the desired result if $\tilde{\epsilon}$ is sufficiently small.

 \end{proof}

\subsubsection{Almost invariants close to the actions}\label{subsec:Almost invariants}
In what follows, we show  \emph{an invariant} of the
modulation system   and  \emph{its relationship with the
  actions}.

\begin{prop} {(\textbf{Almost invariant.})}  There exits  $\tilde{\epsilon} \mathcal{J}_{\langle
j\rangle}(\tilde{\tau})$ such that $$
\sum\limits_{j\in  \mathcal{M}}\abs{\omega_j}^{s} \left|\frac{d}{d
\tilde{\tau}} \mathcal{J}_{\langle j\rangle}(\tilde{\tau})\right|\leq
Ch\tilde{\epsilon}^{N+3},$$
where  $\tilde{\tau}\leq 1$  and  $C$ depends on $\max_{j\in
\mathcal{M}}\big \{ \frac{1}{\abs{\cos(\frac{1}{2}h
\omega_j)}}\big\}$.  Moreover, it is true that $$ \mathcal{J}_{\langle j\rangle}(\tilde{\tau})= \frac{1}{2}
\abs{z_{j}^{\langle j\rangle}(\tilde{\tau})}^2+\mathcal{O}\Big(h
\tilde{\epsilon}^2\Big).$$
\end{prop}
\begin{proof} Let $$\mathcal{U}(\www)= \sum\limits_{k^1+k^2-k^3-k^4=0
}\frac{1}{(2\pi)^d}\int_{[-\pi,\pi]^d}\int_{0}^{1}w^{k^1}w^{k^2}
\overline{w^{k^3}}\overline{w^{k^4}}d\sigma dx.$$
From the above analysis, we can write the defect formula $d^{k}$ as
\begin{equation}\label{djk pote}
\begin{aligned}
& \tilde{L}^{k} z^{k}=-\ii h
\sum\limits_{k^1+k^2-k^3=k
}\int_{0}^{1}w^{k^1}_{j(k^1)}w^{k^2}_{j(k^2)}\overline{w^{k^3}_{j(k^3)}}d\sigma+
d^{k}.
\end{aligned}
\end{equation}
Here  we use  $\tilde{L}^{k}$ to denote  the truncation of
the operator $L^{k}$ after the $\tilde{\epsilon}^N$ term. The
transformation $w^{k}\rightarrow \mathrm{e}^{\mathrm{i}(k
\cdot \mu) \theta} w^{k}$ for real sequences
$\mu = (\mu_l)_{l\geq0}$ and $\theta\in \RR$ and the
choice of $k=\langle j\rangle$ leaves $\mathcal{U}$ invariant
\begin{equation*}
\begin{aligned}
&0=h\frac{d}{d\theta}\mid_{\theta=0}
\mathcal{U}\Big((\mathrm{e}^{\mathrm{i}(\langle j\rangle \cdot
\mu) \theta} w^{\langle j\rangle})_{\langle j\rangle}\Big) \\
=& -4h \textmd{Re}\Big(\sum\limits_{j} \mathrm{i}(\langle j\rangle
\cdot \mu) \overline{ w_{j}^{\langle j\rangle}}
\sum\limits_{k^1+k^2-k^3=\langle j\rangle
}\int_{0}^{1}w^{k^1}_{j(k^1)}w^{k^2}_{j(k^2)}\overline{w^{k^3}_{j(k^3)}}d\sigma\Big)\\
=& 4 \textmd{Re}\Big(\sum\limits_{j} (\langle j\rangle \cdot
\mu) \overline{w_{j}^{\langle j\rangle}}\big(\tilde{L}^{\langle j\rangle} z^{\langle j\rangle}_j-d^{\langle j\rangle}_j\big)\Big)\\
=& 4 \textmd{Re}\Big(\sum\limits_{j} (\langle j\rangle \cdot
\mu) \overline{L_3^{\langle
j\rangle}(\sigma)z_{j}^{\langle j\rangle}}\big(\tilde{L}^{\langle
j\rangle} z^{\langle j\rangle}_j-d^{\langle j\rangle}_j\big)\Big).
\end{aligned}
\end{equation*}

Since the right-hand side is independent of $\sigma$, we choose
$\sigma=1/2$ in the following analysis. With  the above formula, we
have
\begin{equation}
\begin{aligned}
 4 \textmd{Re} \sum\limits_{j} (\langle j\rangle \cdot
\mu)  \overline{L_3^{\langle
j\rangle}(1/2)z_{j}^{\langle j\rangle}} \tilde{L}^{\langle j\rangle}
z^{\langle j\rangle}_j=4 \textmd{Re} \sum\limits_{j} (\langle
j\rangle \cdot \mu) \overline{L_3^{\langle
j\rangle}(1/2)z_{j}^{\langle j\rangle}} d^{\langle j\rangle}_j.
\end{aligned}
\label{almost 1-2}%
\end{equation}
By the expansions   of $L_3^{\langle j\rangle}(1/2)$ and
$\tilde{L}^{\langle j\rangle}$ and the ``magic formulas" on p. 508
of \cite{hairer2006}, it is known that the left-hand side of
\eqref{almost 1-2} is a total derivative of function $\tilde{\epsilon}
\mathcal{J}_{\mu}(\tilde{\tau})$. Therefore \eqref{almost 1-2}
is identical to $$\tilde{\epsilon}\frac{d}{d \tilde{\tau}}
\mathcal{J}_{\mu}=4 \textmd{Re} \sum\limits_{j}
(\langle j\rangle \cdot \mu)  \overline{L_3^{\langle
j\rangle}(1/2)z_{j}^{\langle j\rangle}} d^{\langle j\rangle}_j.$$
Considering the special case of
$\mu=\frac{\sinc(\frac{1}{2}h
\omega_j)}{\cos(\frac{1}{2}h \omega_j)}\langle j\rangle$  and
for the first result, it needs to prove that
$$\sum\limits_{j\in
\mathcal{M}}\abs{\omega_j}^{s}\abs{\frac{\sinc(\frac{1}{2}h
\omega_j)}{\cos(\frac{1}{2}h \omega_j)}}
\left|\overline{L_3^{\langle j\rangle}(1/2)z_{j}^{\langle j\rangle}}
d^{\langle j\rangle}_j\right|\leq Ch\tilde{\epsilon}^{N+4}. $$

By the
property of $L_3$, we have $$\sum\limits_{j\in
\mathcal{M}}\abs{\omega_j}^{s} \abs{\frac{\sinc(\frac{1}{2}h
\omega_j)}{\cos(\frac{1}{2}h \omega_j)}}\left|\overline{L_3^{\langle
j\rangle}(1/2)z_{j}^{\langle j\rangle}} d^{\langle
j\rangle}_j\right|\leq  C\sum\limits_{j\in
\mathcal{M}}\abs{\omega_j}^{s} \abs{  z_{j}^{\langle j\rangle}}
\abs{\dot{h}^{\langle j\rangle}_j }.$$
Taking advantage of
Cauchy-Schwarz inequality, one gets
\begin{equation*}
\begin{aligned}
&\sum\limits_{j\in
\mathcal{M}}\abs{\omega_j}^{s}\abs{\frac{\sinc(\frac{1}{2}h
\omega_j)}{\cos(\frac{1}{2}h \omega_j)}}
\left|\overline{L_3^{\langle j\rangle}(1/2)z_{j}^{\langle j\rangle}}
d^{\langle j\rangle}_j\right| \\
&\leq C \sqrt{\sum\limits_{j\in
\mathcal{M}}\big(\abs{\omega_j}^{\frac{s}{2}}\big)^2 \abs{
z_{j}^{\langle j\rangle}}^2} \sqrt{\sum\limits_{j\in
\mathcal{M}}\big(\abs{\omega_j}^{\frac{s}{2}}\big)^2
\abs{\dot{h}^{\langle j\rangle}_j }^2}\\
&\leq C\sqrt{\tilde{\epsilon}^2}\sqrt{h^2\tilde{\epsilon}^{p+4}}= Ch
\tilde{\epsilon}^{\frac{p}{2}+3}= Ch \tilde{\epsilon}^{\frac{L}{2}+3},
\end{aligned}
\end{equation*}
where the results \eqref{OO} and \eqref{defects lem1} are used here.
The first statement is immediately obtained by considering $L=
2N+2$.

Then, using  the Taylor expansions of $L_3^{\langle j\rangle}(1/2)$
and $L^{\langle j\rangle}$ and the  ``magic formulas" on p. 508 of
\cite{hairer2006} gives the construction of $ \mathcal{J}_{\langle
j\rangle}$.  \hfill\end{proof}

{After obtaining the almost invariant, its relationship with the actions is derived below.}
\begin{prop} {(\textbf{The relationship between the almost invariant and the actions.})}  It is true that
$\sum\limits_{j\in  \mathcal{M}}\abs{\omega_j}^{s} \left|
\mathcal{J}_{\langle j\rangle}(\tilde{\tau})-I_j(u^n,
\overline{u^n})\right|\leq C\tilde{\epsilon}^{\frac{7}{2}},$
where  $\tilde{\tau}\leq 1$. 
\end{prop}
\begin{proof}
This result can be obtained by following the proof of  Proposition 6
given in \cite{18}.
 \hfill\hfill\end{proof}

 \subsubsection{Near-conservation of  actions, density and momentum}\label{subsec:momentum}
 According to the analysis stated above, we {consider}  the
interface between the modulated Fourier expansions and {extend}  it
from short to long time intervals  in the same way used in Sects.
4.10-4.11 of \cite{12}. Then the  near conservation  of actions
given in Theorem \ref{main theo} is obtained. Meanwhile,  it follows
from the results presented in Sect. 6.4 of \cite{19} and Sect. 4.11
of \cite{12} that the long-time near-conservation of actions implies
the long-time near-conservation of   density and of momentum.
Therefore, the other statements of Theorem \ref{main theo} are
proved.

This concludes the proof of  Theorem \ref{main theo} for the
integrator EP1.

\subsubsection{Proof for  EP2}\label{subsec:modification}
Consider the  one-point
  quadrature formula with $(\tilde{c}_1,\tilde{d}_1)$ and then the scheme of \eqref{cs ei-spe} becomes
\begin{equation}
 \begin{aligned}
u^{n+1}&=e^{V}u^n+h \tilde{d}_1A_{1,\tilde{c}_1}(V)  f\Big(C_{\tilde{c}_1}(V)u^n+A_{\tilde{c}_1,\tilde{c}_1}(V) A^{-1}_{1,\tilde{c}_1}(V)(u^{n+1}-e^{V}u^n) \Big).
\end{aligned} \label{cs ei-spe-new}%
\end{equation}
 In terms of  this formula, we can derive the modulation equations
for  the modulation functions $z_j^{k}$ as $L^{k}
z_j^{k}(\tilde{\epsilon} t)= -\ii h
\sum\limits_{k^1+k^2-k^3=k } z^{k^1}_{j(k^1)}(\tilde{\epsilon}
t)z^{k^2}_{j(k^2)}(\tilde{\epsilon}
t)\overline{z^{k^3}_{j(k^3)}}(\tilde{\epsilon} t)$
 by defining
\begin{equation*}
\begin{aligned}L^{k}:&=\big(A_{\tilde{c}_1,\tilde{c}_1} A^{-1}_{1,\tilde{c}_1}(\mathrm{e}^{-\mathrm{i}(k \cdot \omega)
h}\mathrm{e}^{\tilde{\epsilon} hD}- \mathrm{e}^{\ii
h\Omega})+C_{\tilde{c}_1}\big)^{-1}(\mathrm{e}^{-\mathrm{i}(k
\cdot \omega) h}\mathrm{e}^{\tilde{\epsilon} hD}-
\mathrm{e}^{\ii h\Omega})
(\tilde{d}_1B_{\tilde{c}_1})^{-1}.\end{aligned}
\end{equation*} It can be seen that this formula has more concise expression than that of EP1.
 Then by
modifying the nonlinearity and concerning the property of
$L^{k}$, the analysis given above can be changed accordingly for  EP2.

\begin{rem}It is noted that the scheme \eqref{cs ei-spe-new} has been analysed in
\cite{12}. Under an assumption  on the coefficient functions of exponential integrator,
long term conservations have been derived there. However, for the coefficients $A_{\tilde{c}_1,\tilde{c}_1}(V), A_{1,\tilde{c}_1}(V)$  of EP2,
they do not satisfy that assumption required  in \cite{12}. Thus the part 4.2 of the proof given in \cite{12} cannot be used for EP2. Therefor we consider the above approach to proving the result. On the other side, the operator $L^{k}$ determined by EP3 does not have similar property as \eqref{L expansion}. Therefore, there is no invariant  of the
modulation system  and the near conservations are not true for EP3.
\end{rem}

\section{Numerical experiment}\label{sec:num exp}
\renewcommand\arraystretch{1.0}
\begin{table}[t!]$$
\begin{array}{|c|c|c|c|c|c|}
\hline
\text{Methods} &\text{Energy conservation} &\text{Optimal convergence}   & \text{Near conservations}      \\
\hline
 \text{EP1} & \surd    &\times\  (h^2)                 &    \surd \cr
  \text{EP2} & \surd & \surd \ (\varepsilon h^2)    & \surd \cr
 \text{EP3} & \surd  &\surd\  (\varepsilon h^3)     &  \times \cr
 \hline
\end{array}
$$
\caption{Properties of the methods.} \label{praERKN}
\end{table}
For the algorithms presented in this paper, their properties are
summarized in Table \ref{praERKN}. In order to show their advantages,  we choose the second-order
explicit exponential integrator  which is termed  pseudo
steady-state approximation which was given in
\cite{Verwer94} (denoted by EEI) and the fourth-order explicit exponential Runge--Kutta method which was given in
  \cite{Hochbruck05} (denoted by IEI4).
As a numerical experiment, we consider the problem with $d=1$ and
$\lambda=-2$ and the pseudospectral method with 64
points. In the practical computations, we apply  the three-point Gauss-Legendre's
rule to the integral in \eqref{cs ei} and use a fixed-point
iteration  with the error tolerance $10^{-16}$ and  the maximum
number $100$  for each iteration. In order to show the obtained methods behave well for different initial and boundary conditions, we will use various conditions in the experiment.

\textbf{Energy conservation.}
The initial value is given by $u^0(x)= 0.5\ii + 0.025
\cos(\mu x)$ and the periodic boundary condition is $u(t,0)=u(t,L)$. We
consider $L =4\sqrt{2}\pi$   and  integrate this problem  on $[0, 100]$  with $h=1/100$ for different $\varepsilon$. The conservation of discretised energy is shown in
Figures \ref{fig1}. From these results, it can be seen clearly that the EP integrators EP1-EP3 preserve the energy with a very good
accuracy, which supports the results of Theorem \ref{thm:ep}.
  \begin{figure}[h]
\centering
\includegraphics[width=6.5cm,height=5.0cm]{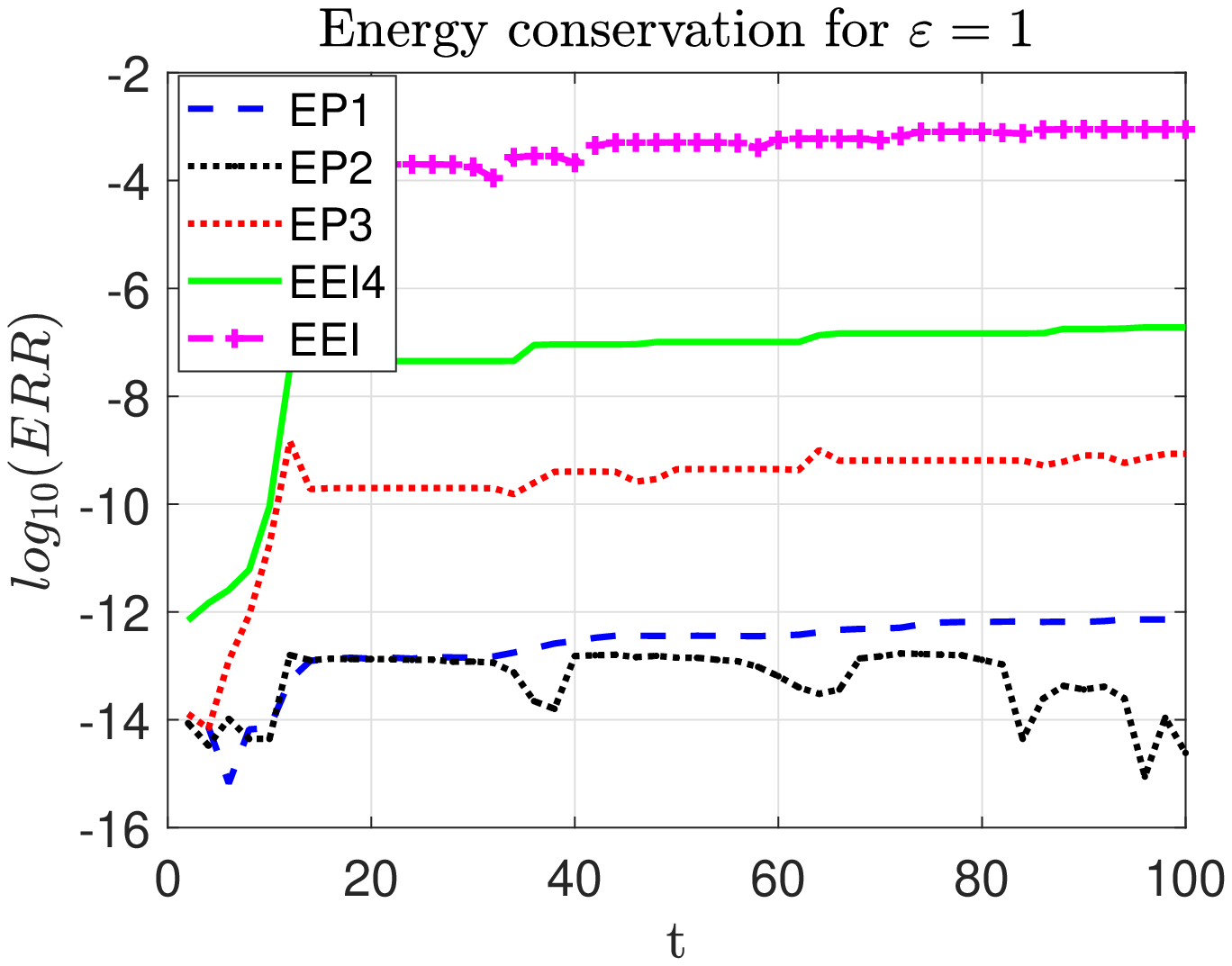}
\includegraphics[width=6.5cm,height=5.0cm]{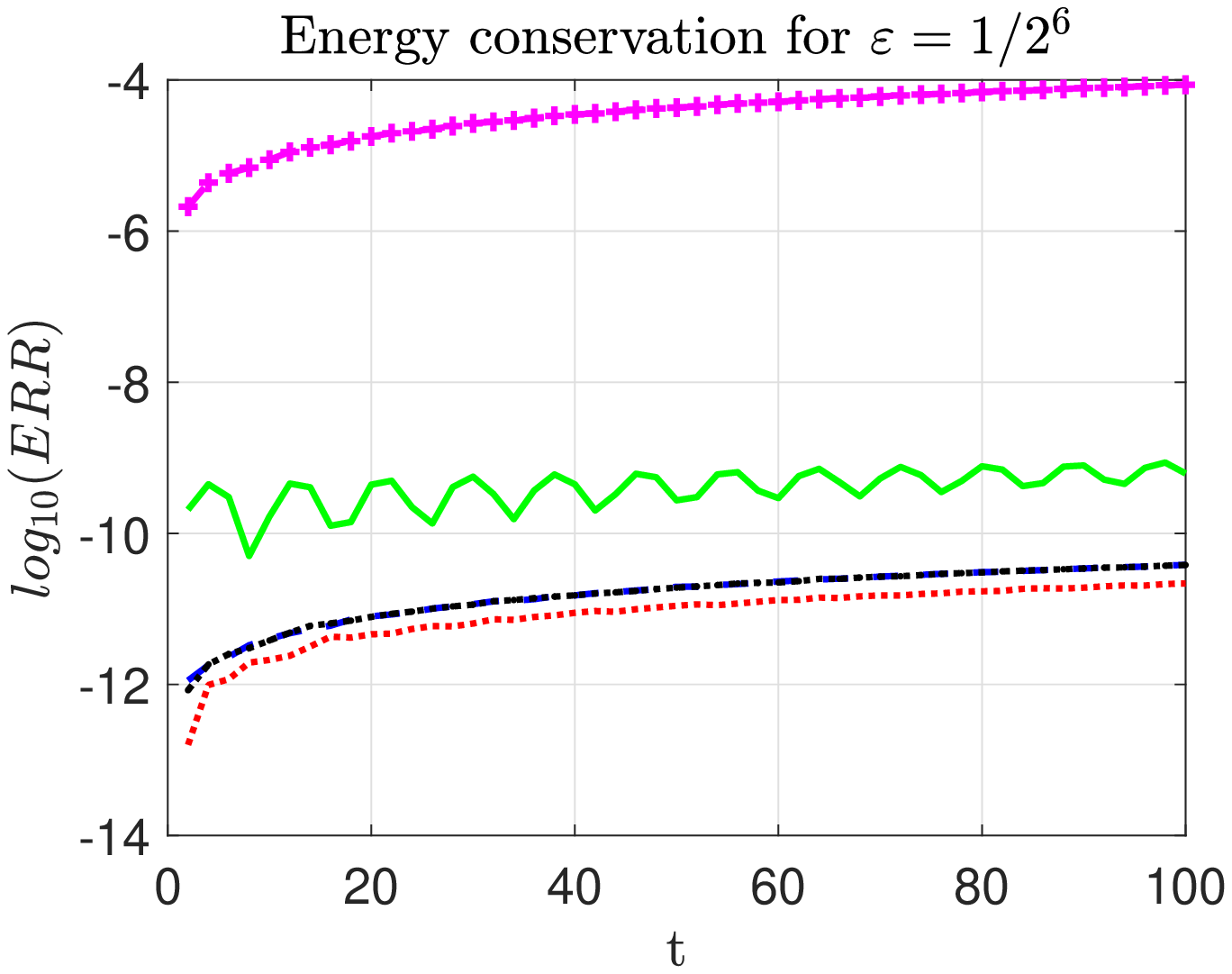}
\includegraphics[width=6.5cm,height=5.0cm]{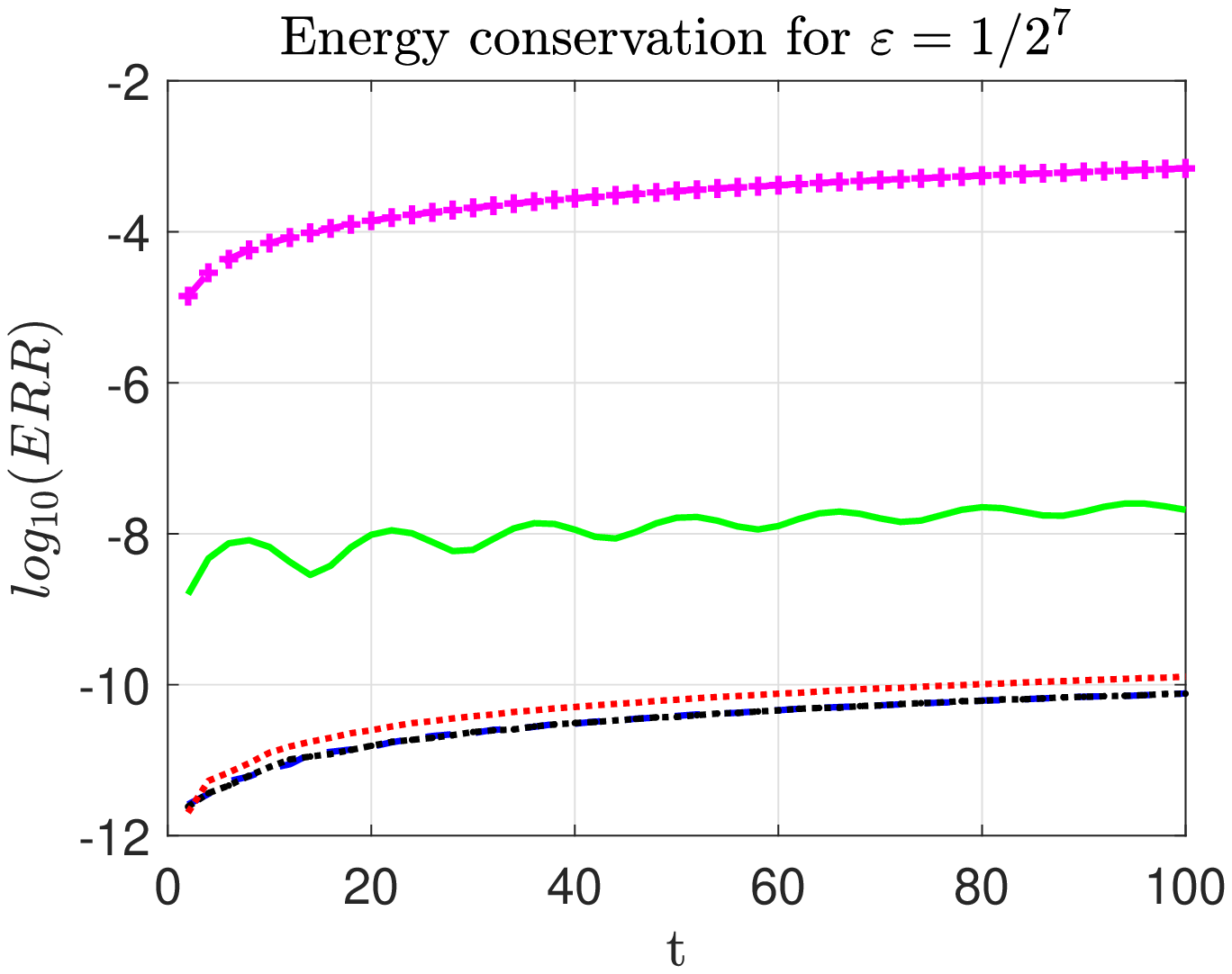}
\includegraphics[width=6.5cm,height=5.0cm]{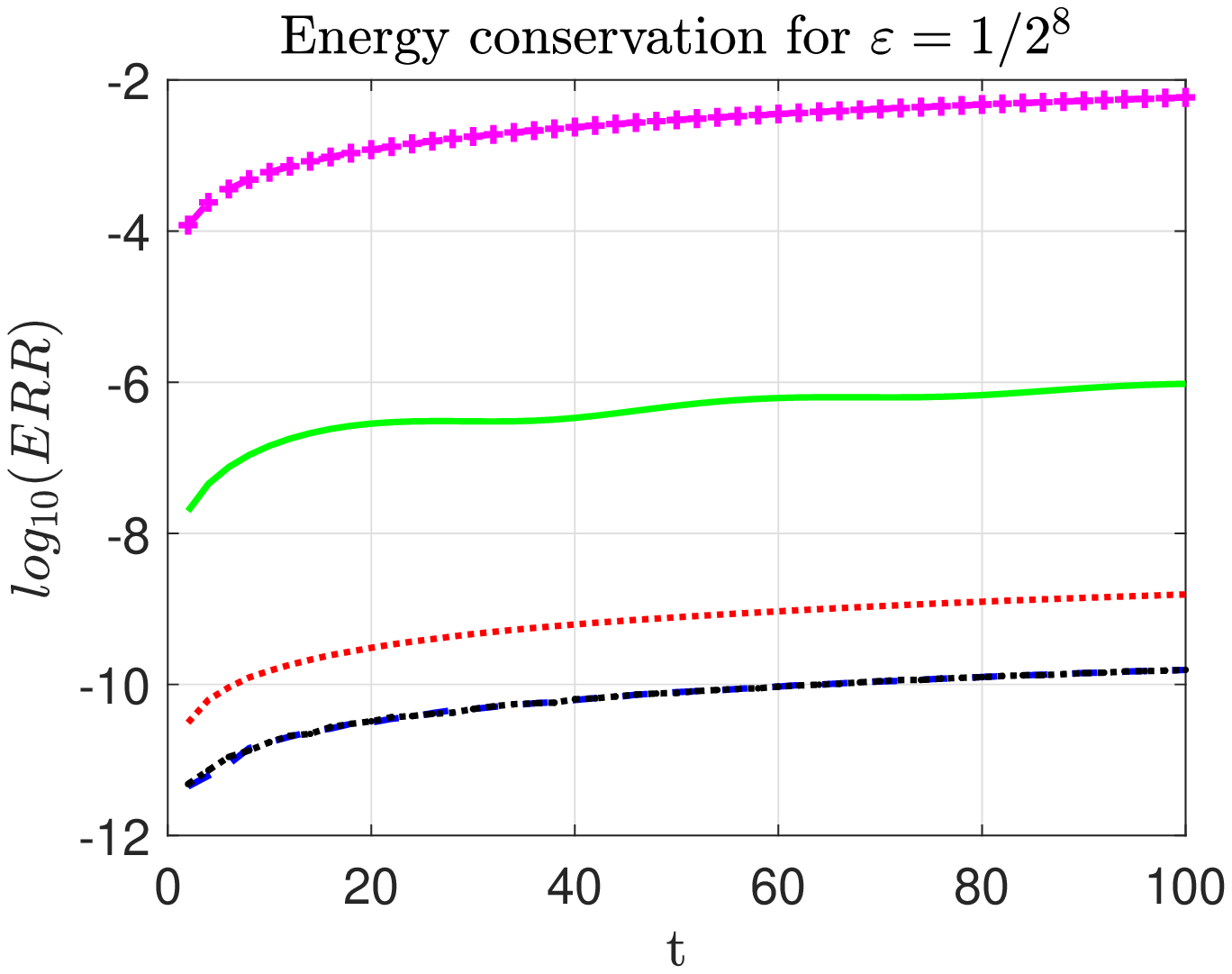}
\caption{The relative error of  discrete  energy (EER)  against $t$.} \label{fig1}
\end{figure}

\textbf{Convergence.} Following  \cite{Chartier16}, $u^0(x)$ is chosen as $u^0(x)=\cos(x)+\sin(x)$ and the boundary condition is $u(t,0)=u(t,2\pi)$.
The long term NSE \eqref{lon sch system} is solved in $[0,T/\varepsilon]$ with $T=1$ and
 $h= 1/2^{i}$ for  $i=1,\ldots,6$. The global errors of our methods measured in $L^2$ and $H^1$ for different $\varepsilon$ are presented in Figure \ref{fig0}.  For comparison, the errors of EEI are also displayed in  Figure \ref{fig0}.
It follows that EP1 only has the global error $\mathcal{O}(h^2)$ while EP2 has the error bound $\mathcal{O}(\varepsilon h^2)$ and EP3 shows $\mathcal{O}(\varepsilon h^3)$. This agrees with the results of Theorem  \ref{thm ful convergence}.  It seems here that EP3 has a better convergence than $\mathcal{O}(\varepsilon h^3)$. But after presenting the errors for   $\varepsilon=1$ in  Figure \ref{fig0-0}, it can be observed that EP3 still shows a third-order convergence.

\textbf{Near-conservations in other aspects.}
In order to show the near-conservations in other aspects, small initial value is required. Following  \cite{12,19}, we change the initial value into $u^0(x)=0.1\big(\frac{x}{\pi}-1\big)^3\big(\frac{x}{\pi}+1\big)^2+\ii
\times0.1\big(\frac{x}{\pi}-1\big)^3\big(\frac{x}{\pi}+1\big)^3$ and
consider the periodic boundary condition $u(t,-\pi)=u(t,\pi)$.
The problem is solved on $[0,10000]$ with $h=\frac{1}{100}$ and the relative errors of density and momentum
are shown in Figures \ref{fig21}-\ref{fig31}, respectively
\footnote{The methods show similar conservation of actions and we omit the corresponding numerical results for brevity.}
.  It can
be observed clearly from these results that the density and momentum
  are conserved  well by EP1-EP2 but not by EP3 over long terms, which supports the results stated in Theorem
\ref{main theo}. 

%

Based on the numerical results, we can draw the following observations.

1) The energy-preserving methods EP1-EP3 preserve the energy with a very good accuracy for both regimes of $\varepsilon$, which is much better than the existed exponential integrators EEI and EEI4 (see Figure \ref{fig1}).  

2) For the highly oscillatory regime, the  integrators EP2-EP3 show improved error bounds while EP1 and EEI do not have the optimal convergence (see Figure \ref{fig0}). For the regime $\varepsilon=1$, EP1-EP3 show the normal global errors (see Figure \ref{fig0-0}).

3) The  integrators EP1-EP2 have  the long term near conservations in the density, momentum
and action  but the methods EP3, EEI and EEI4 do not show such long time behaviour (see Figures \ref{fig21}-\ref{fig31}).

\begin{figure}[ptb]
\centering
\includegraphics[width=6.5cm,height=5.0cm]{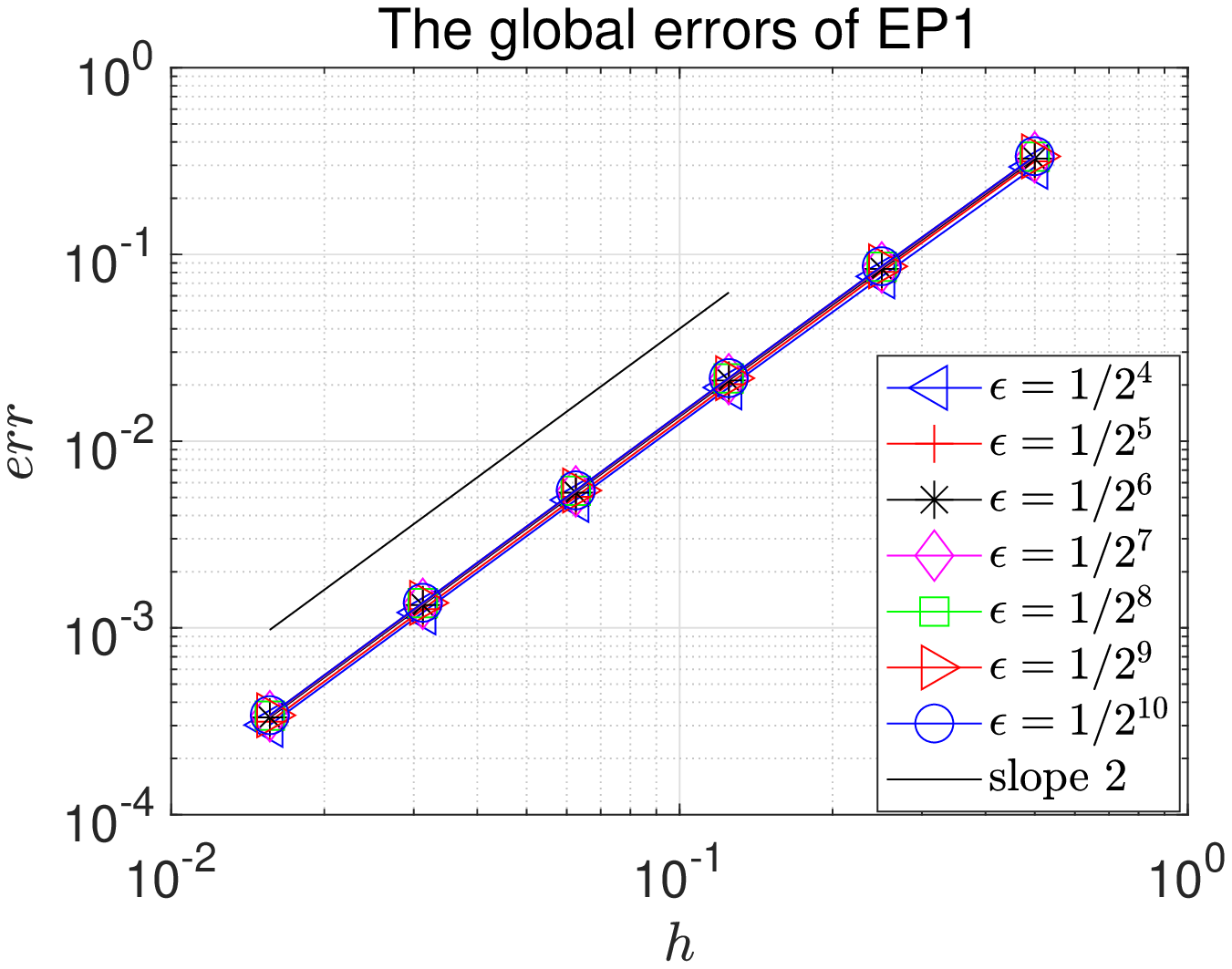}
\includegraphics[width=6.5cm,height=5.0cm]{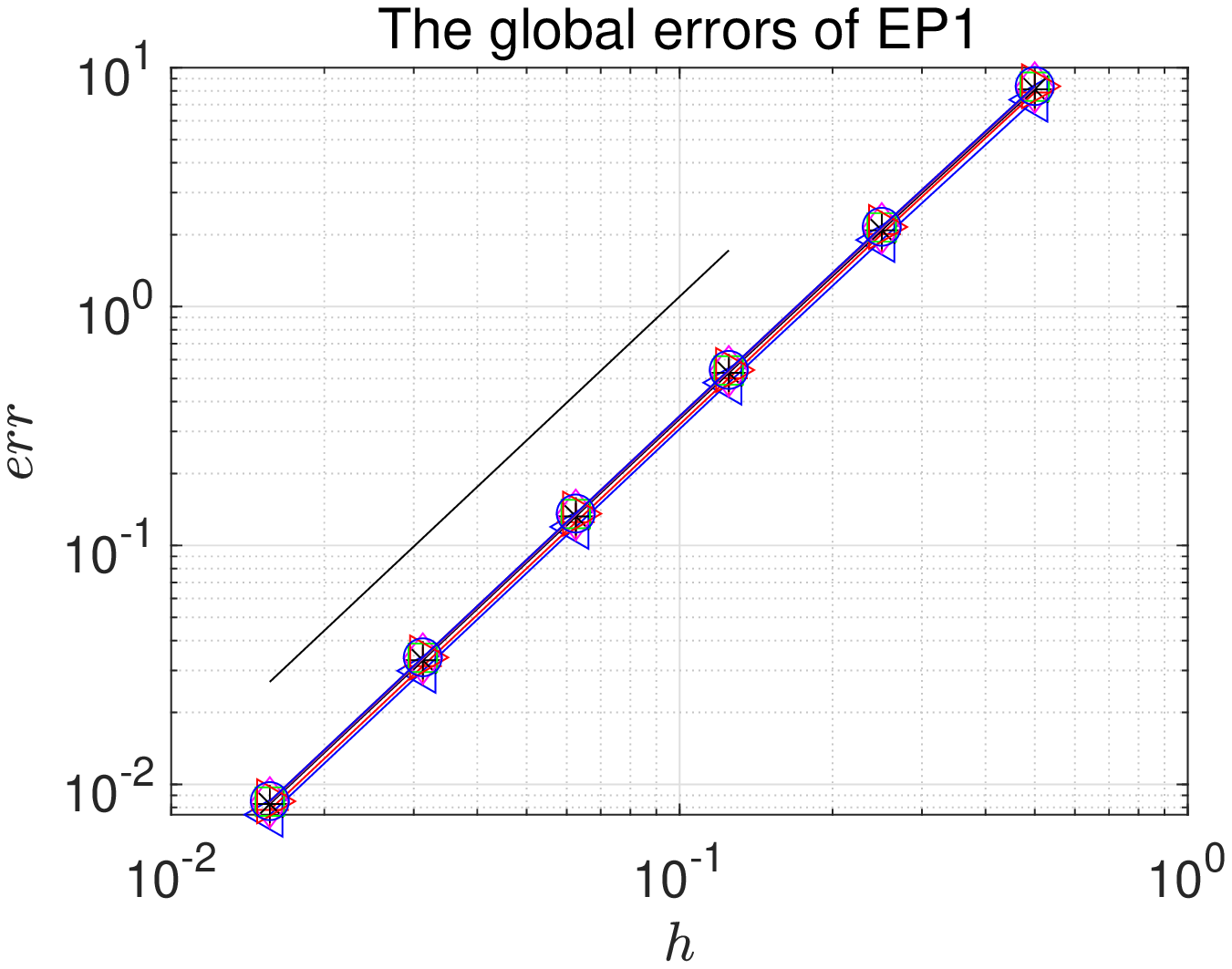}
\includegraphics[width=6.5cm,height=5.0cm]{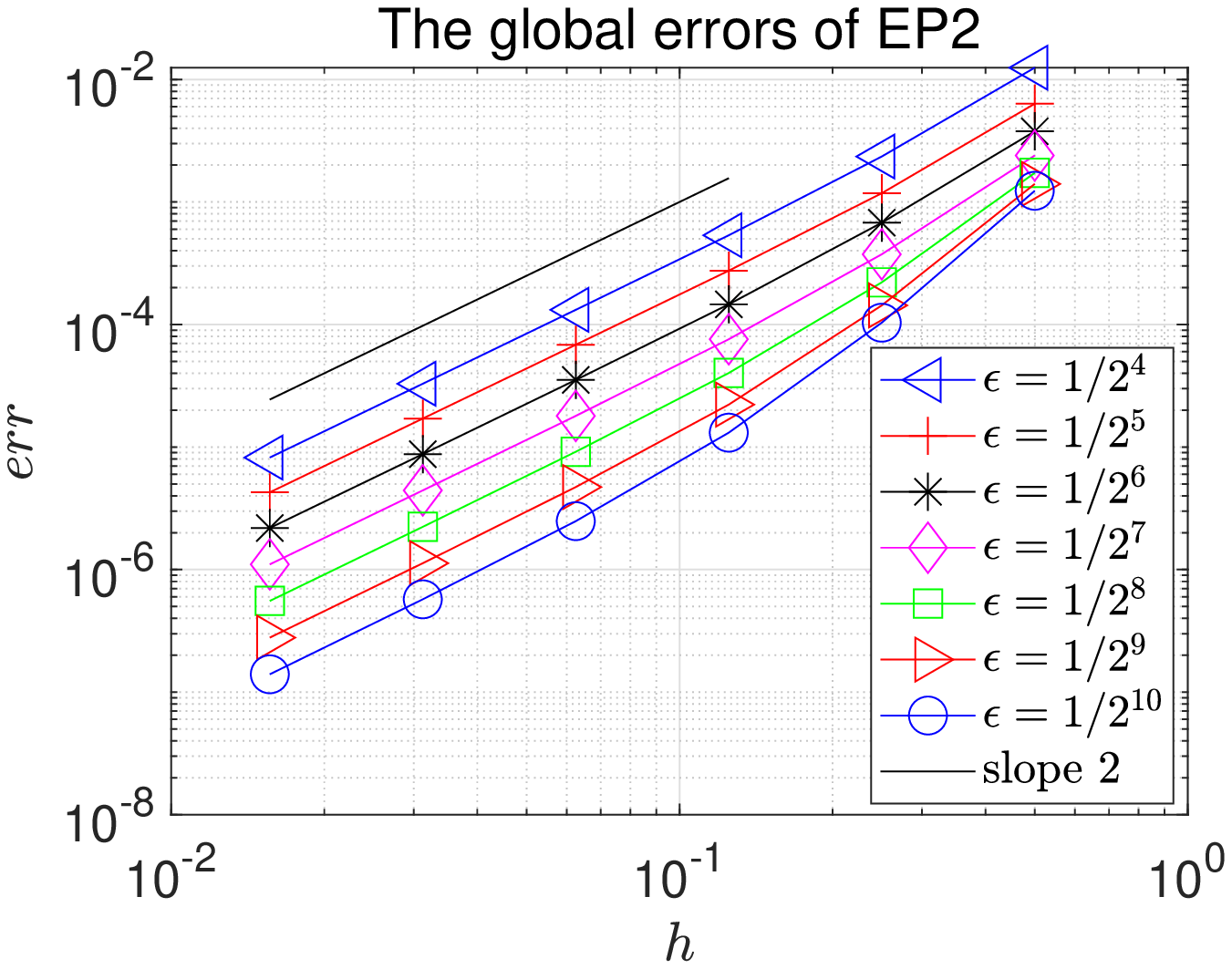}
\includegraphics[width=6.5cm,height=5.0cm]{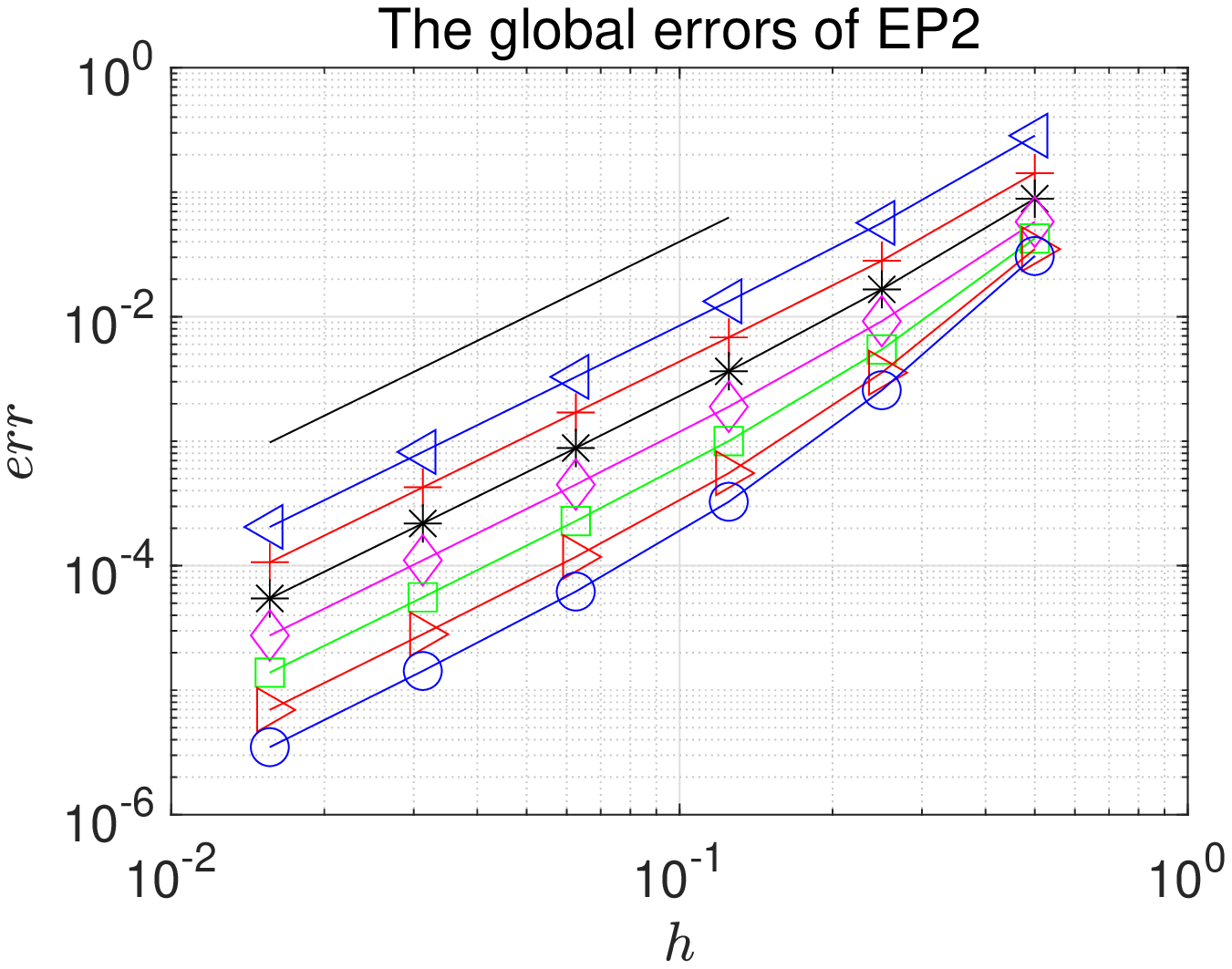}
\includegraphics[width=6.5cm,height=5.0cm]{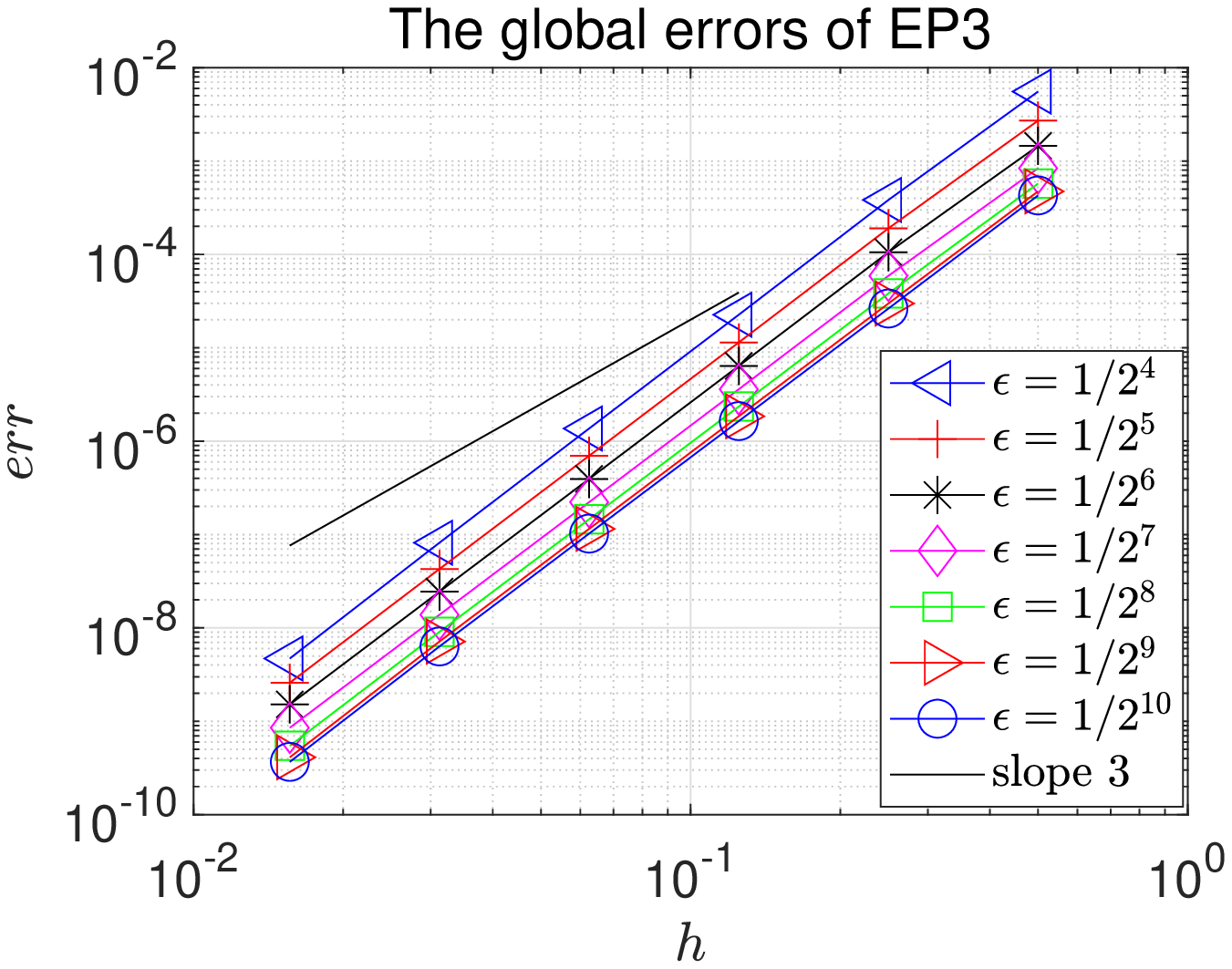}
\includegraphics[width=6.5cm,height=5.0cm]{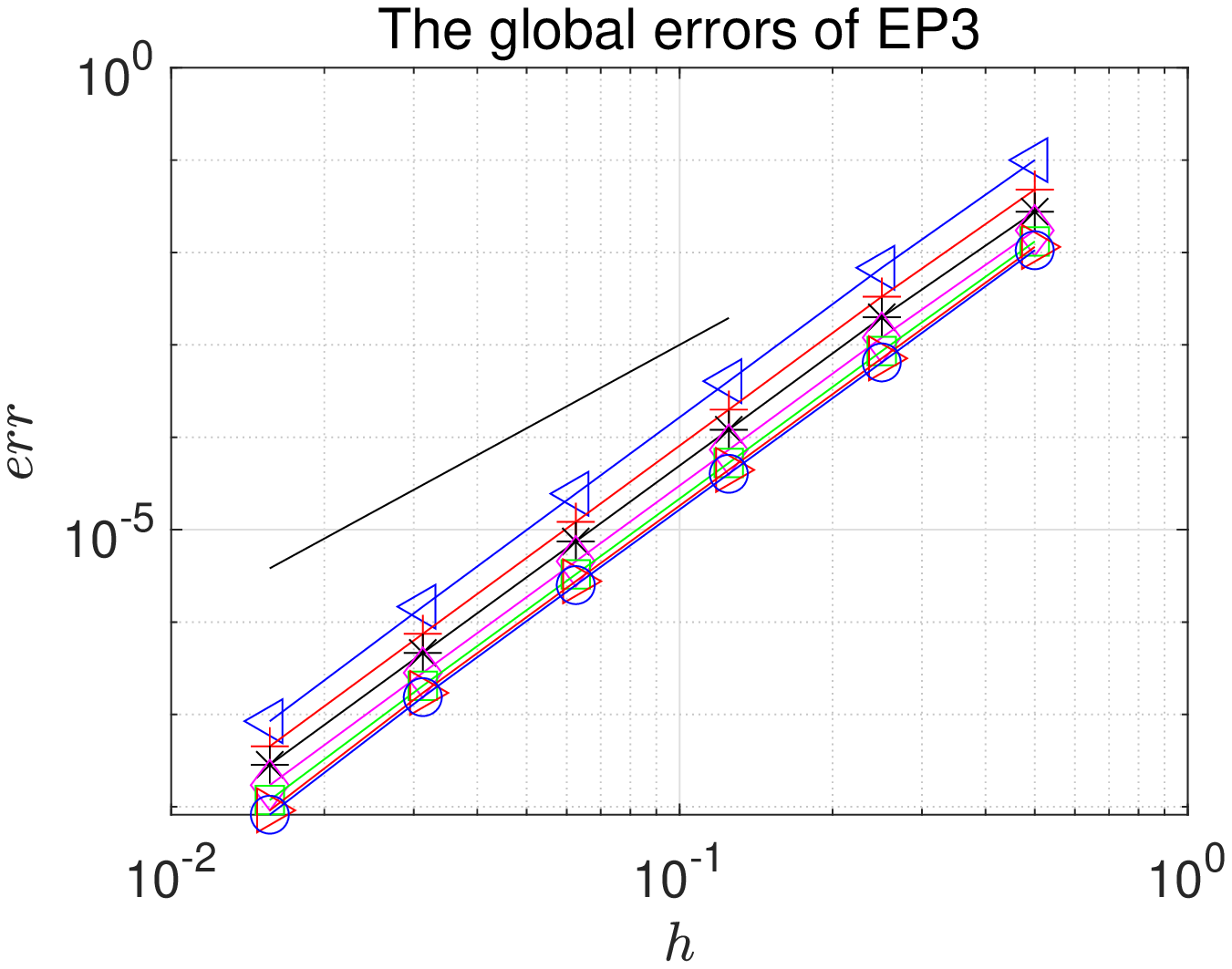}
\includegraphics[width=6.5cm,height=5.0cm]{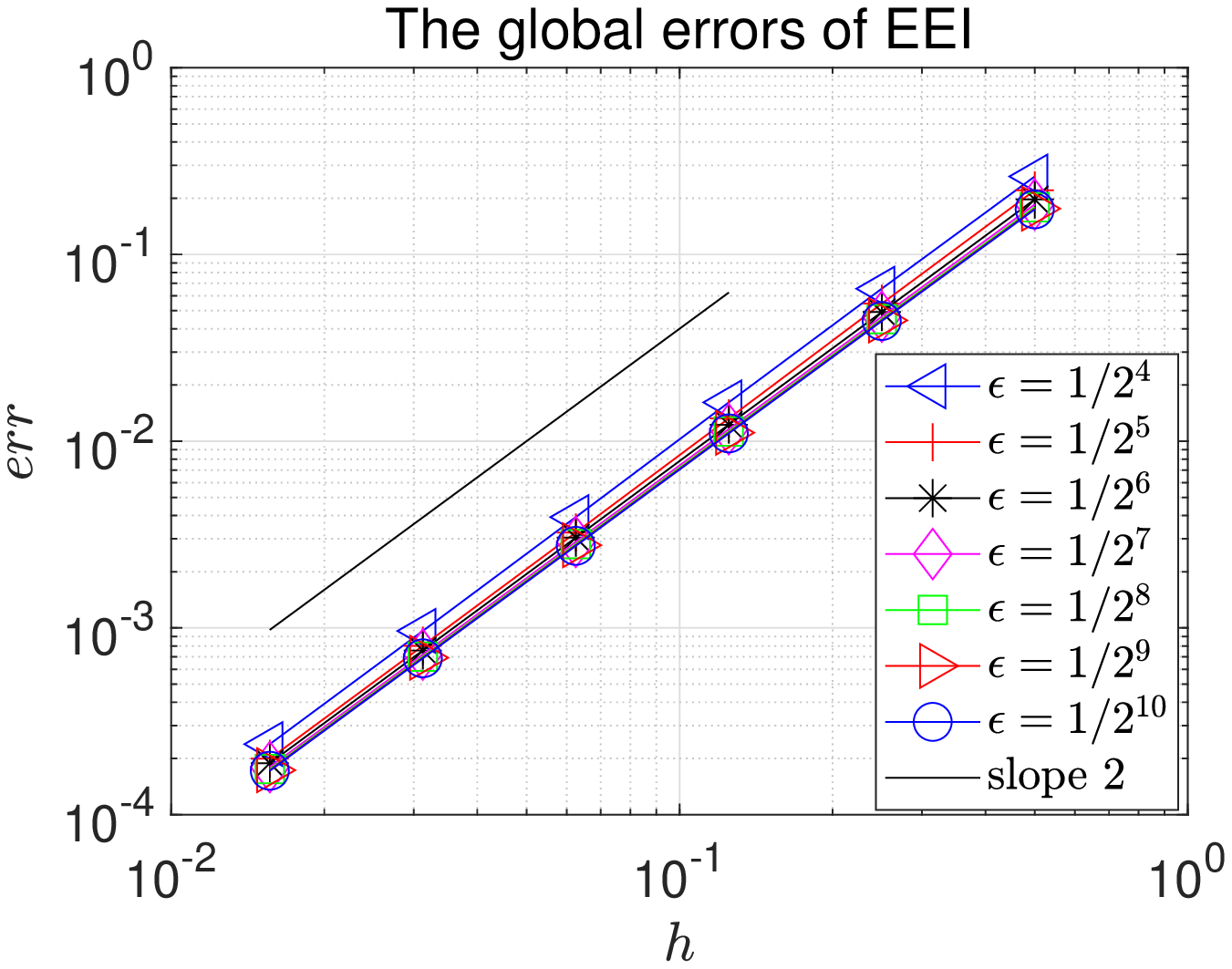}
\includegraphics[width=6.5cm,height=5.0cm]{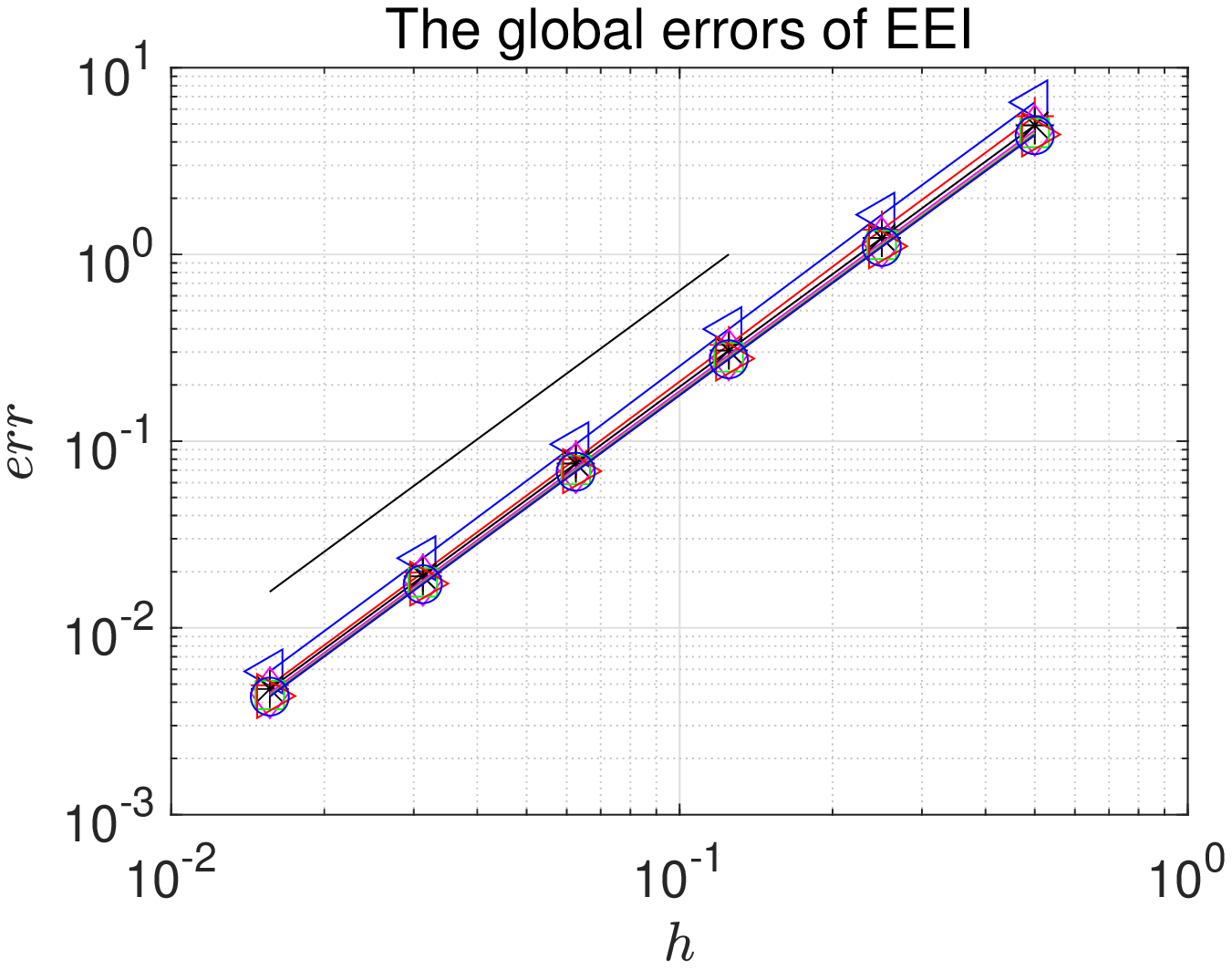}
\caption{The  global error (err) measured in $L^2$ (left) and $H^1$ (right) against the stepsize. } \label{fig0}
\end{figure}
\begin{figure}[ptb]
\centering
\includegraphics[width=6.5cm,height=5.0cm]{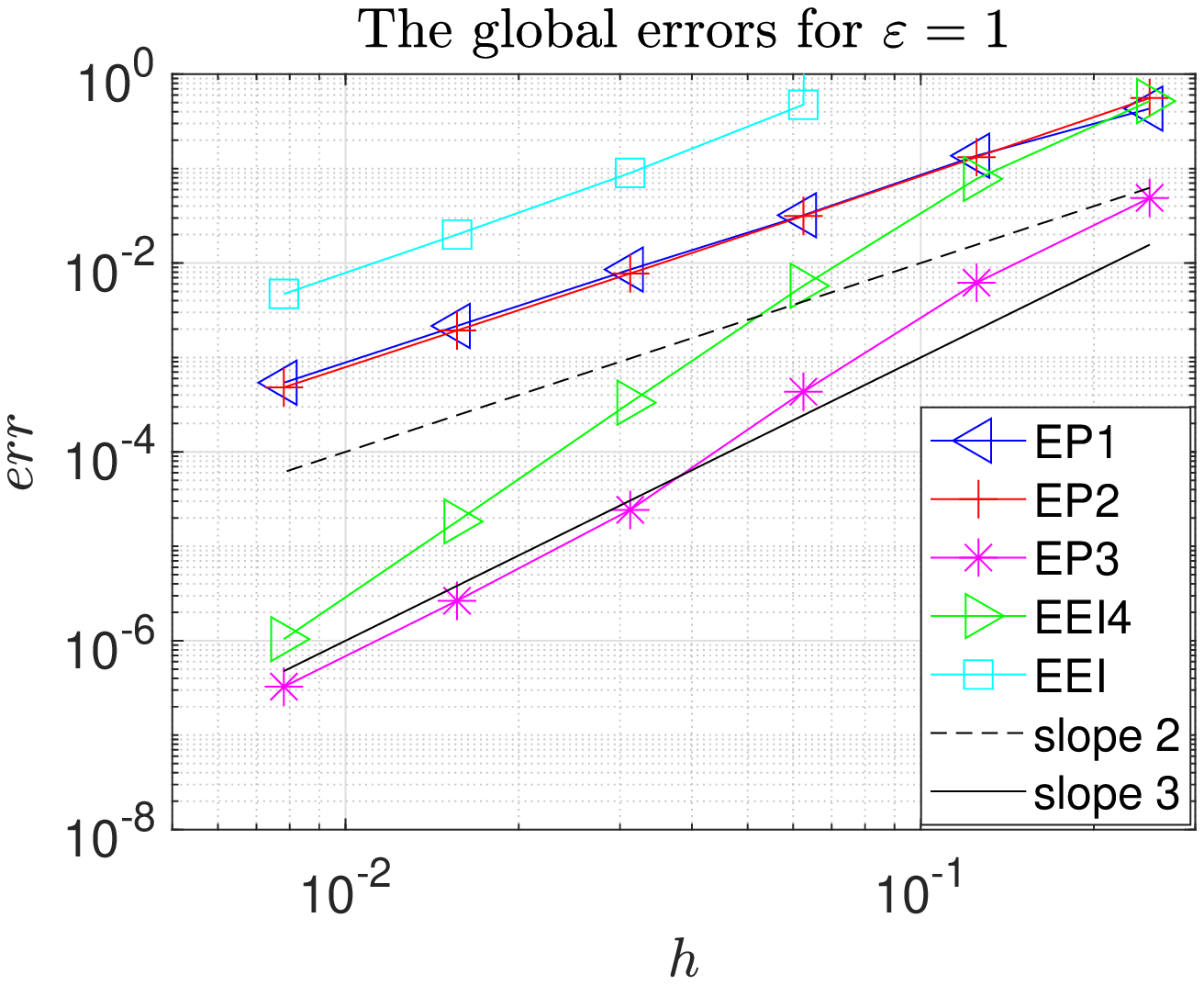}
\includegraphics[width=6.5cm,height=5.0cm]{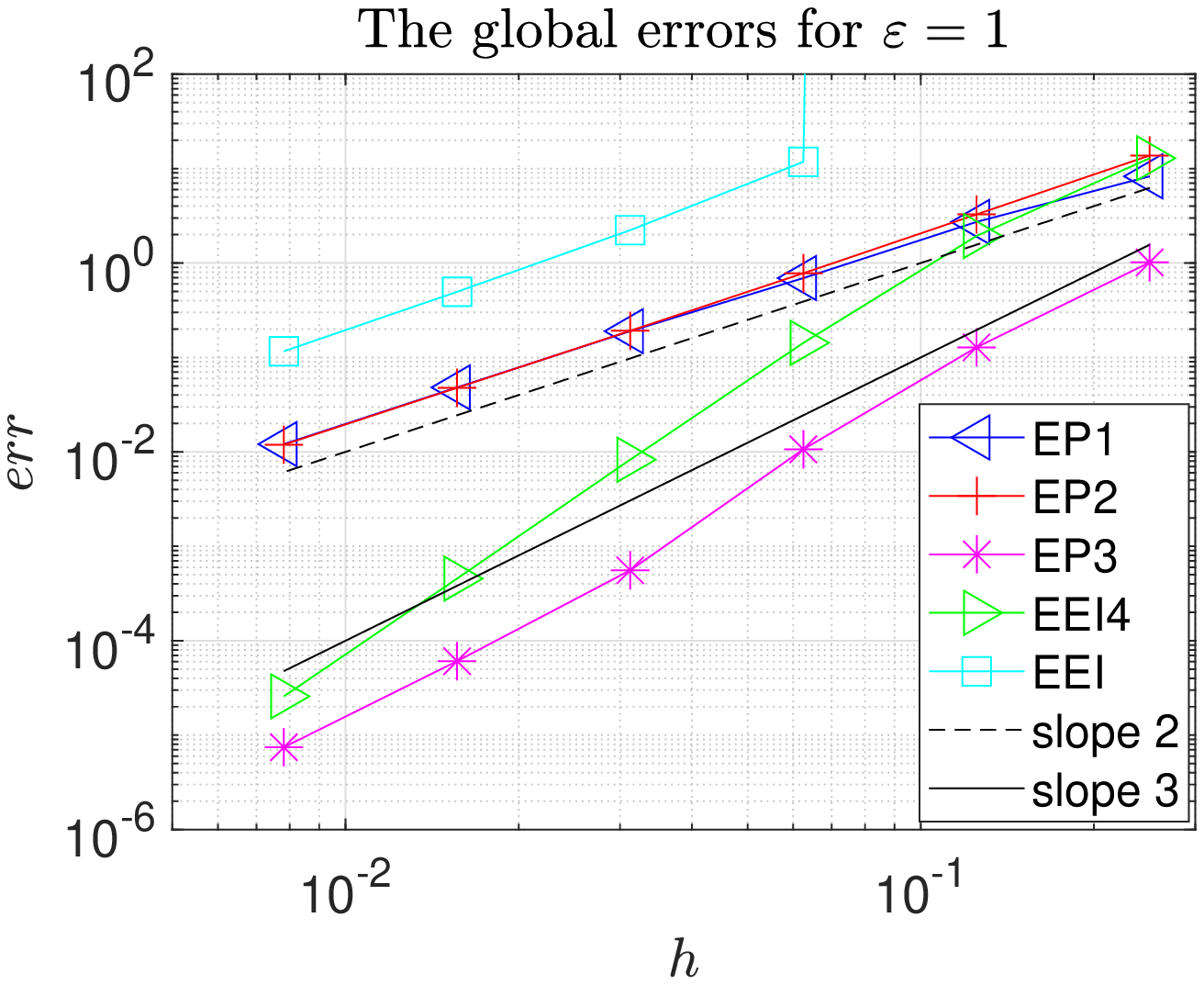}
\caption{The  global error (err) measured in $L^2$ (left) and $H^1$ (right) against the stepsize. } \label{fig0-0}
\end{figure}\begin{figure}[ptb]
\centering
\includegraphics[width=6.5cm,height=5.0cm]{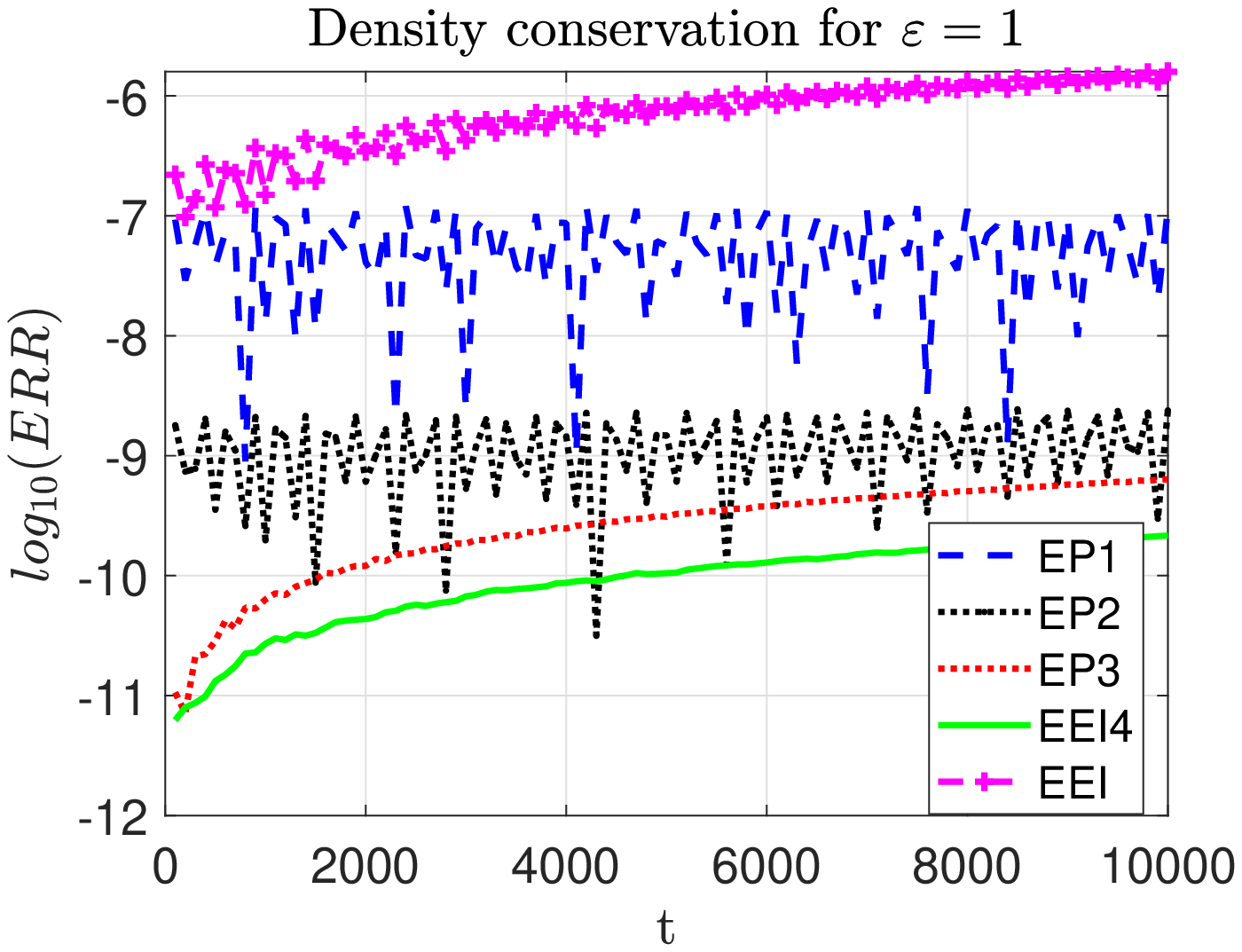}
\includegraphics[width=6.5cm,height=5.0cm]{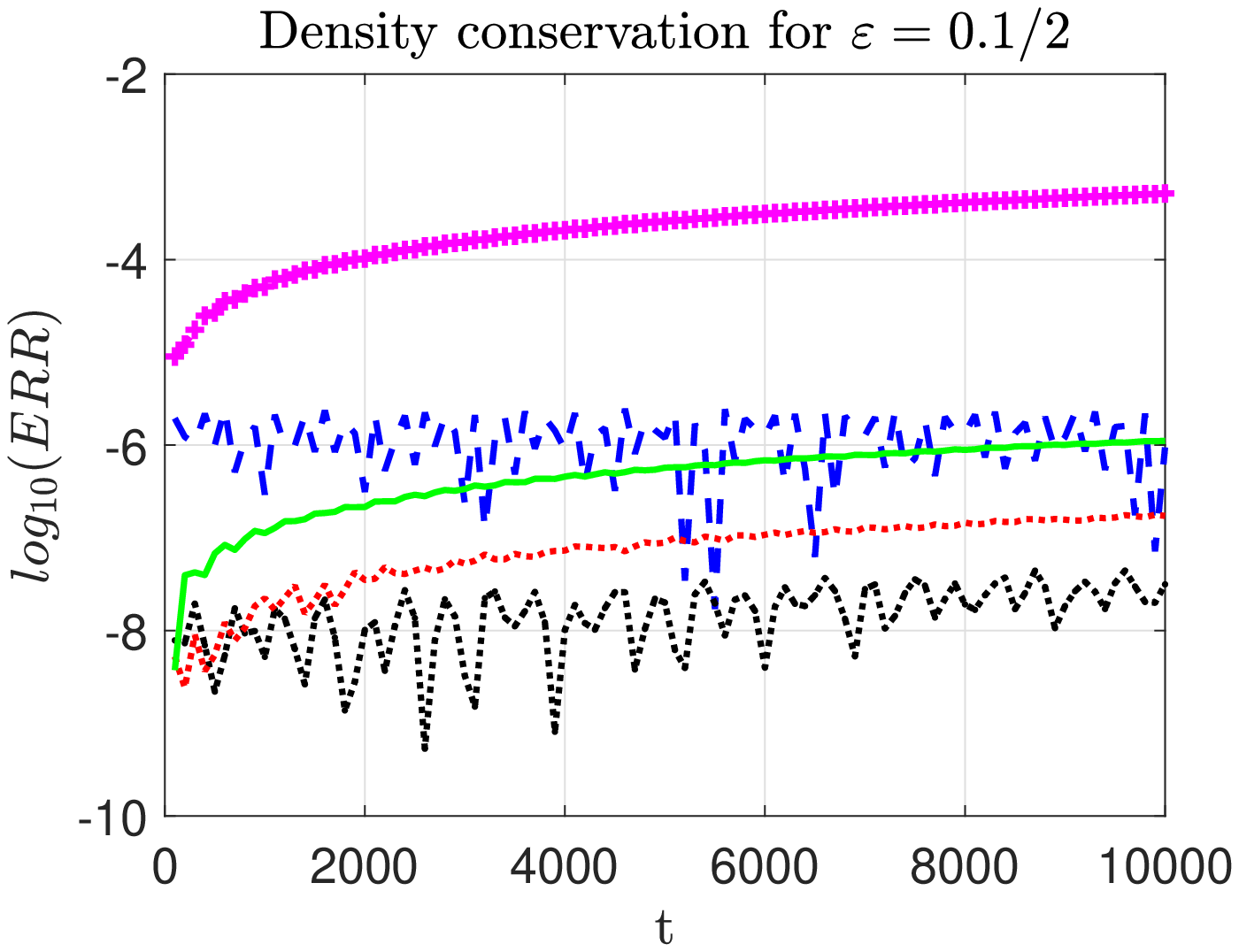}
\includegraphics[width=6.5cm,height=5.0cm]{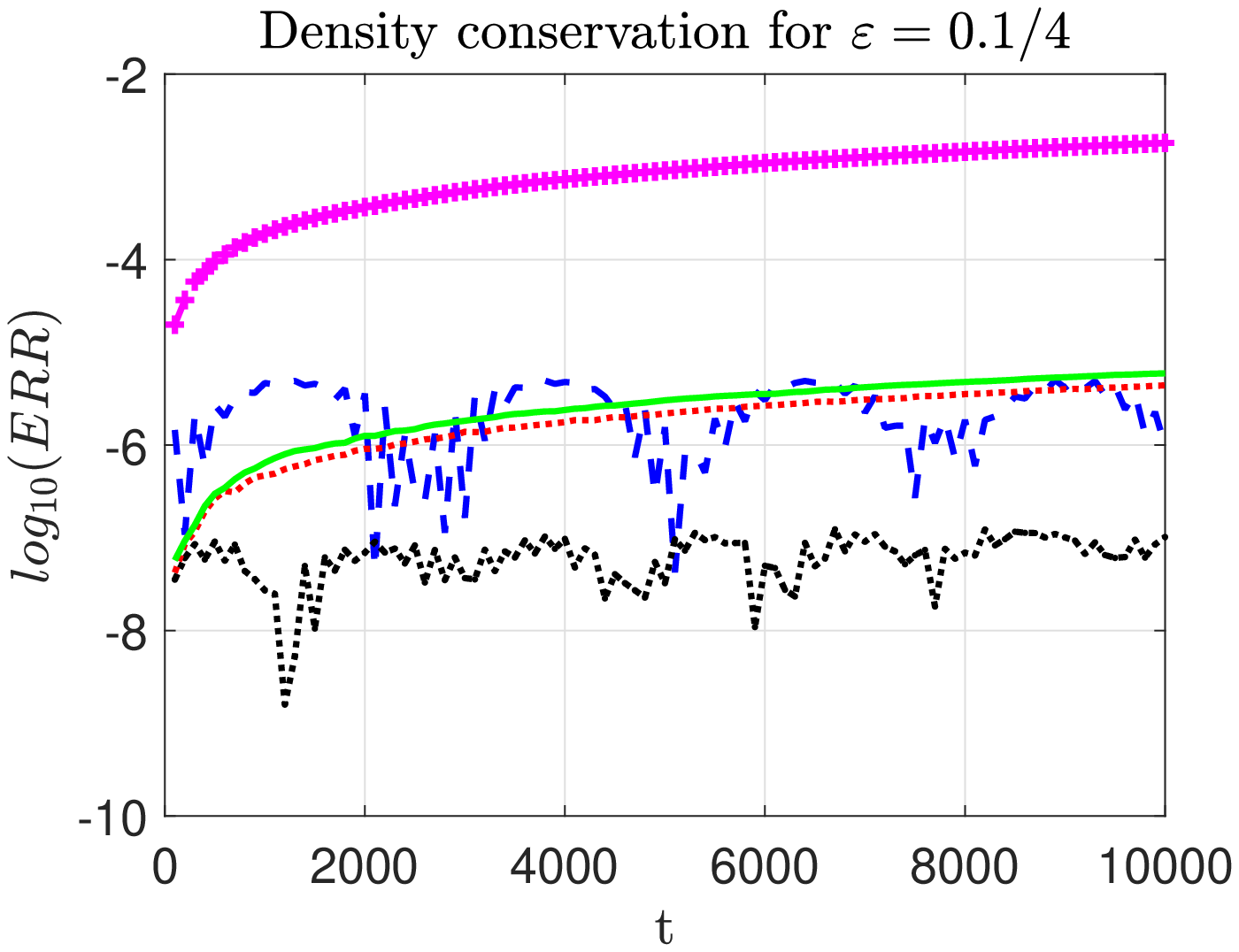}
\includegraphics[width=6.5cm,height=5.0cm]{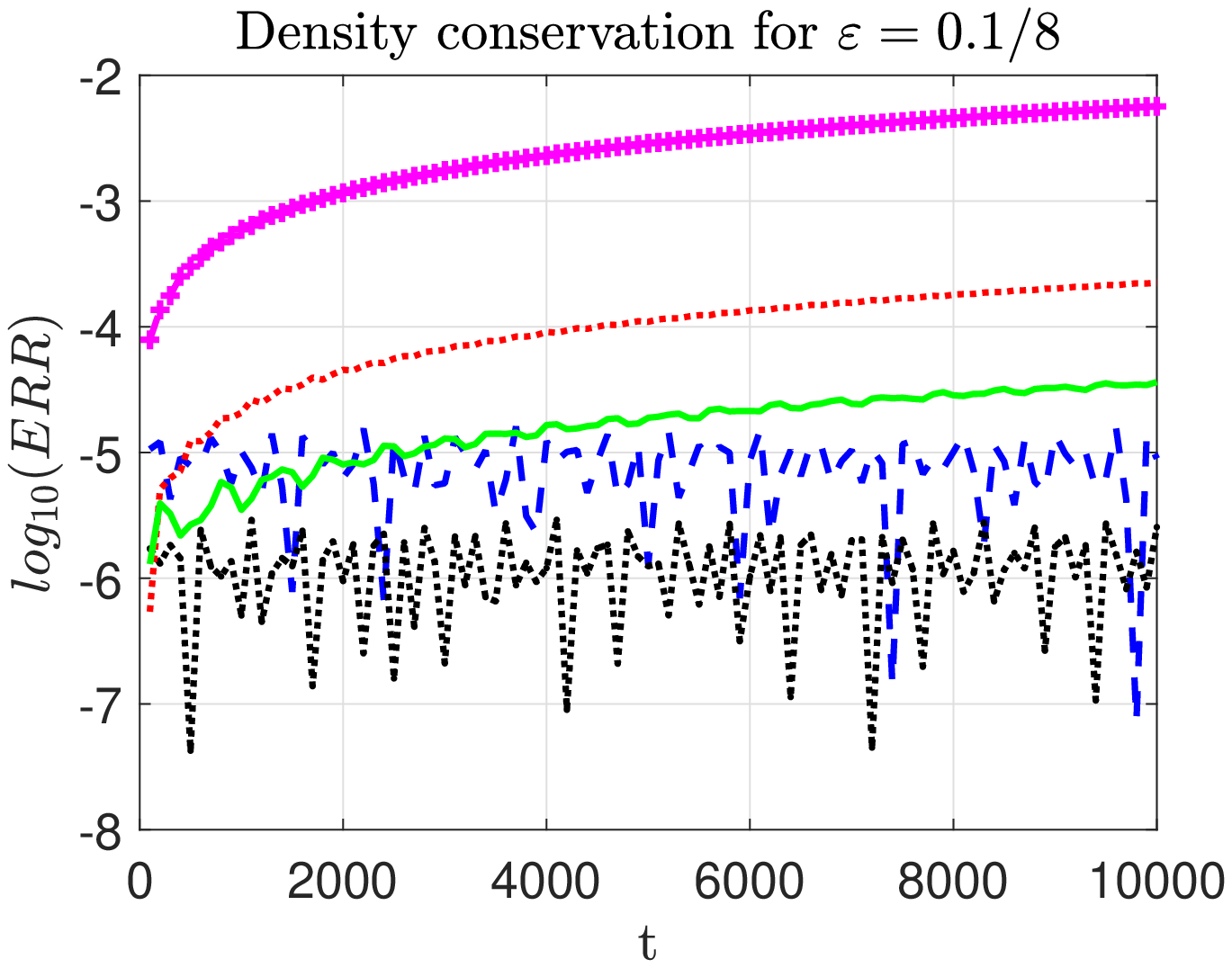}
\caption{The relative error of density against
$t$.}\label{fig21}
\end{figure}
 \begin{figure}[ptb]
\centering
\includegraphics[width=6.5cm,height=5.0cm]{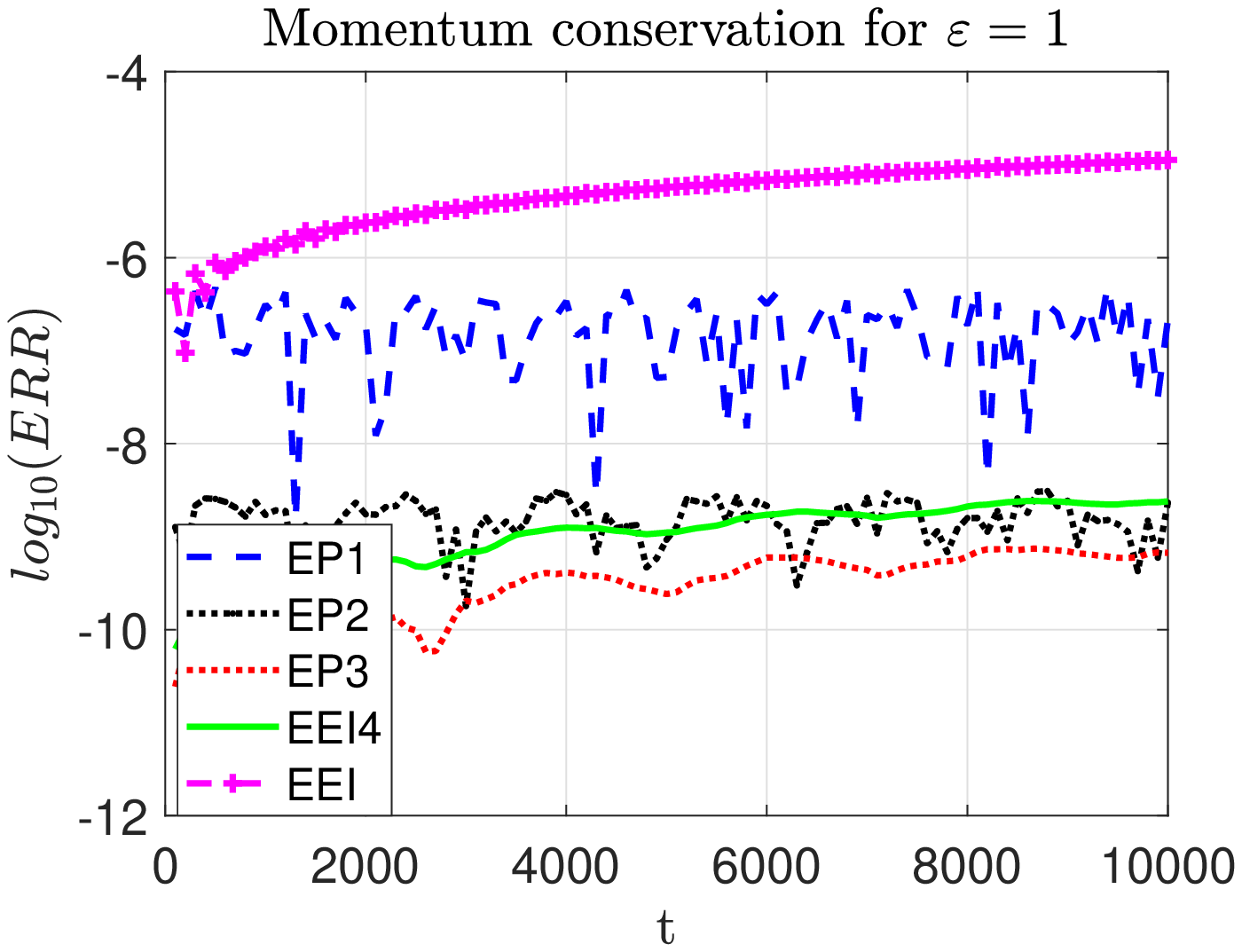}
\includegraphics[width=6.5cm,height=5.0cm]{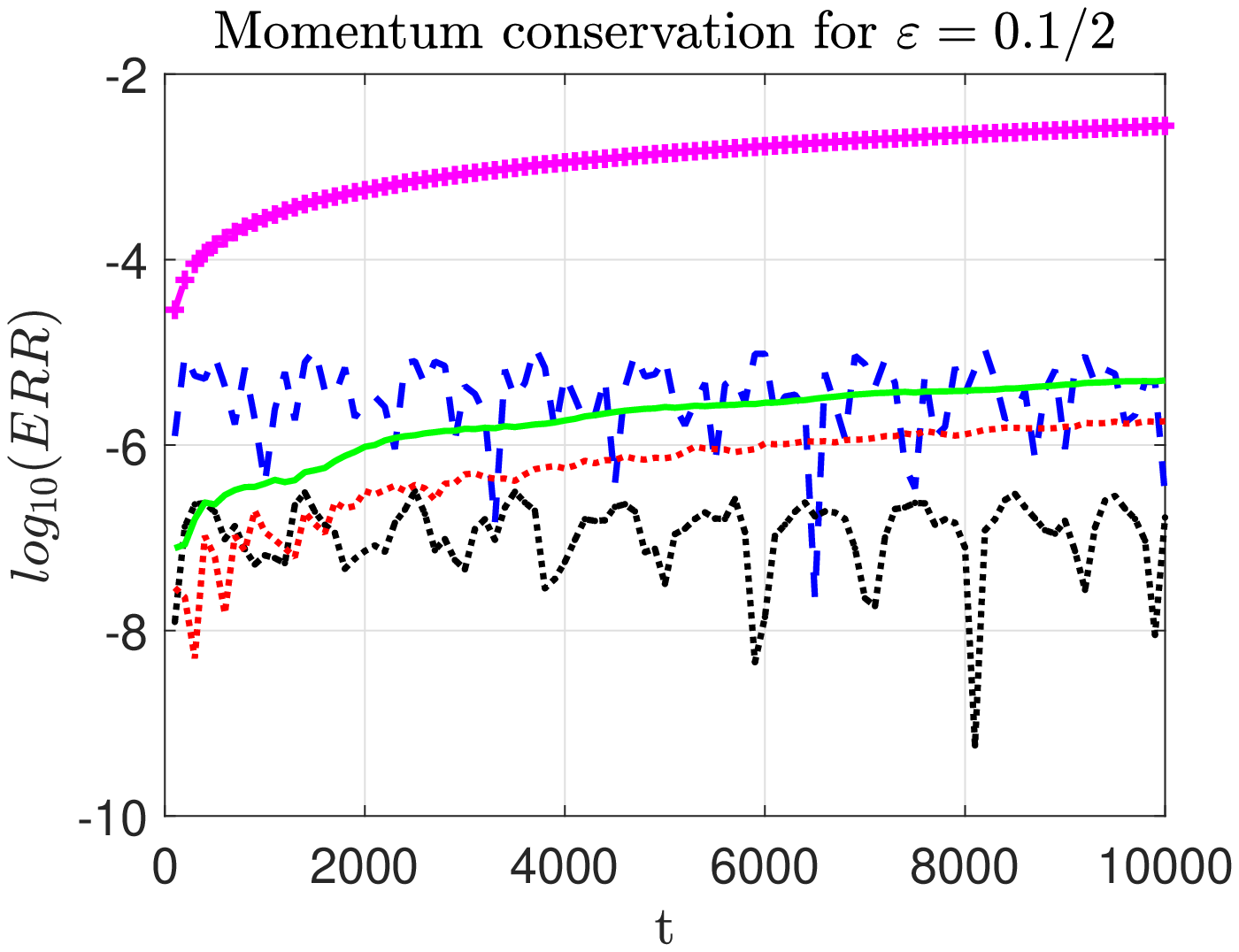}
\includegraphics[width=6.5cm,height=5.0cm]{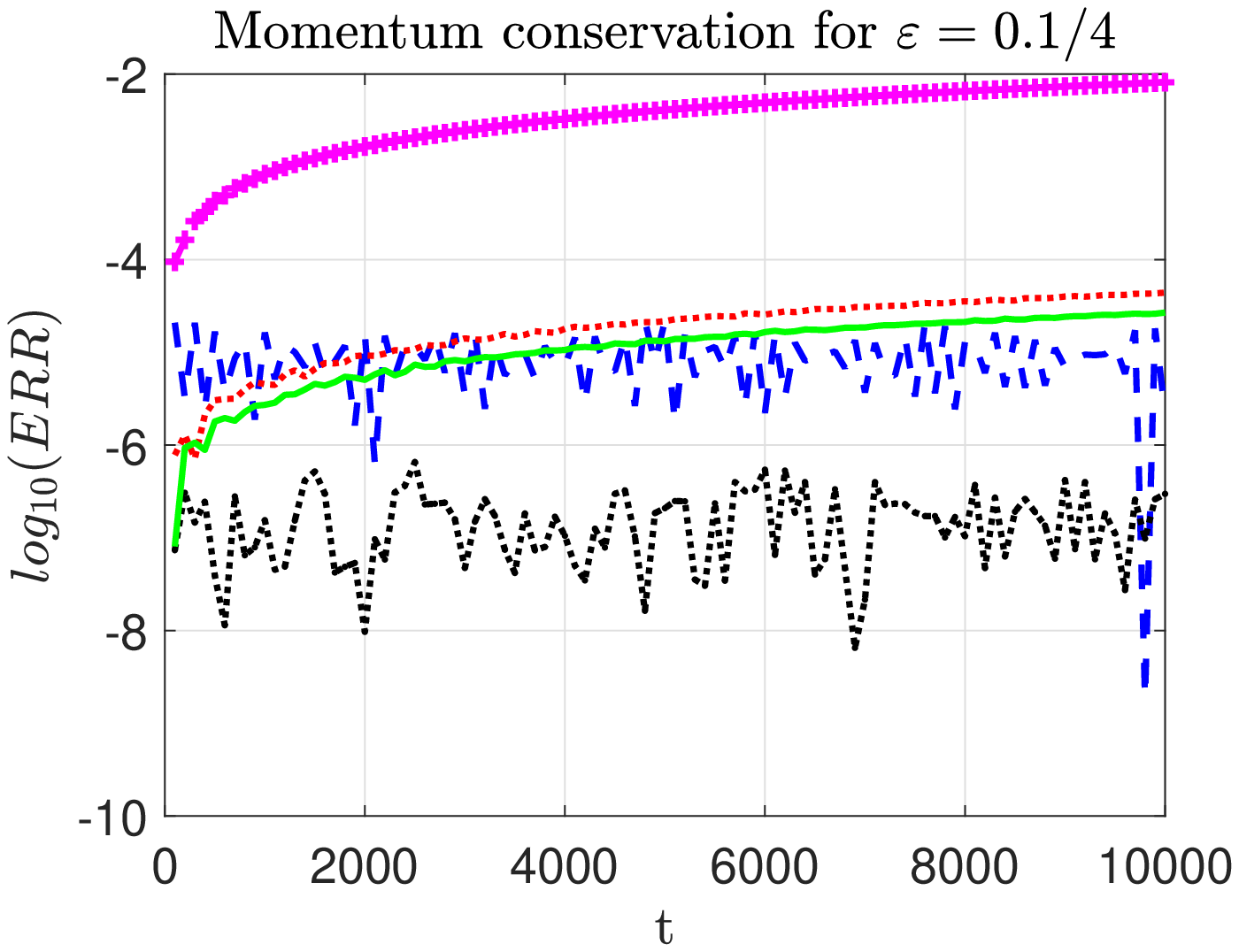}
\includegraphics[width=6.5cm,height=5.0cm]{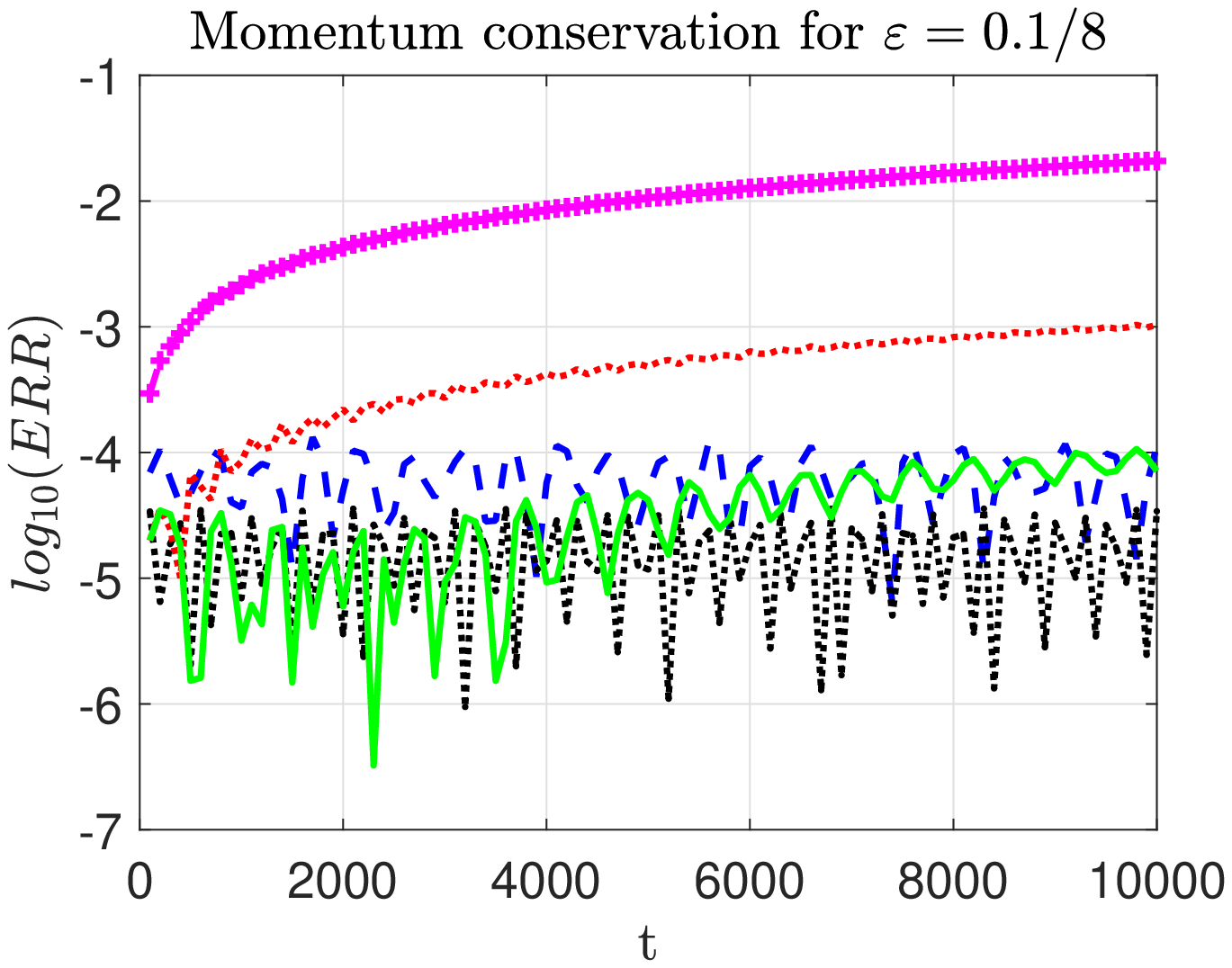}
\caption{The relative error of momentum against
$t$.}\label{fig31}
\end{figure}

{ \section{Applications and future issues}

This is a preliminary  research on the long-time behaviour of
 energy-preserving exponential integrators  and it is noted that
the algorithms  can be  extended to the numerical
solutions of the following equations (see Table \ref{appERKN}) by
replacing $\ii \mathcal{A}$ and $f$ in \eqref{sch ode} with the new ones.
\begin{table}
\newcommand{\tabincell}[2]{\begin{tabular}{@{}#1@{}}#2\end{tabular}}  
  \centering
  \begin{tabular}{|c|c|c|}
  \hline
     \tabincell{c}{\text{Systems}}   & \text{Replace\ $\ii
\mathcal{A}$ by} & \text{New\ $f$} \\ 
\hline
   \tabincell{c}{\text{Hamiltonian system with} \\  $H(q,p)=\frac{1}{2}p^{\intercal}p+
\frac{1}{2} q^{\intercal}\Omega q+U (q)$}   & \text{$\left(
    \begin{array}{cc}
      0 & I \\
      -\Omega  & 0 \\
    \end{array}
  \right)$}    &\text{$\left(
                           \begin{array}{c}
                             0 \\
                          -\nabla U(q) \\
                           \end{array}
                         \right)$}\\ 
\hline
  \tabincell{c}{Wave equation\\ $u_{tt}-a^2\Delta u=g(u)$ } & $\left(
    \begin{array}{cc}
      0 & I \\
      -a^2\Delta  & 0 \\
    \end{array}
  \right)$ & $\left(
                           \begin{array}{c}
                             0 \\
                          g(u) \\
                           \end{array}
                         \right)$ \\
\hline
  \tabincell{c}{Damped Helmholtz-Duffing
oscillator\\ $q''+2\upsilon q'=-Aq-Bq^2-\varepsilon q^3$ } &$\left(
    \begin{array}{cc}
      0 & I \\
      0  & 2\upsilon \\
    \end{array}
  \right)$ & $\left(
                           \begin{array}{c}
                             0 \\
                          -Aq-Bq^2-\varepsilon q^3 \\
                           \end{array}
                         \right)$ \\
\hline
  \tabincell{c}{Charged-particle dynamics \\ in
a   constant magnetic field\\ $x''=\tilde{B} x'    +F(x)$} & $\left(
    \begin{array}{cc}
      0 & I \\
      0  & \tilde{B} \\
    \end{array}
  \right)$ & $\left(
                           \begin{array}{c}
                             0 \\
                          F(x) \\
                           \end{array}
                         \right)$ \\
\hline
  \tabincell{c}{First-order ODEs\\ $x'=\frac{1}{\varepsilon}A x    +f(x)$ } & $\frac{1}{\varepsilon}A$& $f(x)$ \\
\hline
\end{tabular}
  \caption{Some systems which the presented methods can be applied.} \label{appERKN}
\end{table}

 We also note that there are some issues which
can be further considered.
\begin{itemize}\itemsep=-0.2mm

\item  
The extensions of  the methods as well as their analysis in this
paper to the logarithmic Schr\"{o}dinger equation (\cite{Bao19}) and  time-dependent Schr\"{o}dinger equation
in semiclassical scaling (\cite{Lubich2020}) will   be researched
in future.

\item    The    long term analysis of other kinds of energy-preserving integrators
in other PDEs such as  Vlasov-Poisson system (\cite{Sonnend15,Sonnend14}) and Maxwell equations   will also be considered.


\item Another issue for future exploration is the analysis of
parareal algorithms  of Schr\"{o}dinger equations.

\end{itemize}}

\section*{Appendix}
\subsection*{Appendix I. Proof of Proposition
\ref{pro: lem1}}

$\bullet$ Proof of the first result.

For the case that $(j, k)\in\mathcal{L}_{\tilde{\epsilon},M,h}$, we
have $\abs{ 2\ii \sin\big(\frac{1}{2}h(\omega_j-k
\cdot\omega)\big)}>\tilde{\epsilon}^{\frac{1}{2}}h$. Then from this it follows
that $ \abs{\frac{\sinc\big(\frac{1}{2}h\omega_j\big)}{2\ii
\sin\big(\frac{1}{2}h(\omega_j-k
\cdot\omega)\big)}}\leq\frac{1}{\tilde{\epsilon}^{\frac{1}{2}}h},$ which yields
$|||\WW^{-1}\vv|||_s\leq \tilde{\epsilon}^{-\frac{1}{2}}h^{-1}|||\vv|||_s.$
For  other $(j, k)$, the statement is obtained by
considering the definition of $\WW$.

$\bullet$ Proof of the second result. 

 Taking into account  the result proved in  \cite{18} $[[k^1]] +
[[k^2 ]] + [[k^3]]\geq \max([[k^1]],2)+1,$   one has
\begin{equation*}
\begin{aligned}
|||\FF(\vv)|||^2_s=&\sum\limits_{j}\abs{\omega_j}^s\Big(\sum\limits_{k}\tilde{\epsilon}
^{-\max([[k]],2)}\\
&\abs{\ii h \sum\limits_{k^1+k^2-k^3=k
}\int_{0}^{1}\tilde{\epsilon}^{[[k^1]] + [[k^2 ]] +
[[k^3]]}v^{k^1}_{j(k^1)}v^{k^2}_{j(k^2)}\overline{v^{k^3}_{j(k^3)}}d\sigma}\Big)^2\\
\leq &h^2\tilde{\epsilon}^2\sum\limits_{j}\abs{\omega_j}^s
\sum\limits_{k^1,k^2,k^3}\Big(\abs{
\int_{0}^{1}v^{k^1}_{j(k^1)}v^{k^2}_{j(k^2)}\overline{v^{k^3}_{j(k^3)}}d\sigma}\Big)^2\\
= &h^2\tilde{\epsilon}^2\norm{  \sum\limits_{k^1,k^2,k^3}
\int_{0}^{1}v^{k^1}_{j(k^1)}v^{k^2}_{j(k^2)}\overline{v^{k^3}_{j(k^3)}}d\sigma}^2
\leq C h^2\tilde{\epsilon}^2(|||\check{\mathbf{v}}|||^3_s) ^2.
\end{aligned}
\end{equation*}

$\bullet$ Proof of the last result. 

According to  \cite{18}, the following result is true
\begin{equation*}
\begin{aligned}
 &a_1a_2a_3- b_1b_2b_3
 =\sum\limits_{j=1}^3 2^{-j}(a_1+b_1)\cdots(a_{j-1}+b_{j-1})
 (a_{j}+b_{j})(a_{j+1}\cdots a_3+b_{j+1}\cdots b_3).
\end{aligned}
\end{equation*}
Then from this result and by a similar calculation to that for the
second result,  the last statement is arrived at.

  The same calculation is also true for $\hat{\vv}, \hat{\FF}$ and $|||\cdot|||_{\frac{d+1}{2}}$
instead of $\vv, \FF$ and $|||\cdot|||_{s}$, respectively.

\subsection*{Appendix II. Proof of Proposition
\ref{pro: f}}

$\bullet$ Proof of \eqref{bounds f12}.

 In the light of the choice of the initial iteration, we
 have
\begin{equation*}
\begin{aligned}
&|||\big[\mathbf{a}(\tilde{\tau})\big]^{0}|||_s\leq C,\
|||\big[\mathbf{a}^{(n)}(\tilde{\tau})\big]^{0}|||_s=0 \ \ \ \textmd{for}\ \
\ n\geq1,\\
&|||\big[\mathbf{b}^{(n)}(\tilde{\tau})\big]^{0}|||_s=0 \ \ \ \textmd{for}\
\ \ n\geq0.
\end{aligned}
\end{equation*}
From the third equality  of \eqref{variables iteration}, it follows
that 
\begin{equation*}
\begin{aligned}
|||\big[\mathbf{a}(0)\big]^{l+1}|||_s&=\Big(\sum\limits_{j}\abs{\omega_j}^s\abs{\big[\mathbf{a}_j^{\langle
j\rangle}(0)\big]^{l+1}}^2\Big)^{1/2} \leq
\tilde{\epsilon}^{-1}\norm{u(0)}_s+ \tilde{\epsilon}\norm{\big[\bb(0)\big]^{l}}_s.
\end{aligned}
\end{equation*}
According to the first  equality  of \eqref{variables iteration}, we
have
\begin{equation*}
\begin{aligned}
|||\big[\mathbf{b}^{(n)}\big]^{l+1}|||_s &\leq
 |||\big[\Omega^{-1}\BB(\bb)^{(n)}\big]^{l}|||_s+|||\big[\Omega^{-1}\FF(\vv^{l})\big]^{(n)}|||_s\\
 &\leq\tilde{\epsilon}^{\frac{1}{2}}|||\big[\bb^{(n+1)}\big]^{l}|||_s+h^{-1}\tilde{\epsilon}^{-\frac{1}{2}}|||\big[\FF(\vv^{l})\big]^{(n)}|||_s.
\end{aligned}
\end{equation*}
With the second  equality  of \eqref{variables iteration}, it is
 deduced that
\begin{equation*}
\begin{aligned}
|||\big[\mathbf{a}^{(n+1)}\big]^{l+1}|||_s &\leq
 |||\big[\frac{\sinc(\frac{1}{2}h\Omega)}{h\tilde{\epsilon}}\AA(\aa)^{(n)}\big]^{l}|||_s+ |||\big[\frac{\sinc(\frac{1}{2}h\Omega)}{h\tilde{\epsilon}}\FF(\vv^{l})\big]^{(n)}|||_s\\
&\leq
h^{2}\tilde{\epsilon}^{2}|||\big[\aa^{(n+2)}\big]^{l}|||_s+h^{-1}|||\big[\FF(\vv^{l})\big]^{(n)}|||_s,\ \ \ l=0,1,\ldots,\\
|||\big[\mathbf{a}\big]^{l+1}|||_s &\leq
 |||\big[\aa(0)\big]^{l+1}|||_s+\sup_{\tilde{\tau}}|||\big[\dot{\aa}(\tilde{\tau})\big]^{l+1}|||_s.
\end{aligned}
\end{equation*}
By Proposition \ref{pro: lem1} and the same analysis as that
described in Section 3.6 of \cite{18}, the result \eqref{bounds f12}
can be proved.

$\bullet$ Proof of \eqref{O}. 

For
$\tilde{u}=[\tilde{u}]^L=\sum\limits_{k}[z^{k}]^L\mathrm{e}^{-\ii(k
\cdot\omega) t}$ with $L$ the number of ending iterate
and by the same calculations as those presented in \cite{18}, one
has
\begin{equation*}
\begin{aligned}
\norm{\tilde{u}}^2_s
&=\sum\limits_{j}\abs{\omega_j}^s\abs{\sum\limits_{k}[z_j^{k}]^L\mathrm{e}^{-\ii(k
\cdot\omega) t}}^2 \leq \tilde{\epsilon}^2
\sum\limits_{j}\abs{\omega_j}^s\Big(\sum\limits_{k}\abs{[c_j^{k}]^L}\Big)^2=\tilde{\epsilon}^2
|||[\cc]^{L}|||_s^2,
\end{aligned}
\end{equation*}
which proves \eqref{O}.

$\bullet$ Proof of \eqref{OO}.

We now turn to the size of the variables $\hat{\aa}$ and $\hat{\bb}$
in the second rescaling, and we have
$|||\hat{\aa}|||_{\frac{d+1}{2}}=|||\aa|||_{s},
|||\hat{\bb}|||_{\frac{d+1}{2}}=|||\bb|||_{s}.$ Then from this fact
and the above analysis, \eqref{OO} is obtained.

\subsection*{Appendix III. Proof of Proposition
\ref{pro: defect}}

$\bullet$ Proof of the first result. 

In order to estimate $\mathbf{f}$, the nonresonance condition
\eqref{non-resonance cond} and Proposition \ref{pro: lem1} are
considered.  Under these conditions and for $l = 0,\ldots,L$, one
has
\begin{equation*}
\begin{aligned}
|||[\mathbf{f}]^{l}|||^2_s=&\sum\limits_{j}\abs{\omega_j}^s\Big(\sum\limits_{k:
(j, k)\in\mathcal{R}_{\tilde{\epsilon},M,h}}\abs{[f_j^{k}]^{l}}\Big)^2\\
=&\sum\limits_{j}\abs{\omega_j}^{\frac{d+1}{2}}\Big(\sum\limits_{k:
(j,
k)\in\mathcal{R}_{\tilde{\epsilon},M,h}}\frac{\abs{\omega_j}^{\frac{2s-d-1}{4}}\tilde{\epsilon}^{\max([[k]],2)}}{\abs{\omega^{\frac{2s-d-1}{4}\abs{k}}}}
\abs{[\hat{\FF}(\hat{\mathbf{u}})_j^{k}]^{l}}\Big)^2\\
\leq &
|||[[\hat{\FF}(\hat{\mathbf{u}})]^n|||_{\frac{d+1}{2}}^2\sup_{(j,
k)\in\mathcal{R}_{\tilde{\epsilon},M,h}}\Big(\frac{\abs{\omega_j}^{\frac{2s-d-1}{4}}}{\abs{\omega^{\frac{2s-d-1}{4}\abs{k}}}}
\tilde{\epsilon}^{[[k]]} \Big)^2 \\
\leq & C (h\tilde{\epsilon})^2\tilde{\epsilon}^{2N+4}=Ch^2(\tilde{\epsilon}^{N+3})^2.
\end{aligned}
\end{equation*}

$\bullet$ Proof of the second result. 

 From $\norm{k}>K$, it follows that $[[k]]\geq(K+2)/2=N+2$.  With the same arguments as those given in the proof of Proposition \ref{pro: lem1}, we
  obtain \begin{equation*}
\begin{aligned}
&\norm{\sum\limits_{ \norm{k}>K}[g^{k}]^{l}}_s\\
=&
\norm{\sum\limits_{\norm{k}>K}\tilde{\epsilon}^{[[k]]}\Big[\tilde{\epsilon}^{-[[k]]}\sum\limits_{k^1+k^2-k^3=k
}\int_{0}^{1}\tilde{\epsilon}^{[[k^1]] + [[k^2 ]] +
[[k^3]]}w^{k^1}_{j(k^1)}w^{k^2}_{j(k^2)}\overline{w^{k^3}_{j(k^3)}}d\sigma\Big]}_s\\
\leq & C\tilde{\epsilon}^{\frac{K+2}{2}}h\tilde{\epsilon}=C\tilde{\epsilon}^{N+3}h.
\end{aligned}
\end{equation*}

$\bullet$ Proof of the third and fourth results. 

The  off-diagonal part $\mathbf{e}$   and the diagonal part
$\dot{\mathbf{h}}$ of the defect can be expressed respectively by
\begin{equation*}
\begin{aligned}
&[e_j^{k}]
^{l}=\tilde{\epsilon}^{[[k]]}\big([(\WW\bb)_j^{k}]^{l}-[(\WW\bb)_j^{k}]^{l+1}\big),\
\ [h_j^{k}]
^{l}=\tilde{\epsilon}^{3/2}\big([(\WW\aa)_j^{k}]^{l}-[(\WW\aa)_j^{k}]^{l+1}\big).
\end{aligned}
\end{equation*}
Using a Lipschitz estimate given in Proposition \ref{pro: lem1}
for the nonlinearity and by an analysis of the iteration used as in
Sect. 5.7 of \cite{19}, it is obtained that
\begin{equation*}
\begin{aligned}
&|||[ \mathbf{h}(\tilde{\tau})]^{l}|||_s\leq C \tilde{\epsilon}{\frac{p+4}{2}}h,\
|||[ \mathbf{h}^{(n)}(\tilde{\tau})]^{l}|||_s\leq C
\tilde{\epsilon}{\frac{p+4}{2}}h,\ \
l\geq1,\\
&|||[ \mathbf{e}^{(n)}(\tilde{\tau})]^{l}|||_s\leq C
\tilde{\epsilon}{\frac{p+4}{2}}h,\ \ l\geq0
\end{aligned}
\end{equation*}
for $0\leq\tilde{\tau}\leq 1.$

$\bullet$ Proof of the last result.

The last result can follows from  the same arguments as
 the description of (29) in \cite{19}.

 \section*{Acknowledgements}
The authors are grateful to  Christian Lubich for his helpful
comments and discussions on the topic of modulated Fourier
expansions. We also thank Xinyuan Wu and Changying Liu for their
valuable comments. The research is supported in part by the NNSF of
China (11871393).  This work was done in part at UNIVERSITAT T\"{U}BINGEN when the first author worked there as a postdoctoral researcher (2017-2019, supported by the Alexander von Humboldt Foundation).

\end{document}